\documentclass[11pt]{article}
\newcommand{\E}{\mathbb{E}}
\newcommand{\V}{\mathbb{V}}
\newcommand{\R}{\mathbb{R}}
\newcommand{\C}{\mathbb{C}}

\newcommand{\Prob}{\mathbb{P}}
\newcommand{\vnorm}[1]{\left|\left|#1\right|\right|}

\newcommand{\partialx}[2]{\frac{\partial #1}{\partial #2}}

\newcommand{\blinded}[1]{ #1 }

\usepackage{lineno}
\usepackage{float}
\usepackage{enumerate}
\usepackage{authblk}
\usepackage[authoryear,sort&compress,longnamesfirst]{natbib}
\usepackage[margin=1.25in, top=1in, bottom=1in, includefoot]{geometry}
\usepackage{amssymb, amsmath}
\usepackage{graphicx}
\usepackage{amsthm}
\usepackage{verbatim}
\usepackage{url}
\usepackage{paralist}
\RequirePackage[OT1]{fontenc}

\usepackage{subfig}
\RequirePackage[colorlinks,citecolor=blue,urlcolor=blue]{hyperref}
\usepackage{algorithm}
\usepackage[noend]{algorithmic}
\usepackage{amsmath,amscd}
\usepackage{tikz}
\usetikzlibrary{shapes,arrows}
\usepackage{setspace}
\usepackage{rotating}

\newenvironment{edit}[2][Note]{}

\newtheorem{theorem}{Theorem}[section]

\newtheorem{corollary}[theorem]{Corollary}

\newtheorem{definition}[theorem]{Definition}

\newtheorem{lemma}[theorem]{Lemma}

\newtheorem{properties}[theorem]{Properties}
\newtheorem{remark}[theorem]{Remark}

\title{The Lambert Way to Gaussianize heavy tailed data with the inverse of Tukey's h transformation as a special case}

\author{Georg M.\ Goerg}
\affil{Department of Statistics, Carnegie Mellon University \\
Pittsburgh, PA 15213, USA\\
\href{mailto:gmg@stat.cmu.edu}{gmg@stat.cmu.edu}, \url{www.stat.cmu.edu/\textasciitilde gmg}
}

\date{\today}

\begin{document}
\pagenumbering{roman}

\maketitle

\linenumbers

\graphicspath{{images/}} 
\DeclareGraphicsExtensions{.png}

I present a parametric, bijective transformation to generate heavy tail versions $Y$ of arbitrary RVs $X \sim F_X$. 
The tail behavior of the so-called \emph{heavy tail Lambert W $\times$ $F_X$} RV $Y$ depends on a tail parameter $\delta \geq 0 $: for $\delta = 0$, $Y \equiv X$, for $\delta > 0$ $Y$ has heavier tails than $X$. For $X$ being Gaussian, this meta-family of heavy-tailed distributions reduces to Tukey's $h$ distribution. Lambert's $W$ function provides an explicit inverse transformation, which can be estimated by maximum likelihood. This inverse can remove heavy tails from data, and also provide analytical expressions for the cumulative distribution (cdf) and probability density function (pdf). As a special case, these yield explicit formulas for Tukey's $h$ pdf and cdf - to the author's knowledge for the first time in the literature.  Simulations and applications to S\&P 500 log-returns and solar flares data demonstrate the usefulness of the introduced methodology. 

The \textsc{R} package \href{cran.r-project.org/web/packages/LambertW}{\texttt{LambertW}} implementing the presented methodology is publicly available at \href{cran.r-project.org/}{\texttt{CRAN}}.

\newpage
\singlespacing
\tableofcontents
\newpage
\setstretch{1.5}
\pagenumbering{arabic}
\section{Introduction}
\label{sec:introduction}

Statistical theory and practice are both tightly linked to Gaussianity. In theory, many methods require Gaussian data or noise: 
\begin{inparaenum}[i)]
\item regression often assumes Gaussian errors;
\item pattern recognition for images often model noise as a Gaussian random field \citep{AchimTsakalidesBezerianos03_Imagedenoising};
\item many time series models are based on Gaussian white noise \citep{BrockwellDavis98, Engle82,GrangerJoyeux01}.
\end{inparaenum}

In all these cases, a model $\mathcal{M}_{\mathcal{N}}$, parameter estimates and their standard errors, and other properties, are then studied -- all based on the ideal(istic) assumption of Gaussianity.\\

In practice, however, data/noise often exhibits asymmetry and heavy tails; for example wind speed data \citep{Field04}, human dynamics \newline

\citep{Vazquezetal06_Modelinburts}, or Internet traffic data \citep{Gidlund09_3Gtraffic} -- just to a name few. 
Particularly notable examples are financial data \citep{Cont01_Empiricalproperties, School_RobustEstimationskewandkurt} and speech signals \citep{AysalBarner06}, which almost exclusively exhibit heavy tails. Thus a model $\mathcal{M}_{\mathcal{N}}$ developed for the Gaussian case does not necessarily provide accurate inference anymore.

One way to overcome this shortcoming is to replace $\mathcal{M}_{\mathcal{N}}$ with a new model $\mathcal{M}_{G}$, where $G$ is a heavy tail distribution: 
\begin{inparaenum}[i)]
\item regression with Cauchy errors \citep{Smith73_OLSCauchy};
\item image denoising for $\alpha$-stable noise \citep{AchimTsakalidesBezerianos03_Imagedenoising};
\item forecasting long memory processes with heavy tail innovations \citep{Ilow00_ARFIMAheavytails, PalmaZevallos11_NonGaussianLongMemory}, or ARMA modeling of electricity loads with hyperbolic noise \citep{NowickaZagrajek02}. 
\end{inparaenum}

While such fundamental approaches are attractive from a theoretical perspective, they can become unsatisfactory from a practical viewpoint. Many successful statistical models assume Gaussianity, their theory is very well understood, and many algorithms are implemented for the simple -- and often much faster -- Gaussian case. Thus developing models based on an entirely unrelated distribution $G$ is like throwing out the (Gaussian) baby with the bathwater.\\

It would be very useful to transform a Gaussian RV $X$ to a heavy-tailed RV $Y$ and vice versa, and thus rely on knowledge - and software - for the well-understood Gaussian case, while still capturing heavy tails in the data. Optimally such a transformation should: 
\begin{inparaenum}[a)]
\item be bijective;
\item include Normality as a special case for hypothesis testing; and 
\item be parametric so the optimal transformation can be estimated efficiently.
\end{inparaenum}

\begin{figure}[!t]
\centering
\tikzstyle{block} = [draw, fill=blue!20, rectangle, 
    minimum height=3em, minimum width=6em]
\tikzstyle{RV} = [draw, fill=white, circle, 
    minimum height=0.8em, minimum width=0.8em, node distance = 1.2cm]
\tikzstyle{transformation} = [draw, fill=gray!50, rectangle, 
    minimum height=2em, minimum width=3em, node distance = 3.5cm]    
\tikzstyle{cdf} = [draw, fill=gray!30, rectangle, node distance = 3.5cm,
    minimum height=3em, minimum width=5em] 
    \tikzstyle{COOR} = [coordinate]
    \tikzstyle{inference} = [draw, fill=gray!30, rectangle, node distance = 2.7cm,
    minimum height=3em, minimum width=5em, rounded corners]

\tikzstyle{description} = [rectangle, draw=white, thick, fill=white!50, 
 	text width=8em, minimum height=1em, minimum width = 5em, text centered, node distance=-1.35cm, inner sep=2pt, rounded corners]
\tikzstyle{descriptionB} = [rectangle, draw=white, thick, fill=white!50, 
 	text width=8em, minimum height=1em, minimum width = 5em, text centered, node distance=-1cm, inner sep=2pt, rounded corners]
 	\tikzstyle{empty_h} = [rectangle, draw=white, thick, fill=white!50, 
 	text width=2em, minimum height=1em, minimum width = 1em, text centered, node distance=4cm, inner sep=2pt, rounded corners]
 	\tikzstyle{empty_v} = [rectangle, draw=white, thick, fill=white!50, 
 	text width=2em, minimum height=1em, minimum width = 1em, text centered, node distance=1.5cm, inner sep=2pt, rounded corners]
\tikzstyle{pinstyle} = [pin edge={to-,thin,black}]


\begin{tikzpicture}[auto, node distance=2cm,>=latex']
    \node [RV] (input) {$X$ and $\mathbf{x}$};
    \node [empty_v, below of = input] (empty) {\begin{tabular}{c}  \end{tabular}};
    \node [transformation, right of=input] (system_S) {$H_{\tau}(X)$};
    \node [empty_h, right of = system_S, node distance = 3.5cm] (empty2) {\begin{tabular}{c}  \end{tabular}};
    \node [RV, below of=empty2] (output) {$Y$ and $\mathbf{y}$};
    \node [RV, below of = empty] (backtransformed_X)  {$X_{\tau}$ and $\mathbf{x}_{\tau}$};    
    \node [transformation, right of = backtransformed_X] (inverse_system_S_1) {$W_{\tau}(Y) = H^{-1}_{\tau}(Y)$};

       
    \node [cdf, name=Fx, below of = empty] {$F_X$ and $F_{X_{\tau}}$};      
    \node [cdf, name=LambertW_Fx, right of = Fx, node distance = 7cm] {\begin{tabular}{c}Lambert W $\times$ $F_X$ \\ (Tukey's $h$ if $F_X = \mathcal{N}$) \end{tabular}};  

    \node [inference, name=infer_Y, below of = LambertW_Fx] {\begin{tabular}{c}inference \\ on $Y$\end{tabular}};  
    \node [inference, name=infer_Xhat, below of = Fx] {\begin{tabular}{c}inference \\ on $X_{\tau}$ \end{tabular}};       
    \node [transformation, right of=infer_Xhat] (H_delta_u) {$H_{\tau}(\cdot)$};


	\node [description, below of = input] (latent) {\begin{tabular}{c} latent (Gaussian) \end{tabular}};
	\node [description, right of = latent, node distance = 7cm] (observed) {\begin{tabular}{c} observed (heavy tails) \end{tabular}};
	\node [description, below of = Fx, node distance = -0.91cm] (recoverd) {recovered};
\node [descriptionB, below of = inverse_system_S_1] (Gaussianize) {\begin{tabular}{c} ``Gaussianize'' \\ heavy tails \end{tabular}};
    \path[->, solid] (input) edge node {} (system_S);
    \path[->, solid] (system_S) edge node {} (output);
    \path[->, solid] (output) edge node {} (inverse_system_S_1);
    \path[->, solid] (inverse_system_S_1) edge node {} (backtransformed_X);

    \draw[<->, solid] (Fx) edge node {} (LambertW_Fx);
    \draw[<->, dashed] (input) edge node {} (backtransformed_X);
    
    \draw[<->, dashed] (input) edge[bend right= 70] node[near end, left] {} (Fx);
    \draw[<->, dashed] (backtransformed_X) edge[bend right= 70] node[near end, left] {} (Fx);
    \draw[<->, dashed] (output) edge node {} (LambertW_Fx);

	
	 \draw[->, dashed] (Fx) edge node[left] {\begin{tabular}{c}methods/models \\ assuming  \\ Gaussianity\end{tabular}} (infer_Xhat);
	 \draw[->, dashed] (LambertW_Fx) edge node[left] {exact} (infer_Y);
	 \draw[->, dashed] (LambertW_Fx) edge node[right] {\begin{tabular}{c}methods/models \\ for heavy-tailed \\ distributions  (if $\exists$)\end{tabular}} (infer_Y);
     \draw[->,solid] (infer_Xhat) edge node {} (H_delta_u);
     \draw[->,solid] (H_delta_u) edge node {approx.} (infer_Y);   
     
\end{tikzpicture}
\caption{\label{fig:LambertW_heavy_flowchart} Schematic view of the heavy tail Lambert W $\times$ $F_X$ framework. (left) Latent input $X \sim F_X$: $H_{\tau}(X)$ from \eqref{eq:normalized_LambertW_Y} transforms (solid arrows) $X$ to $Y \sim $ Lambert W $\times$ $F_X$ and generates heavy tails. (right) Observed heavy tail world $Y$ and $\mathbf{y}$: (1) use $W_{\tau}(\cdot)$ to back-transform $\mathbf{y}$ to latent ``Normal'' $\mathbf{x}_{\tau}$, (2) use model $\mathcal{M}_{\mathcal{N}}$ of your choice (regression, time series models, hypothesis testing, etc.) for inference on $\mathbf{x}_{\tau}$, and (3) convert results back to the original ``heavy-tailed world'' of $\mathbf{y}$.}
\end{figure}
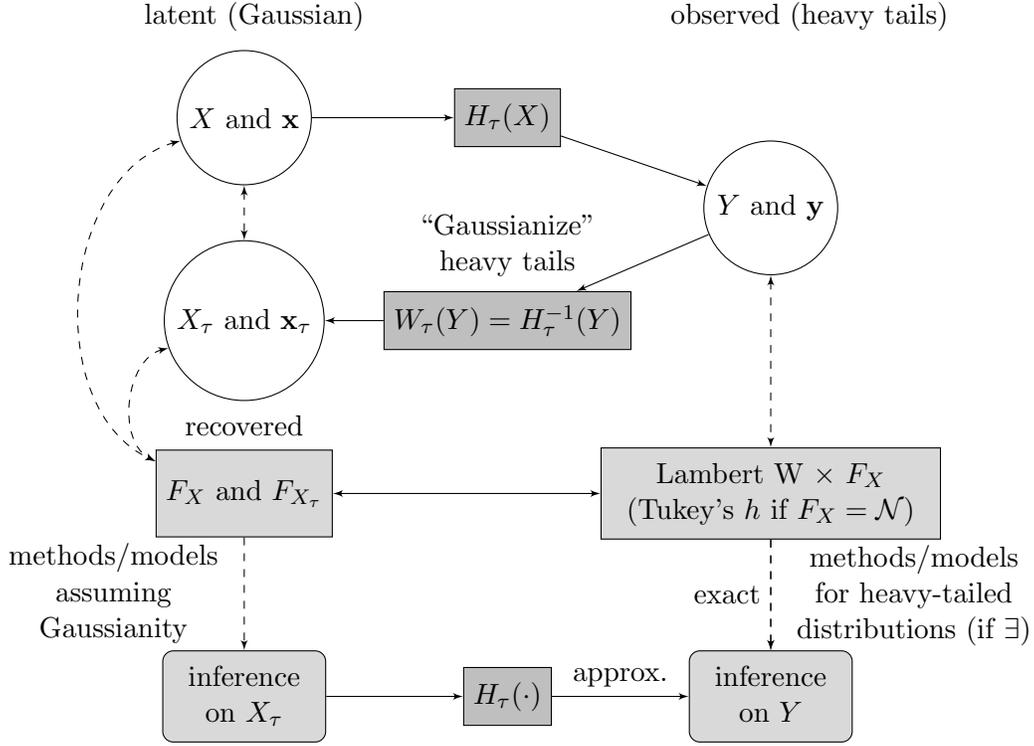

Figure \ref{fig:LambertW_heavy_flowchart} illustrates this pragmatic approach: researchers can make their observations $\mathbf{y}$ as Gaussian as possible ($\mathbf{x}_{\tau}$) before making inference based on their favorite Gaussian model $\mathcal{M}_{\mathcal{N}}$. This avoids the development of - or the data analysts waiting for - a whole new theory of $\mathcal{M}_{G}$ and new implementations based on a particular heavy-tailed distribution $G$, while still improving statistical inference on heavy-tailed data $\mathbf{y}$. For example, consider $\mathbf{y} = (y_1, \ldots, y_{500})$ from a standard Cauchy distribution $\mathcal{C}(0,1)$ in Fig.\ \ref{fig:Cauchy_sample}: modeling heavy tails by a transformation makes it even possible to Gaussianize this Cauchy sample (Fig.\ \ref{fig:Gaussianized_Cauchy_sample}). This ``nice'' data $\mathbf{x}_{\tau}$ can then be subsequently analyzed with common techniques.  For example, the location can now be estimated using the sample average (Fig.\ \ref{fig:Cauchy_Gaussian_avg}). For details see Section \ref{sec:sample_mean_Cauchy}.\\

\citet{Wassermanetal09_Nonparanormal} use a semi-parametric approach, where $Y$ has a \emph{nonparanormal} distribution if $f(Y) \sim \mathcal{N}(\mu, \sigma^2)$ where $f(\cdot)$ is an increasing smooth function; they estimate $f(\cdot)$ using non-parametric methods. This leads to a greater flexibility in the distribution of $Y$, but it suffers from two drawbacks:
\begin{inparaenum}[i)]
\item \label{item:slow_convergence} non-parametric methods have slower convergence rates and thus need large samples, and
\item \label{item:identifiability} for identifiability of $f(\cdot)$, $\E f(Y) \equiv \E Y$ and $\V f(Y) \equiv \V Y$ must hold.
\end{inparaenum} 
While \ref{item:slow_convergence}) is inherent to non-parametric methods, point \ref{item:identifiability}) requires $Y$ to have finite mean and variance, which is especially limiting for heavy-tailed data where this condition is often not met. Thus here we use parametric transformations which do not rely on restrictive identifiability conditions and also work well for small sample sizes.\\

The main contributions of this work are three-fold:
\begin{inparaenum}[a)]
\item  following \citet{GMGLambertW_Skewed} I introduce a meta-family of heavy tail Lambert W $\times$ $F_X$ distributions with Tukey's $h$ \citep{Hoaglin06} as a special case;
\item  I present a bijective transformation to ``Gaussianize'' heavy-tailed data (Section \ref{sec:heavy_tailed}); and
\item  I also provide simple expressions for the cumulative distribution function (cdf) $G_Y(y)$ and probability density function (pdf) $g_Y(y)$ - also for Tukey's $h$ --, which can be easily implemented in statistics software (Section \ref{sec:dist_dens}).
\end{inparaenum}

To the author's knowledge analytic expressions for Tukey's $h$ cdf and pdf are presented here (Section \ref{sec:Gaussian}) for the first time in the literature. Section \ref{sec:estimation} introduces a methods of moments estimator and studies the maximum likelihood estimator (MLE). Section \ref{sec:simulations} shows their finite sample properties.

As has been shown in many case studies, Tukey's $h$ distribution (heavy tail Lambert W $\times$ Gaussian) is useful to model data with unimodal, heavy-tailed densities. Section \ref{sec:applications} not only confirms this finding for S\&P 500 log-returns, but also demonstrates the benefits of removing heavy tails for exploratory data analysis: Gaussianizing $\gamma$-ray intensity data reveals a bimodal density, which even non-parametric estimators fail to detect if heavy tails are not removed. Finally, we discuss the new methodology and future work in Section \ref{sec:discussion_outlook}.  All proofs are given in the Supplementary Material, Appendix \ref{sec:proofs}.

Computations, figures, and simulations were done in \textsc{R} \citep{R10}. \blinded{The R package \href{http://cran.r-project.org/web/packages/LambertW/index.html}{\texttt{LambertW}} is publicly available on CRAN.}

\begin{figure}[!t]
\centering
\subfloat[Random sample $\mathbf{y} \sim \mathcal{C}(0,1)$]
{\includegraphics[width=.45\textwidth]{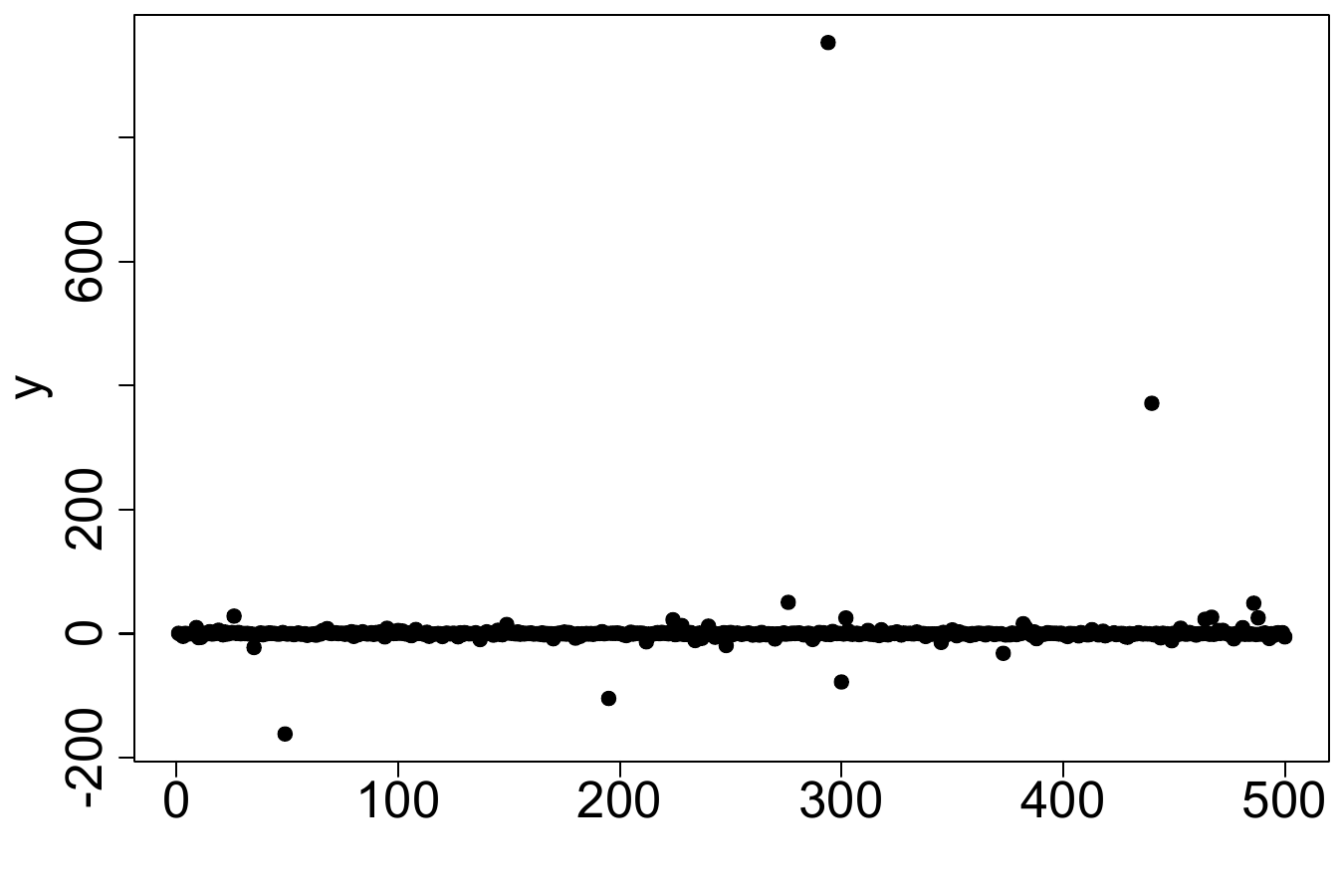}\label{fig:Cauchy_sample}}
\hspace{0.01\linewidth}
\subfloat[Box-Cox transformed $\mathbf{x}_{\widehat{\lambda}_{MLE}}$ with $\widehat{\lambda} = 0.37$]
{\includegraphics[width=.45\textwidth]{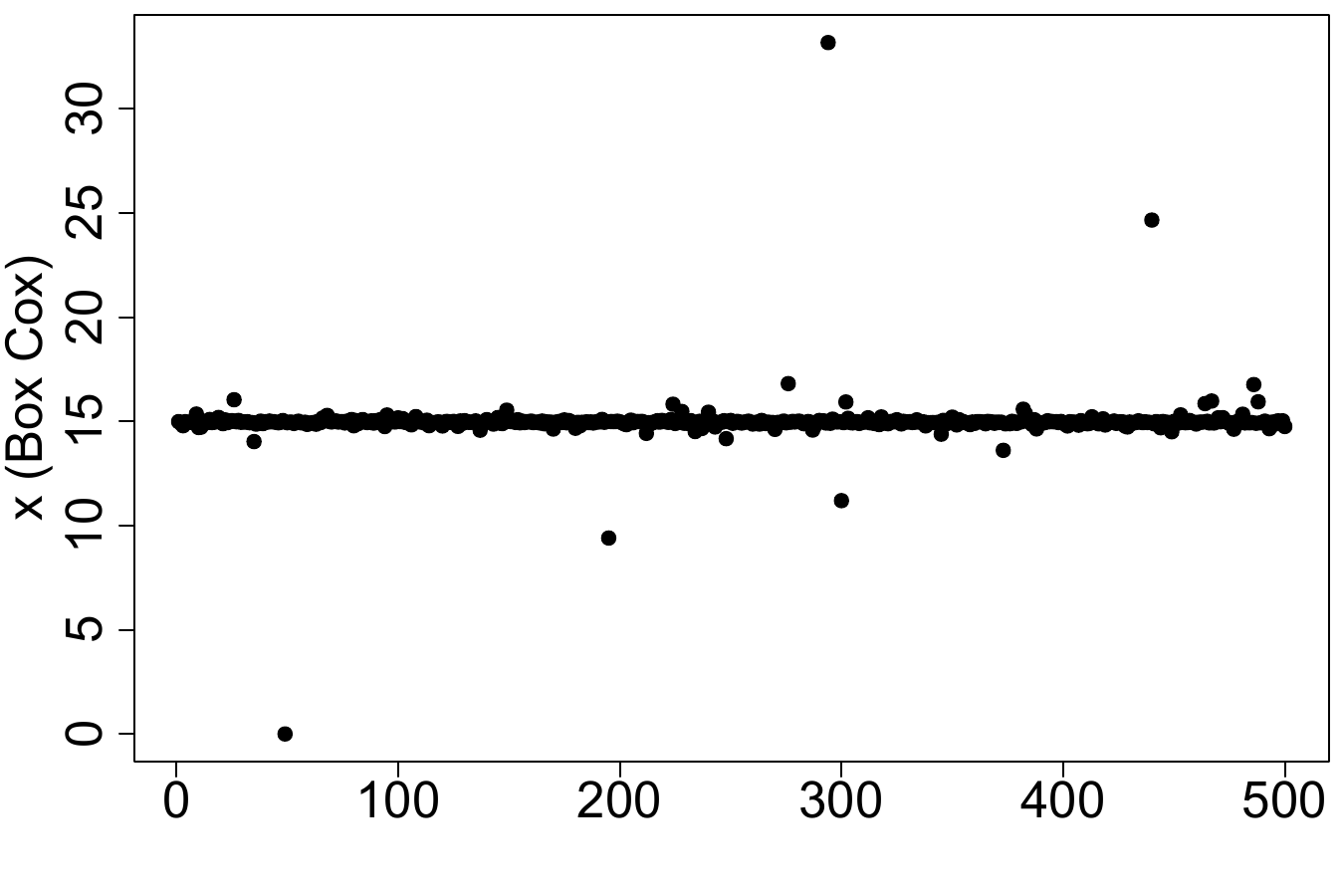}\label{fig:BoxCox_Cauchy_sample}}

\subfloat[Gaussianized $\mathbf{x}_{\widehat{\tau}_{MLE}}$ with $\widehat{\tau} = (0.03, 1.05, 0.86)$]{\includegraphics[width=.45\textwidth]{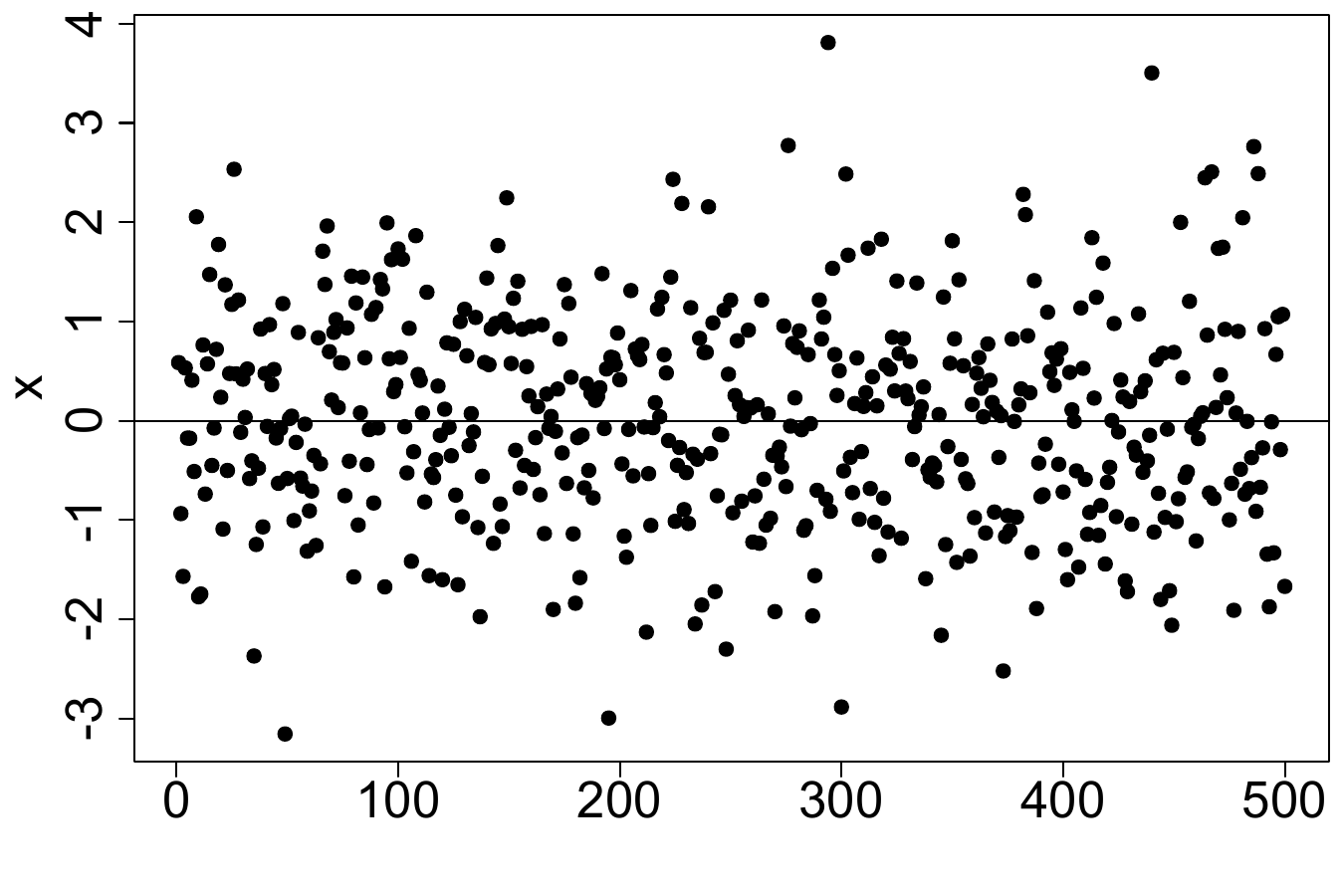}\label{fig:Gaussianized_Cauchy_sample}}
\hspace{0.01\linewidth}
\subfloat[Cumulative sample average]
{\includegraphics[width=.45\textwidth]{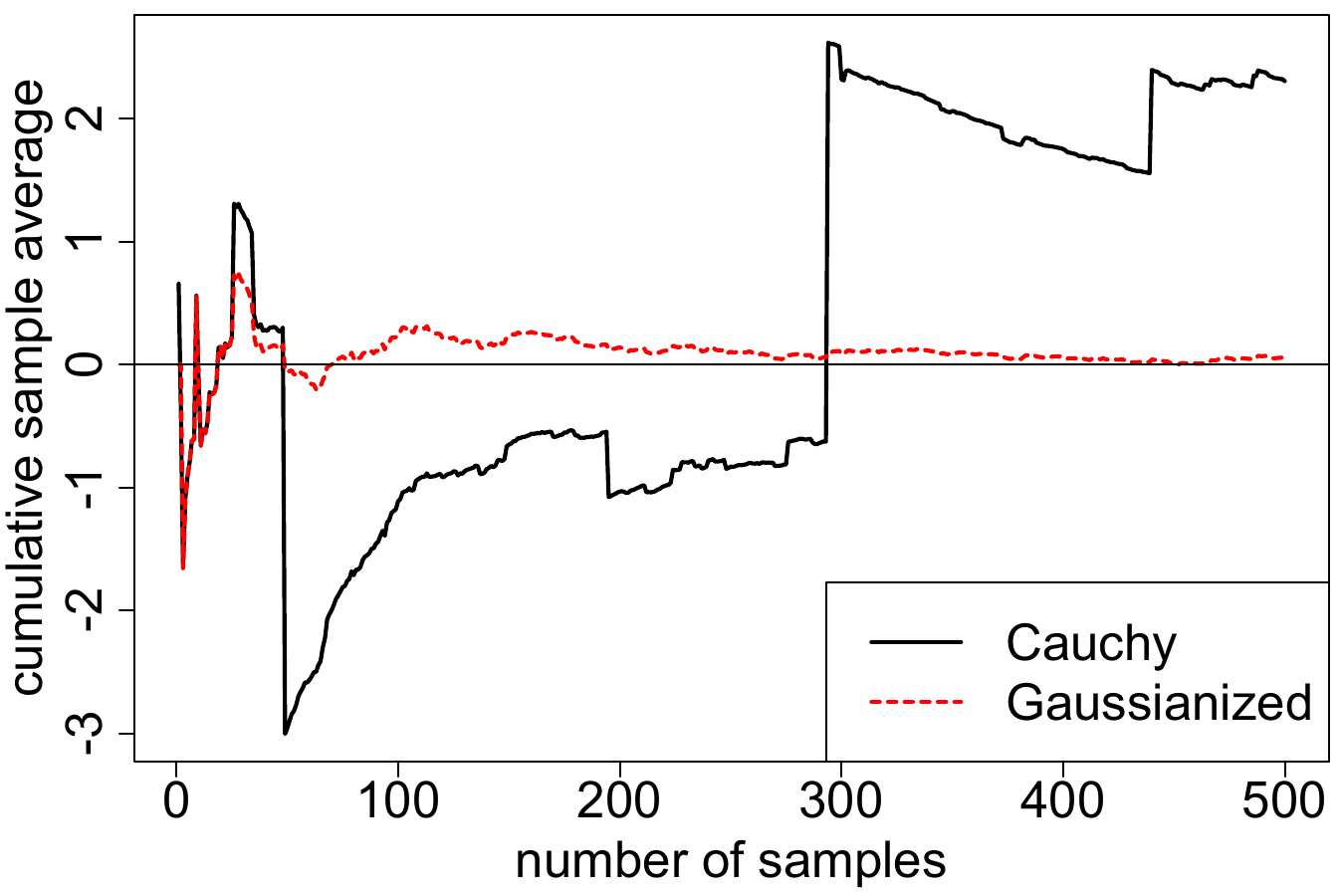}\label{fig:Cauchy_Gaussian_avg}}
\caption{Gaussianizing a standard Cauchy sample. For (d) $\tau^{(n)}$ was estimated for each fixed $n = 5, \ldots, 500$, before Gaussianizing $(y_1, \ldots, y_n)$. }
\label{fig:Gaussianizing_Cauchy}
\end{figure}

\subsection{Multivariate Extensions}
While this work focuses on univariate case, multivariate extensions of the presented methods can be defined component-wise -- analogously to the multivariate version of Tukey's $h$ distribution \citep{FieldGenton06}. While this may not make the transformed RVs jointly Gaussian, it still provides a good starting point for more well-behaved multivariate modeling.

\subsection{Box-Cox Transformation to Remove Heavy Tails}
A popular method to deal with skewed, high variance data is the Box-Cox transformation
\begin{equation}
\label{eq:Box_cox}
\mathbf{y}_{\lambda} = \begin{cases}
\frac{\mathbf{y}^{\lambda} - 1}{\lambda} & \text{ if } \lambda > 0 \\
\log \mathbf{y} & \text{ if } \lambda = 0.
\end{cases} 
\end{equation}
The parameter $\lambda$ can be chosen by MLE. However one major limitation of \eqref{eq:Box_cox} is the non-negativity constraint on $\mathbf{y}$, which prohibits its use in many applications. To avoid this limitation it is common to shift the data, $\tilde{\mathbf{y}} = \mathbf{y} + | \min(\mathbf{y}) | \geq 0$. However, as Fig.\ \ref{fig:BoxCox_Cauchy_sample} shows applying the Box-Cox transformation to the Cauchy sample\footnote{We use $\tilde{\mathbf{y}} = \mathbf{y} + | \min(\mathbf{y}) | + 1$ and use \texttt{boxcox} from the \texttt{MASS} R package; $\widehat{\lambda} = 0.37$.} completely fails. Furthermore, this restricts $Y$ to a half-open interval $[c, \infty)$ and is not desirable if the underlying process can occur on the entire real line, since it undermines statistical inference for yet unobserved data. See \citet{Sakia92_BoxCoxReview} for a more detailed discussion and the Box-Cox transformation in general.\\

Furthermore, the main purpose of the Box-Cox transformation is to stabilize variance \citep{Blaylocketal80_BoxCoxHeteroskedastic, Lawrance87_BoxCox, Tsiotas07_BoxCoxconditionalVariance} and remove right tail skewness \citep{GoncalvesMeddahi11}; a lower empirical kurtosis is merely a by-result of the variance stabilization. In contrast, the Lambert W framework is designed to model heavy-tailed RVs and remove heavy tails from data, and has no difficulties with negative values.

\section{Generating Heavy Tails Using Transformations}
\label{sec:heavy_tailed}
Random variables exhibit heavy tails if more mass than for a Gaussian RV lies at the outer end of the density support. A RV $Z$ has a tail index $a$ if its cdf satisfies $1 - F_Z(z) \sim L(z) z^{-a}$, where $L(z)$ is a slowly varying function at infinity, i.e.\ $\lim_{z \rightarrow \infty} \frac{L(tz)}{L(z)} = 1$ for all $t > 0$ \citep{Baek10_HeavyTail}.\footnote{There are various similar definitions of heavy/fat/long tails; for this work these differences are not essential.}  The heavy tail index $a$ is an important characteristic of $Z$; for example, only moments up to order $a$ exist.

\subsection{Tukey's $h$ Distribution}
A parametric transformation is the basis of Tukey's $h$ RVs \citep{Hoaglin06}
\begin{equation}
\label{eq:Tukey_h}
Z = U \exp \left( \frac{h}{2} U^2 \right), \quad h \geq 0,
\end{equation}
where $U$ is standard Normal RV and $h$ is the heavy tail parameter. The RV $Z$ has tail parameter $a = 1/h$ \citep{MorgenthalerTukey00_FittingQuantiles} and reduces to the Gaussian for $h=0$. \citet{MorgenthalerTukey00_FittingQuantiles} extend the $h$ distribution to the skewed, heavy-tailed family of $hh$ RVs
\begin{equation}
\label{eq:trafo_hh}
Z = \begin{cases}
U \exp \left( \frac{\delta_{\ell}}{2} U^2 \right), & \text{if } U \leq 0, \\
U \exp \left( \frac{\delta_r}{2} U^2 \right),  & \text{if } U > 0,
\end{cases}
\end{equation}
where again $U \sim \mathcal{N}(0,1)$. Here $\delta_{\ell} \geq 0$ and $\delta_r \geq 0$ shape the left and right tail of $Z$, respectively; thus transformation \eqref{eq:trafo_hh} can model skewed and heavy-tailed data - see Fig.\ \ref{fig:G_2delta}. For simplicity let $H_{\delta}(u) := u \exp \left( \frac{\delta}{2} u^2 \right)$.

\begin{figure}
\centering
\subfloat[ $\delta_{\ell} \delta_{r}$ transformation \eqref{eq:trafo_hh} ]{\includegraphics[width=.32\textwidth]{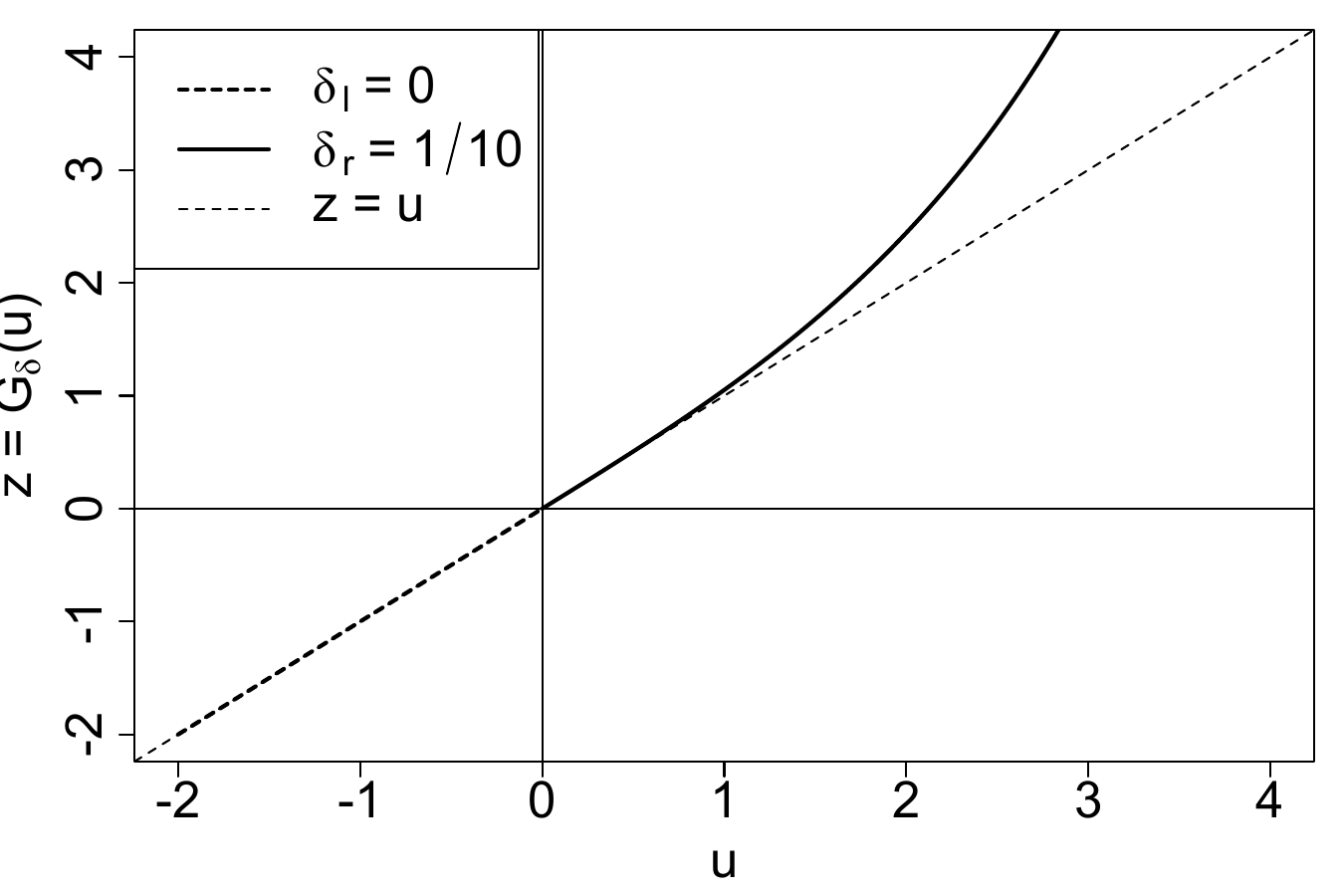}\label{fig:G_2delta}} 
\hspace{0.01\textwidth}
\subfloat[ Inverse $W_{\delta_{\ell}, \delta_{r}}(z)$ in \eqref{eq:backtrafo_hh} ]{\includegraphics[width=.32\textwidth]{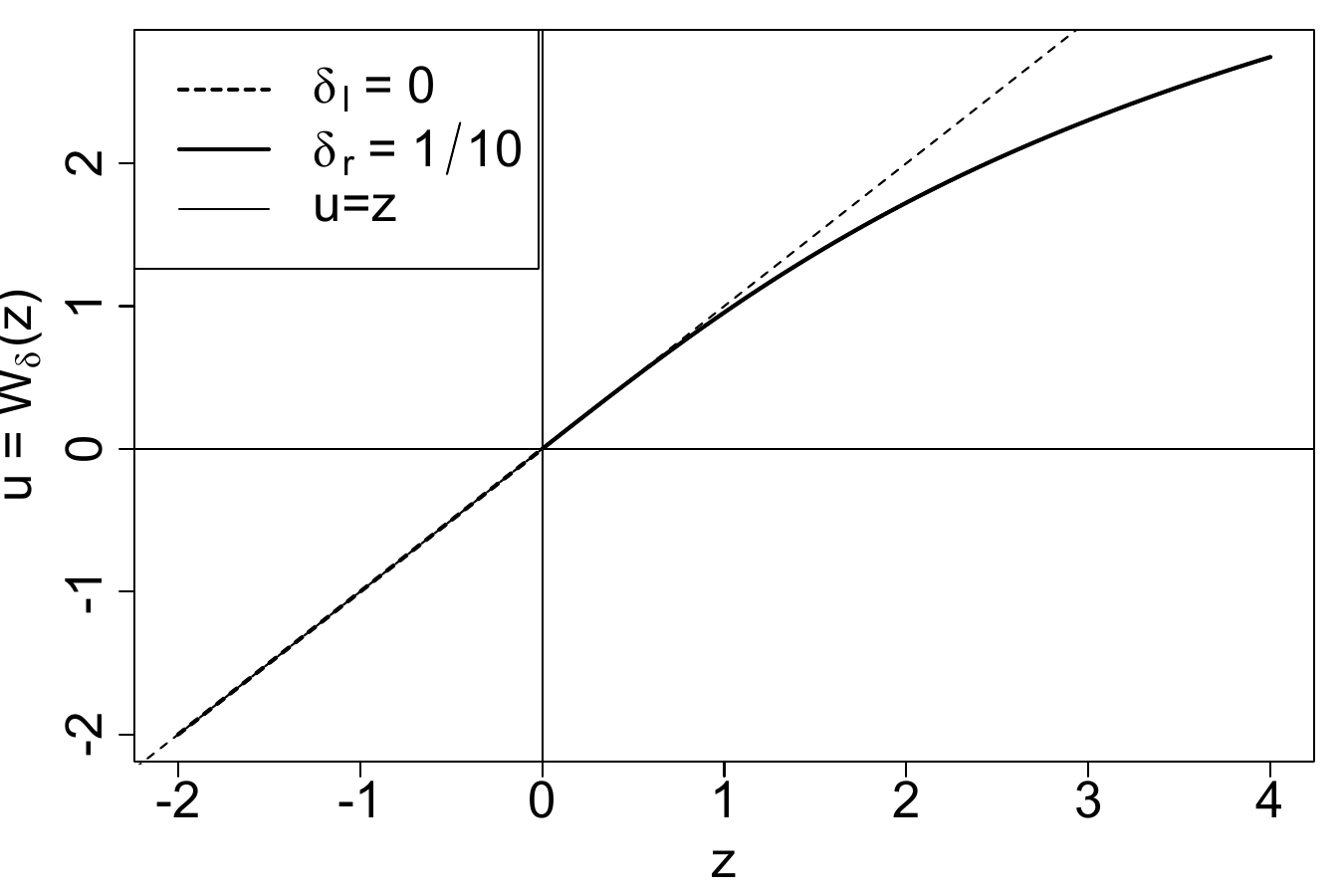}\label{fig:W_2delta}}
\hspace{0.01\textwidth}
\subfloat[ Inverse $W_{\delta}(z)$ \eqref{eq:W_delta} as a function of $\delta$ and $z$] {\includegraphics[width=.32\textwidth]{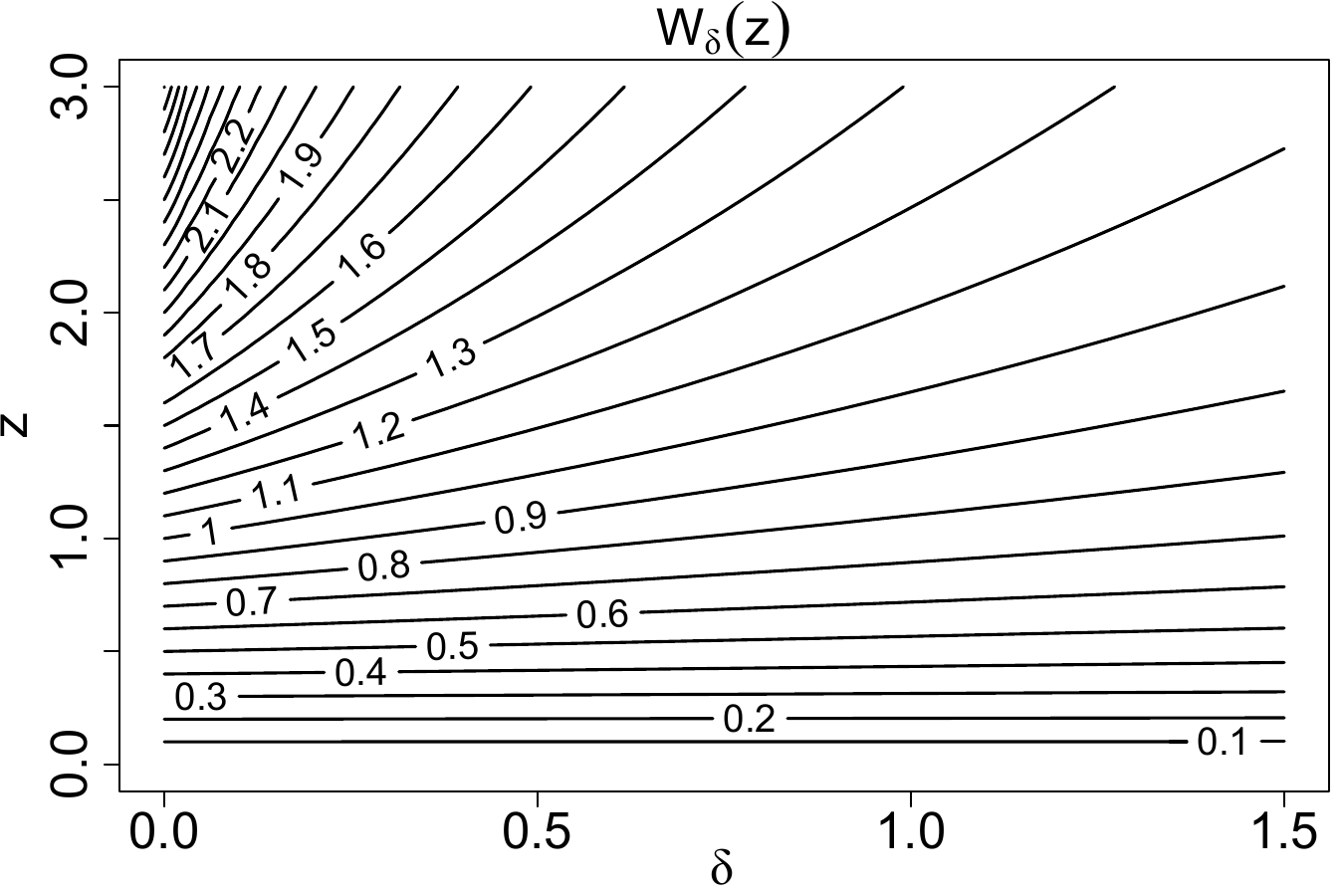}\label{fig:W_delta_z_contour}}
\caption{Transformation and inverse transformation for $\delta_{\ell} = 0$ and $\delta_r = 1/10$: identity on the left (same tail behavior) and a heavy-tailed transformation in the right tail of input $U$.}
\label{fig:G_W}
\end{figure}

Despite their great flexibility they are not popular in statistical practice, because the inverse of \eqref{eq:Tukey_h} or \eqref{eq:trafo_hh} has not been found. Consequently, no closed-form expressions for the cdf or pdf are available.  Although \citet{MorgenthalerTukey00_FittingQuantiles} express the pdf of \eqref{eq:Tukey_h} as ($h \equiv \delta$)
\begin{equation}
g_Z(z) = \frac{f_U \left( H_{\delta}^{-1}(z) \right)}{ H_{\delta}^{'}\left( H_{\delta}^{-1}(z) \right)},
\end{equation}
they fall short of explicitly specifying $H_{\delta}^{-1}(z)$. So far this inverse has been considered analytically intractable \citep{Field04}, or was only numerically approximated  \citep{Todd08_Parametric_pdf_cdf_tukey, Fischer10_financial_teletraffic}. Thus parameter inference relied on matching empirical and theoretical quantiles \citep{Field04, MorgenthalerTukey00_FittingQuantiles}, or by the method of moments \citep{Todd08_Parametric_pdf_cdf_tukey}. Only recently \citet{Todd08_Parametric_pdf_cdf_tukey} provided numerical approximations. Hence, a closed form, analytically tractable pdf that can be computed efficiently is essential for a wide-spread use of Tukey's $h$ (\& variants).\\

In this work I present this long sought explicit inverse, which is readily available in standard statistics software. For ease of notation and concision main results are shown for $\delta_{\ell} = \delta_r = \delta$; analogous results for $\delta_{\ell} \neq \delta_r$ will be stated without details.
\subsection{Heavy Tail Lambert W Random Variables}
Tukey's $h$ transformation \eqref{eq:Tukey_h} is strongly related to the approach taken by \citet{GMGLambertW_Skewed}
to introduce skewness in continuous RVs $X \sim F_X(x)$. In particular, if $Z \sim$ Tukey's $h$, then $Z^2 \sim $ skewed Lambert W $\times \chi^2_1$ with skew parameter $\gamma = h$.
 
Adapting the skew Lambert W $\times$ $F_X$ input/output idea\footnote{Most concepts and methods from the skew Lambert W $\times F_X$ case transfer one-to-one to the heavy tail Lambert W RVs presented here. Thus for the sake of concision I refer to \citet{GMGLambertW_Skewed} for details and background information on the Lambert W framework.} (see Fig.\ \ref{fig:LambertW_heavy_flowchart}), Tukey's $h$ RVs can be generalized to \emph{heavy-tailed Lambert W $\times$ $F_X$ RVs}.

\begin{definition} 
\label{def:noncentral_nonscaled}
Let $U$ be a continuous RV with cdf $F_U\left(u \mid \boldsymbol \beta \right)$, pdf $f_U\left(u \mid \boldsymbol \beta \right)$, and parameter vector $\boldsymbol \beta$. Then,
\begin{equation}
\label{eq:non_central_trafo_def}
Z = U \exp \left(\frac{\delta}{2} U^2 \right), \quad \delta \in \R,
\end{equation}
is a \emph{non-central, non-scaled heavy tail Lambert W $\times$ $F_X$} RV with parameter vector $\theta = \left( \boldsymbol \beta, \delta \right)$, where $\delta$ is the tail parameter.
\end{definition}
Tukey's $h$ distribution results for $U$ being a standard Gaussian $\mathcal{N}(0, 1)$.

\begin{definition} 
For a continuous location-scale family RV $X \sim F_X( x \mid \boldsymbol \beta)$ define a \emph{location-scale heavy-tailed Lambert W $\times$ $F_X$} RV
\begin{equation}
\label{eq:normalized_LambertW_Y}
Y = \left\{ U \exp \left(\frac{\delta}{2} U^2 \right) \right\} \sigma_x + \mu_x, \quad \delta \in \R,
\end{equation}
with parameter vector $\theta = (\boldsymbol \beta, \delta)$, where $U = (X - \mu_x)/\sigma_x$.
\end{definition}

The input is not necessarily Gaussian but can be any other location-scale continuous RV, e.g., from a uniform distribution: $X \sim U(a,b)$. 
\begin{definition} 
Let $X \sim F_X(x/s \mid \boldsymbol \beta)$ be a continuous scale-family RV, with standard deviation $\sigma_x$; let $U = X/\sigma_x$. Then,
\begin{equation}
\label{eq:non_central_trafo}
Y = X \exp \left(\frac{\delta}{2} U^2 \right), \quad \delta \in \R,
\end{equation}
is a \emph{scaled heavy-tailed Lambert W $\times$ $F_X$} RV with parameter $\theta = \left( \boldsymbol \beta, \delta \right)$.
\end{definition}

Let $\tau := \left( \mu_x(\boldsymbol \beta), \sigma_x(\boldsymbol \beta), \delta \right)$ define transformation \eqref{eq:normalized_LambertW_Y}.\footnote{For non-central, non-scale input set $\tau = (0,1,\delta)$; for scale-family input $\tau = (0, \sigma_x, \delta)$.}  If $X \in (-\infty, \infty)$, then for all $\delta \geq 0$ also the location-scale $Y \in (-\infty, \infty)$. For a scale family $X \in [0, \infty)$ also the scale Lambert W $\times$ $F_X$ RV $Y \in [0, \infty)$.\\

The shape parameter $\delta$ ($=$ Tukey's $h$) governs the tail behavior of $Y$: for $\delta > 0$ values further away from $\mu_x$ are increasingly emphasized, leading to a heavy-tailed version of $F_X(x)$; for $\delta = 0$, $Y \equiv X$; and for $\delta < 0$ values far away from the mean are mapped back again closer to $\mu_x$. Thus heavy tail Lambert W $\times$ $F_X$ RVs generalize $X \sim F_X(x)$ to heavy-tailed versions of itself, $Y \sim G_Y(y)$, with a reduction to $X$ for $\delta = 0$. 

The Lambert W formulation of heavy tail modeling is more general than Tukey's $h$ distribution as $X$ can have any distribution $F_X(x)$, not necessarily Gaussian (Fig.\ \ref{fig:LambertW_densities}). 

\begin{remark}[Only non-negative $\delta$]
Although $\delta < 0$ leads to interesting properties of $Y$, it yields a non-bijective transformation and thus to parameter-dependent support and non-unique input. Thus for the remainder of this work I tacitly assume $\delta \geq 0$, unless stated otherwise.
\end{remark}

\subsection{Inverse Transformation: ``Gaussianize'' Heavy-Tailed Data}
Transformation \eqref{eq:normalized_LambertW_Y} is bijective and its inverse can be obtained via Lambert's $W$ function, which is the inverse of $z = u \exp(u)$, i.e., that function which satisfies $W(z) \exp(W(z)) = z$. Lambert's $W$ has been studied extensively in mathematics, physics, and other areas of science \citep{Rosenlicht69, Corlessetal96, ValluriJeffreyCorless00}, and is implemented in the GNU Scientific Library (GSL) \citep{GSL11_Manual}. Only very recently it received attention in the statistics literature \citep{Jodra09_LambertWforGompertz_Makeham,RathieSilva11_ApplicationsLambertW,GMGLambertW_Skewed, Pakes11_LambertW}. It has many useful properties (see Appendix \ref{sec:Auxiliary} and \citet{Corlessetal96}), in particular $W(z)$ is bijective for $z \geq 0$.

\begin{lemma}
\label{lem:inverse_trafo_h}
The inverse transformation of \eqref{eq:normalized_LambertW_Y} is
\begin{equation}
\label{eq:backtrafo_normalized_X}
W_{\tau}(Y) := W_{\delta}\left(\frac{Y - \mu_x}{\sigma_x}\right) \sigma_x + \mu_x = U \sigma_x + \mu_x = X,
\end{equation}
where
\begin{equation}
\label{eq:W_delta}
W_{\delta}(z) := \operatorname{sgn}\left( z \right)  \left( \frac{W\left(\delta z^2\right)}{\delta} \right)^{1/2},
\end{equation}
and $\operatorname{sgn}(z)$ is the sign of $z$. $W_{\delta}(z)$ is bijective for all $\delta \geq 0$ and all $z \in \R$.
\end{lemma}

Lemma \ref{lem:inverse_trafo_h} gives for the first time an analytic, bijective inverse of Tukey's $h$ transformation: $H_{\delta}^{-1}(y)$ of \citet{MorgenthalerTukey00_FittingQuantiles} is now analytically available as \eqref{eq:backtrafo_normalized_X}. Bijectivity implies that for any data $\mathbf{y}$ and parameter $\tau$, the exact input $\mathbf{x}_{\tau} = W_{\tau}(\mathbf{y}) \sim F_X(x)$ can be obtained. 

In view of the importance and popularity of Gaussianity, we clearly want to back-transform heavy-tailed data to a Gaussian rather than yet another heavy-tailed distribution. Typically tail behavior of RVs are compared by their kurtosis $\gamma_2(X) = \E (X - \mu_x)^4 / \sigma_x^4$, which for a Gaussian RV equals $3$. Hence for the future when we ``normalize $\mathbf{y}$'' we can not only subtract the mean, and divide by the standard deviation, but also transform it to $\mathbf{x}_{\tau}$ with $\widehat{\gamma}_2\left(\mathbf{x}_{\tau} \right) = 3$ -- a ``Normalization'' in the true sense of the word (see Fig.\ \ref{fig:Gaussianized_Cauchy_sample}). 

This data-driven view of the Lambert W framework can also be useful for kernel density estimation (KDE), where multivariate data is often pre-scaled to unit-variance, so the same bandwidth can be used in each dimension \citep{Wasserman07_Nonparametric, Hwang94_nonparametricmultivariate}. Thus ``normalizing'' the Lambert Way might likely also improve KDE for heavy-tailed data \citep[see also][]{Markovich05_KDE_HeavyTail, Maiboroda04_KDE_HeavyTail}.

\begin{corollary}[Inverse transformation for Tukey's $hh$]
\label{cor:inverse_trafo_hh}
The inverse transformation of \eqref{eq:trafo_hh} is
\begin{equation}
\label{eq:backtrafo_hh}
W_{\delta_{\ell}, \delta_r}(z) =
 \begin{cases}
W_{\delta_{\ell}}(z), & \text{if } z \leq 0, \\
W_{\delta_r}(z),  & \text{if } z > 0.
\end{cases}
\end{equation}
\end{corollary}

Figure \ref{fig:W_2delta} shows $W_{\delta_{\ell}, \delta_r}(z)$ for $\delta_l = 0$ and $\delta_r = 1/10$. The transformation in Fig.\ \ref{fig:G_2delta} generates a right heavy tail version of $U$ (x-axis) by stretching only the positive axis (y-axis). By definition $W_{\delta_{\ell}, \delta_r}(z)$ removes the heavier right tail in $Z$ (positive y-axis).  Figure \ref{fig:W_delta_z_contour} shows how $W_{\delta}(z)$ operates for various degrees of heavy tails and $z \in [0,3]$. If $\delta$ is close to zero, then also $W_{\delta}(z) \approx z$; for larger $\delta$, the inverse maps $z$ to (much) smaller $u$.

\begin{remark}[Generalized transformation]
Transformation \eqref{eq:Tukey_h} can be generalized to 
\begin{equation}
\label{eq:delta_alpha_trafo}
Z = U \exp \left(\frac{\delta}{2 \alpha}  \left( U^2 \right)^{\alpha} \right), \quad \alpha > 0.
\end{equation}
The inner term $U^2$ guarantees bijectivity for all $\alpha > 0$. The inverse is
\begin{equation}
W_{\delta, \alpha}(z) := \operatorname{sgn}(z) \left( W\left( \frac{\delta z^{2 \alpha}}{\delta} \right) \right)^{\frac{1}{2 \alpha}}.
\end{equation}

For comparison with Tukey's $h$ I consider $\alpha = 1$ only. For $\alpha = 1/2$ transformation \eqref{eq:delta_alpha_trafo} is closely related to skewed Lambert W $\times$ $F_X$ distributions.
\end{remark}

\subsection{Distribution and Density}
\label{sec:dist_dens}
For ease of notation let 
\begin{align}
\label{eq:z_u_x_defs}
z = \frac{y-\mu_x}{\sigma_x}, \quad u = W_{\delta}(z), \text{ and }  x = W_{\tau}(y) = u \sigma_x + \mu_x.
\end{align}

\begin{theorem}[Distribution and Density of $Y$]
\label{theorem:cdf_Y}
The cdf and pdf of a location-scale heavy tail Lambert W $\times$ $F_X$ RV $Y$ equal 
\begin{eqnarray}
\label{eq:cdf_LambertW_Y}
G_Y\left(y \mid \boldsymbol \beta, \delta \right) = F_X\left(W_{\delta}(z) \sigma_x + \mu_x \mid \boldsymbol \beta \right)
\end{eqnarray}
and
\begin{align}
\label{eq:pdf_LambertW_Y}
g_Y\left(y \mid \boldsymbol \beta, \delta \right) 
& =  f_X\left( W_{\delta}\left(  \frac{y - \mu_x}{\sigma_x} \right) \sigma_x + \mu_x \mid \boldsymbol \beta \right) \cdot \frac{W_{\delta}\left( \frac{y - \mu_x}{\sigma_x}  \right)}{ \frac{y - \mu_x}{\sigma_x}  \left[ 1 + W\left( \delta \left( \frac{y - \mu_x}{\sigma_x} \right)^2 \right) \right]} .
\end{align}

Clearly, $G_Y\left(y \mid \boldsymbol \beta, \delta = 0 \right) = F_X \left( y \mid \boldsymbol \beta \right)$ and $g_Y\left(y \mid \boldsymbol \beta, \delta = 0 \right)  = f_X\left(y \mid \boldsymbol \beta \right) $, since $\lim_{\delta \rightarrow 0}W_{\delta}(z) = z$ and $\lim_{\delta \rightarrow 0} W(\delta z^2) = 0$ for all $z \in \R$.

For scale family or non-central, non-scale input set $\mu_x = 0$ or $\mu_x = 0, \sigma_x = 1$.
\end{theorem}

The explicit formula \eqref{eq:pdf_LambertW_Y} allows a fast computation and theoretical analysis of the likelihood, which is essential for -- either frequentist or Bayesian -- statistical inference. Detailed properties of \eqref{eq:pdf_LambertW_Y} are given in Section \ref{sec:MLE}.\\

Figure \ref{fig:LambertW_densities} shows \eqref{eq:cdf_LambertW_Y} and \eqref{eq:pdf_LambertW_Y} for various $\delta \geq 0$ with for four different input  $X \sim F_X(x \mid \boldsymbol \beta)$: for $\delta = h = 0$ the input equals the output (solid black); for larger $\delta$ the tails of $G_Y(y \mid \theta)$ and $g_Y(y \mid \theta)$ get heavier (dashed colored).

\begin{figure}
\centering
\subfloat{\includegraphics[width=\textwidth, trim = 0 5.1cm 0 0, clip = true]{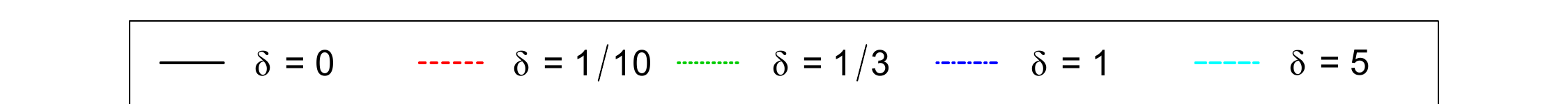}\label{fig:LambertW_pdfs_cdfs_legend_only}}

\subfloat
{\includegraphics[ width=.24\textwidth]{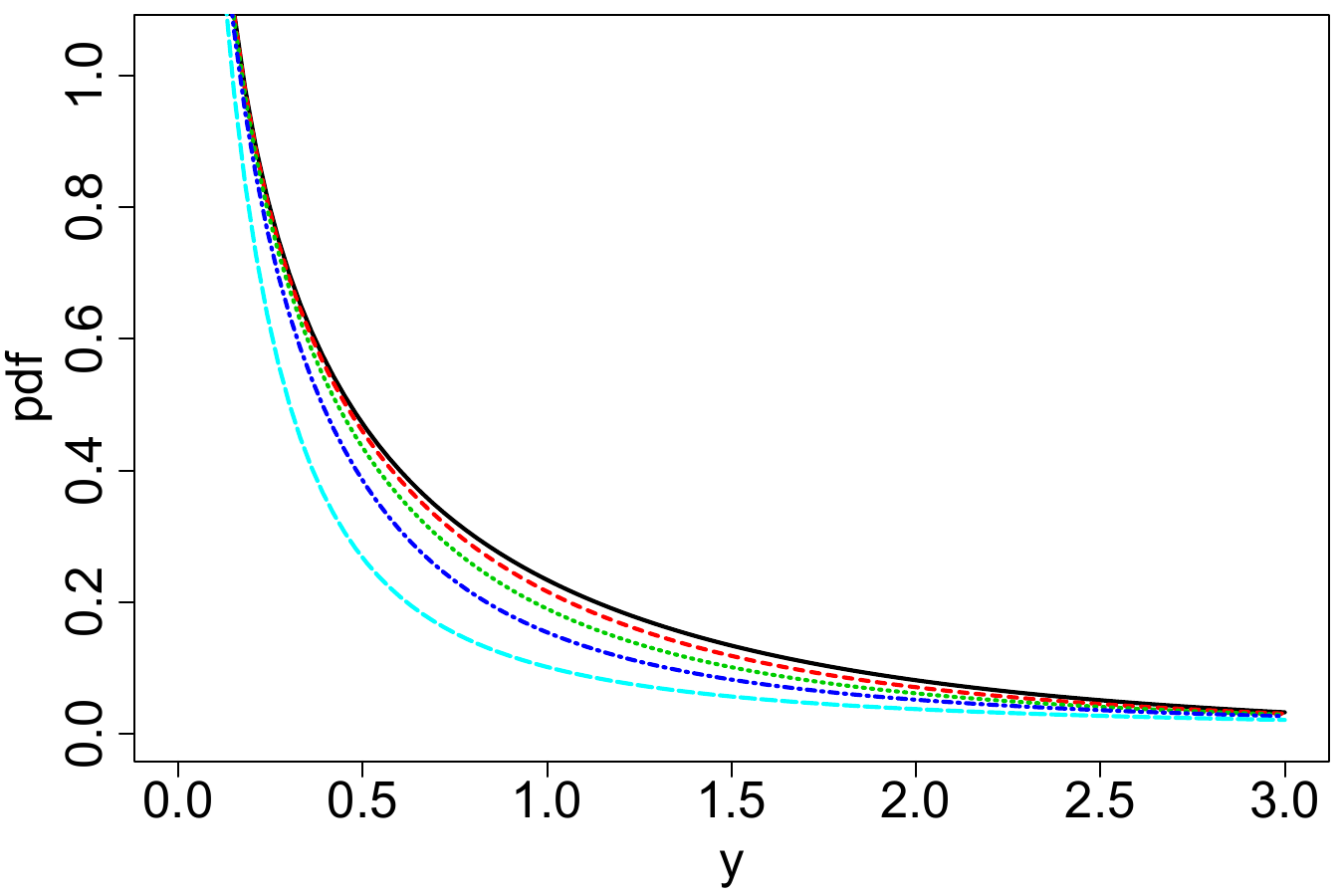}\label{fig:LambertW_chisq_density}}
\subfloat 
{\includegraphics[width=.24\textwidth]{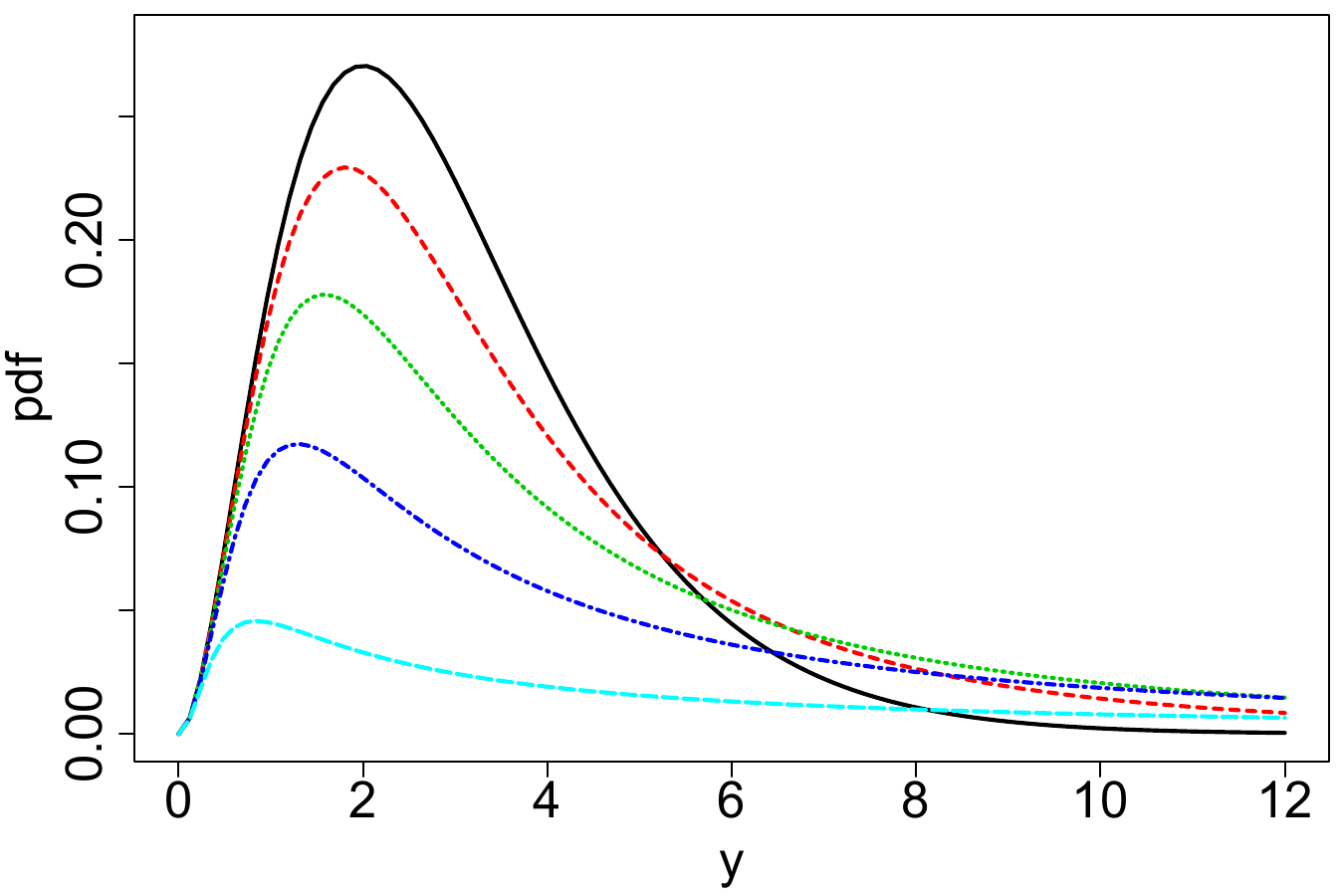}\label{fig:LambertW_exp_density}}
\subfloat 
{\includegraphics[width=.24\textwidth]{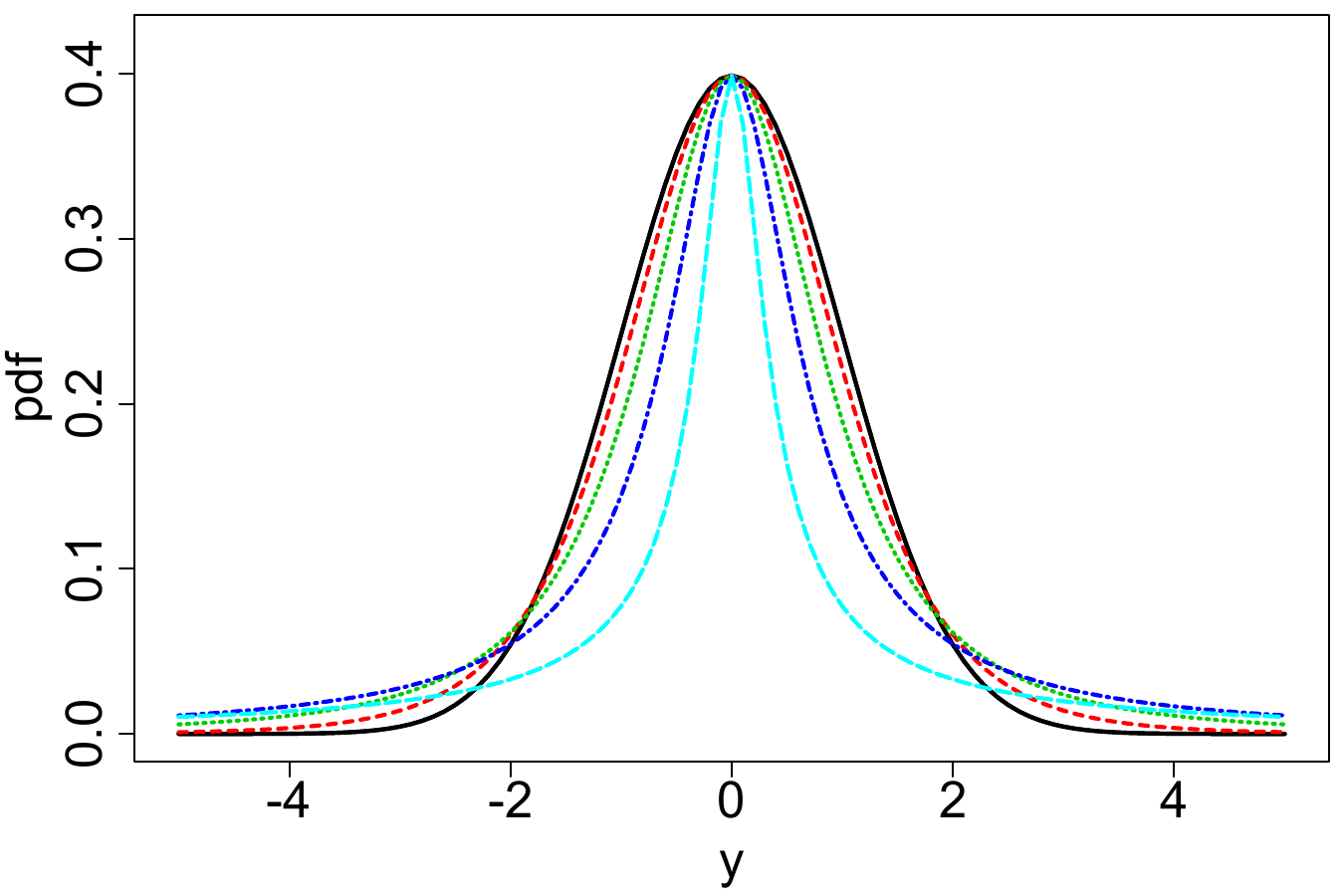}\label{fig:LambertW_normal_density}}
\subfloat 
{\includegraphics[width=.24\textwidth]{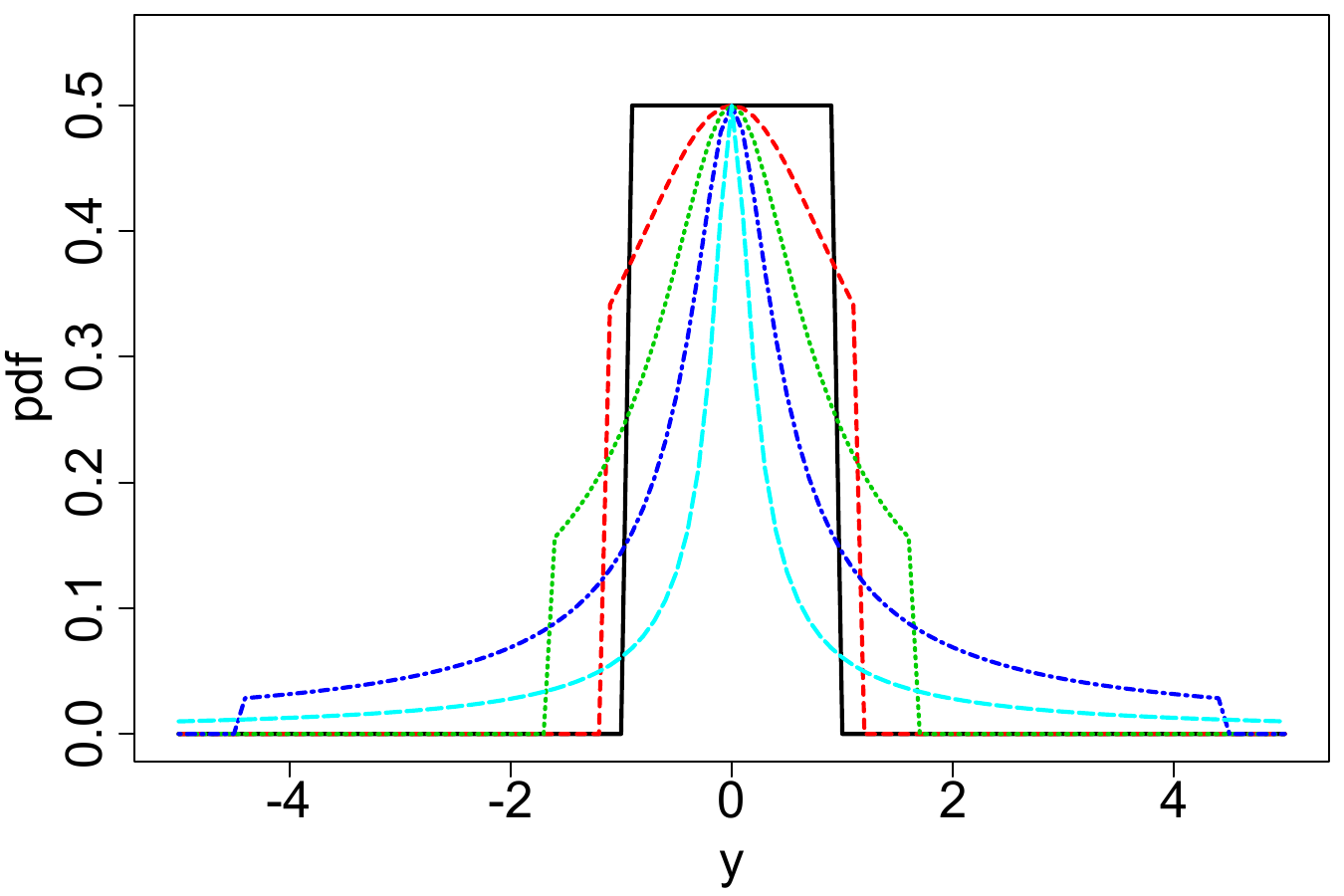}\label{fig:LambertW_unif_density}}

\setcounter{subfigure}{0}
\subfloat[Lambert W $\times$ $\chi_k^2$ \newline with $\boldsymbol \beta = k = 1$.]{\includegraphics[ width=.24\textwidth]{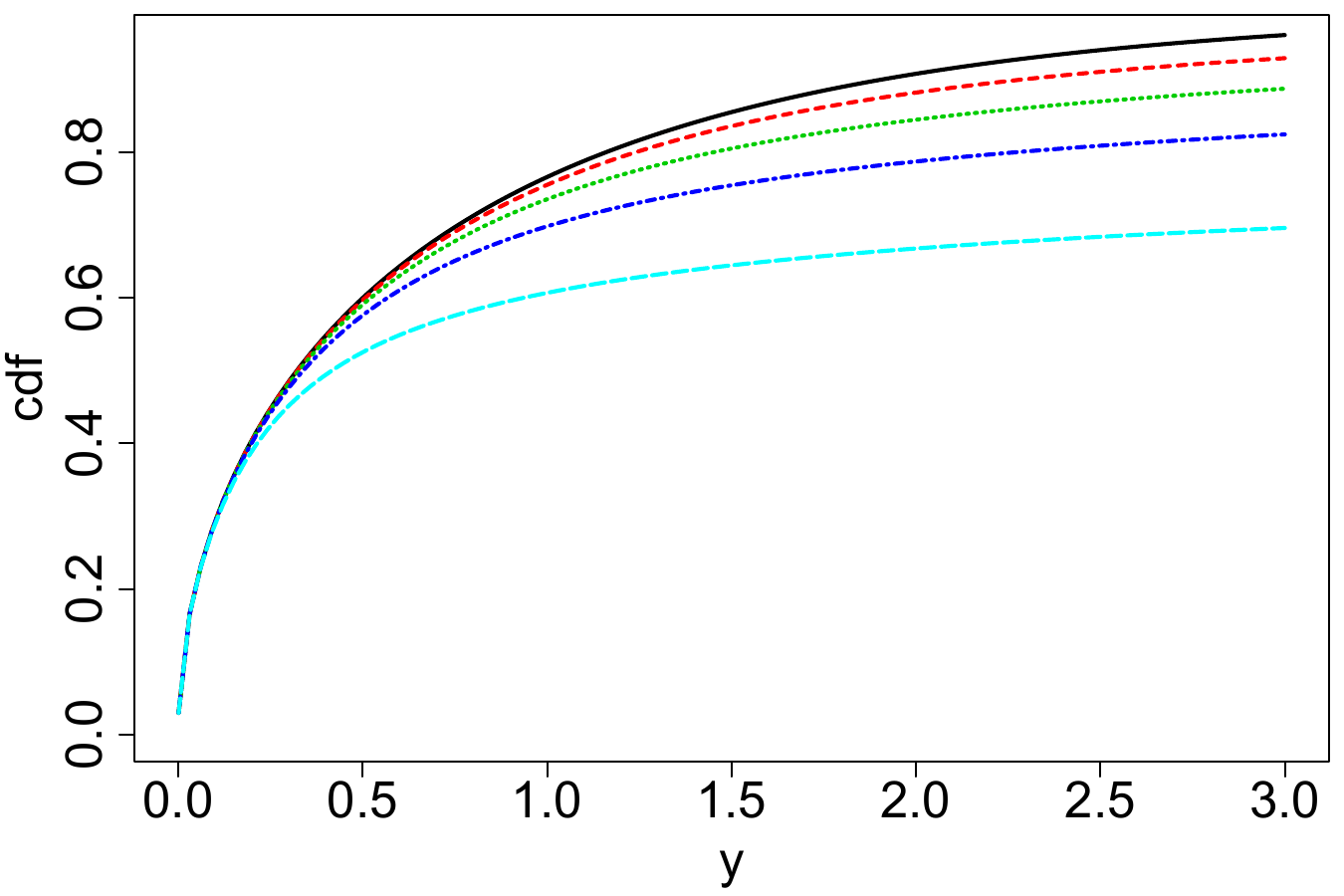}\label{fig:LambertW_chisq_cdfs}}
\subfloat[Lambert W $\times$ $\Gamma(s, r)$ \newline with $\boldsymbol \beta = (s, r) = (3,1)$.]{\includegraphics[width=.24\textwidth]{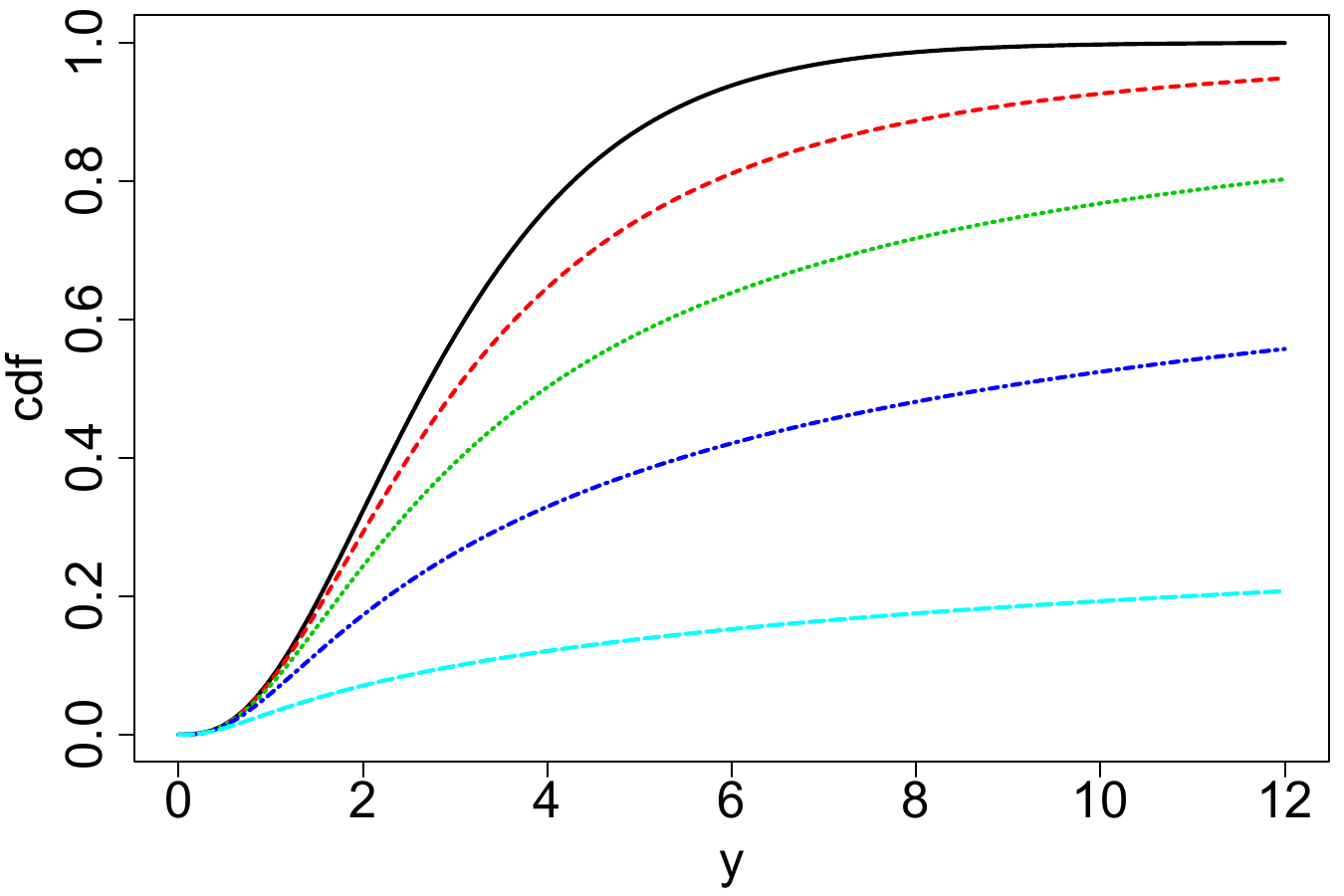}\label{fig:LambertW_exp_cdfs}}
\subfloat[Lambert W $\times$ $\mathcal{N}(\mu,\sigma^2)$ \newline with $\boldsymbol \beta = (\mu, \sigma) = (0,1)$.]{\includegraphics[width=.24\textwidth]{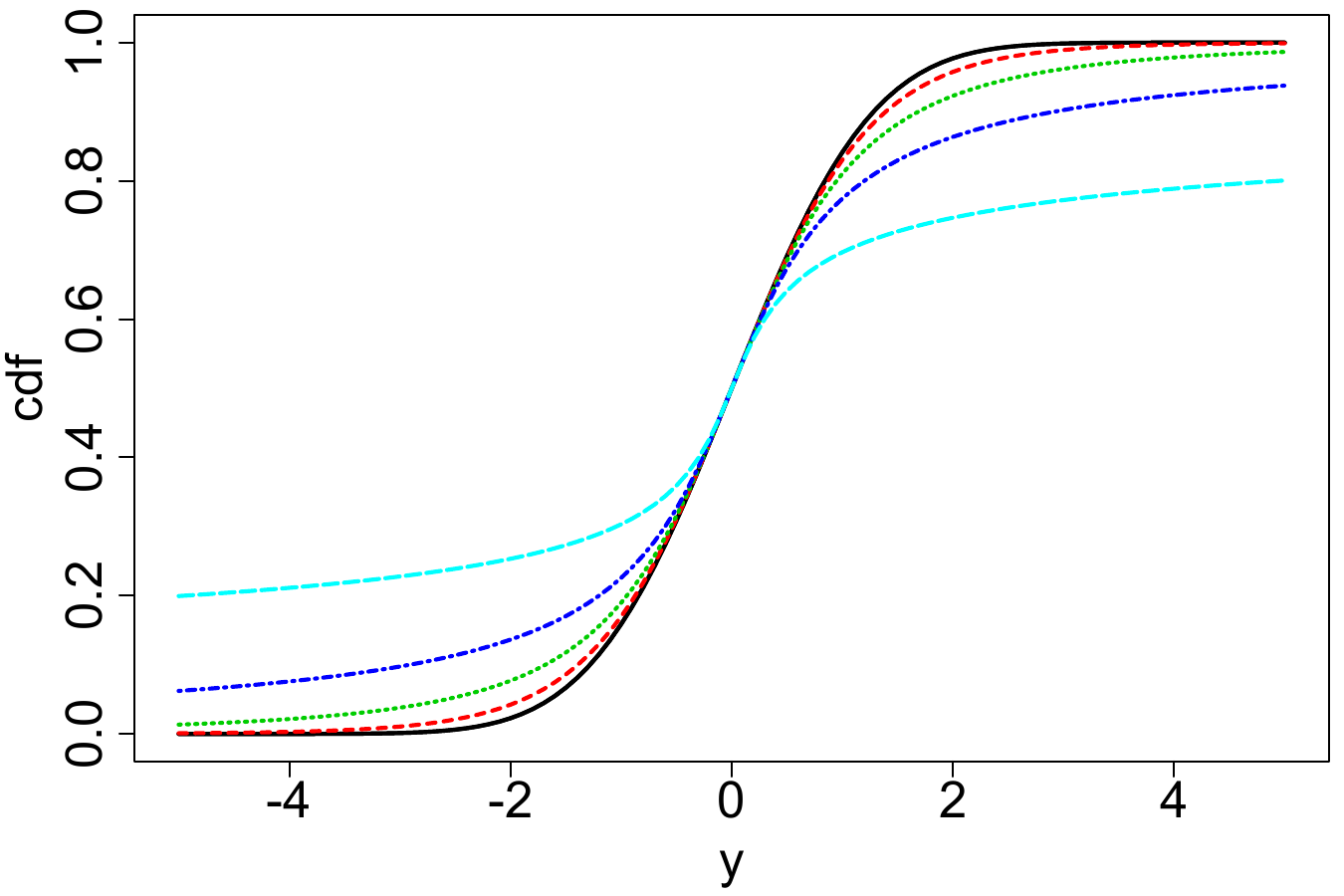}\label{fig:LambertW_normal_cdfs}}\subfloat[Lambert W $\times$ $U(a,b)$ \newline with $\boldsymbol \beta = (a, b) = (-1,1)$.]{\includegraphics[width=.24\textwidth]{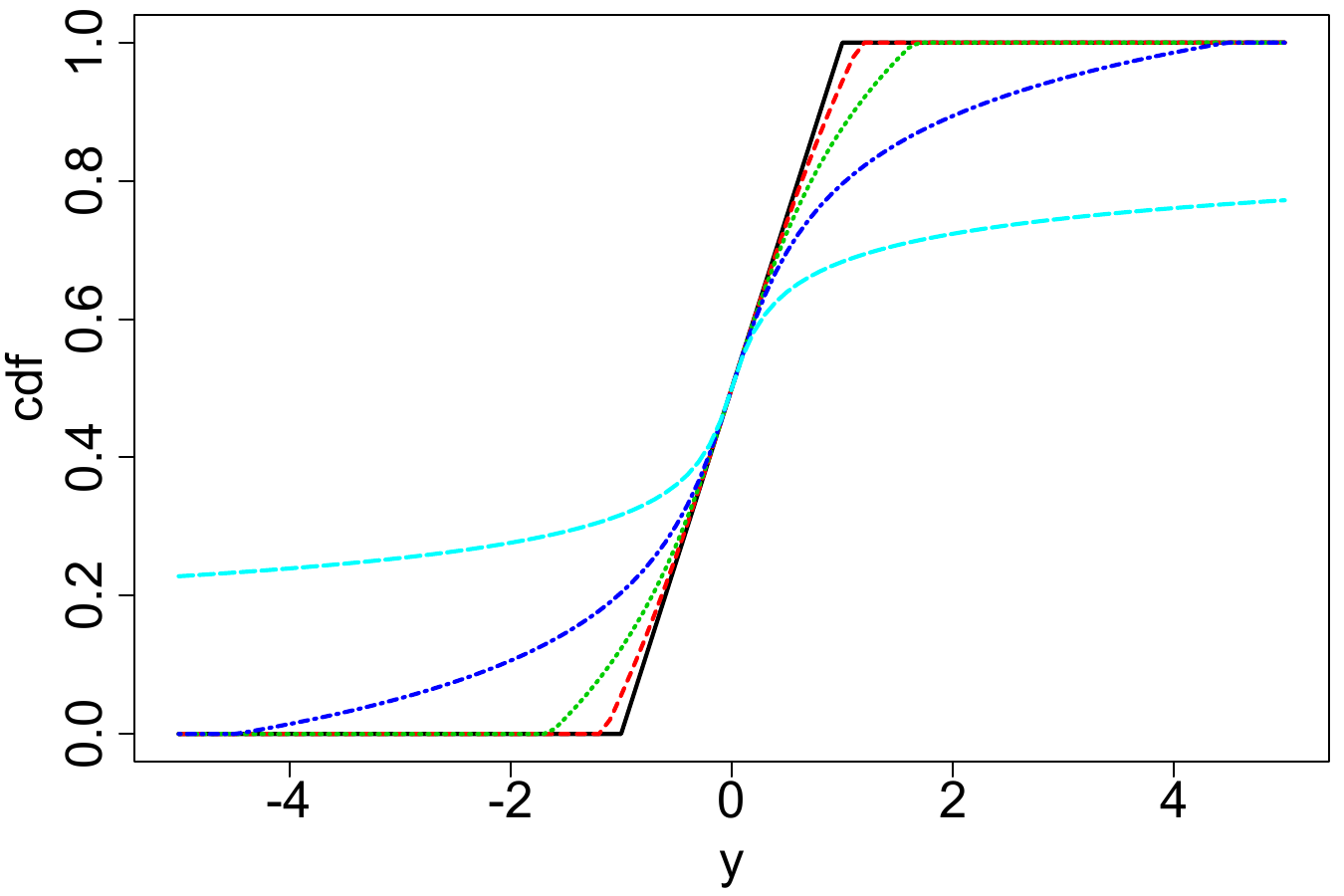}\label{fig:LambertW_unif_cdfs}}\\
\caption{Pdf (top) and cdf (bottom) of a heavy-tail (a) ``non-central, non-scaled'', (b) ``scale'', and (c and d) ``location-scale'' Lambert W $\times$ $F_X$ RV $Y$ for various degrees of heavy tails (color, dashed lines).}
\label{fig:LambertW_densities}
\end{figure} 

\begin{corollary} 
The cdf and pdf of $Z$ in \eqref{eq:trafo_hh} equal
\begin{align}
\label{eq:cdf_hh}
G_Z\left(z \mid \boldsymbol \beta, \delta_{\ell}, \delta_r \right) = 
 \begin{cases}
G_Z\left(z \mid \boldsymbol \beta, \delta_{\ell} \right), & \text{if } z \leq 0, \\
G_Z\left(z \mid \boldsymbol \beta, \delta_r \right),  & \text{if } z > 0,
\end{cases}
\end{align}
and
\begin{align}
\label{eq:pdf_hh}
g_Z\left(z \mid \boldsymbol \beta, \delta \right) =
 \begin{cases}
g_Z\left(z \mid \boldsymbol \beta, \delta_{\ell} \right), & \text{if } z \leq 0, \\
g_Z\left(z \mid \boldsymbol \beta, \delta_r \right),  & \text{if } z > 0.
\end{cases}
\end{align}
\end{corollary}

\subsection{Quantile Function}
\label{subsec:quantile_general}
Quantile fitting has been the standard technique to estimate $\mu_x$, $\sigma_x$, and $\delta$ of Tukey's $h$. In particular, the median of $Y$ and $X$ are equal. Thus for symmetric, location-scale family input the sample median of $\mathbf{y}$ is a robust estimate for $\mu_x$ for any $\delta \geq 0$ (see also Section \ref{sec:simulations}). General quantiles can be computed via \citep{Hoaglin06}
\begin{equation}
\label{eq:quantile_function}
y_{\alpha} = u_{\alpha} \exp \left( \frac{\delta}{2} u_{\alpha}^2 \right) \sigma_x + \mu_x,
\end{equation}
where $u_{\alpha} = W_{\delta}(z_{\alpha})$ are the $\alpha$-quantiles of $F_U(u)$. As quantiles of $U$ are typically tabulated, or easily available in software packages, \eqref{eq:quantile_function} can be computed very efficiently using $u_{\alpha}$ and $\tau$.\\

This simple conversion can be especially useful for education: teaching heavy-tailed statistics in introductory courses soon becomes too difficult using e.g., Cauchy or $\alpha$-stable distributions. Yet, transforming data via Lambert's $W$, using previously learned methods for the Gaussian case, and then transforming the inference back to the ``heavy-tailed world'' -- e.g., transforming quantiles using \eqref{eq:quantile_function} -- is straightforward. Thus the Lambert W $\times$ $F_X$ framework can promote heavy-tailed statistics in introductory courses.
\section{Tukey's h distribution: Gaussian input}
\label{sec:Gaussian}

For Gaussian input Lambert W $\times$ $F_X$ equals Tukey's $h$, which has been studied extensively. \citet{Dutta02_MeasuringSkewnessKurtosis} show that
\begin{equation}
\E Z^n = \begin{cases}
0, & \text{if } \text{n is odd and } n < \frac{1}{\delta}, \\
\frac{n! (1- n \delta)^{\frac{-(n+1)}{2}}}{2^{n/2} (n/2)!} ,  & \text{if } \text{n is even and } n < \frac{1}{\delta}, \\
\nexists ,  & \text{if } n > \frac{1}{\delta},
\end{cases}
\end{equation}
which in particular implies \citep{Todd08_Parametric_pdf_cdf_tukey} 
\begin{align}
\E Z &= \E Z^3 = 0, \text{ if } \delta < 1 \text{ and } 1/3 \text{, respectively} \\
\label{eq:Y_moments}
\text{ and }\E Z^2 & = \frac{1}{(1 - 2 \delta)^{3/2}}, \text{ if } \delta < \frac{1}{2} , \quad \E Z^4 = 3 \frac{1}{(1-4\delta)^{5/2}}, \text{ if } \delta < \frac{1}{4}.
\end{align}
Thus the kurtosis of $Y$ equals (see Fig.\ \ref{fig:var_kurt})
\begin{equation}
\label{eq:kurt_Y_delta}
\gamma_2(\delta) = 3 \frac{(1-2 \delta)^3}{(1-4\delta)^{5/2}} \text{ for } \delta < 1/4.
\end{equation}
For $\delta = 0$, \eqref{eq:Y_moments} and \eqref{eq:kurt_Y_delta} reduce to the familiar Gaussian values. 

\begin{corollary} 
\label{cor:Tukey_h}
The cdf of Tukey's $h$ equals
\begin{align}
\label{eq:cdf_Tukey_h}
G_Y\left(y \mid \mu_x, \sigma_x, \delta \right) = \Phi \left(\frac{ W_{\tau}(y) - \mu_x}{\sigma_x} \right),
\end{align}
where $\Phi(u)$ is the cdf of a standard Normal. The pdf equals (for $\delta > 0$)
\begin{align}
\label{eq:pdf_Tukey_h}
g_Y\left(y \mid  \mu_x, \sigma_x, \delta \right) & = \frac{1}{\sqrt{2 \pi}} \exp{ \left( - \frac{1 + \delta}{2}  W_{\delta}\left( \frac{y - \mu_x}{\sigma_x}  \right)^2 \right)} \cdot \frac{1}{ 1 + W\left( \delta \left( \frac{y - \mu_x}{\sigma_x} \right)^2 \right) }
\end{align}
\end{corollary}
\begin{proof}
Take $X \sim \mathcal{N}\left( \mu_x, \sigma_x^2 \right)$ in Theorem \ref{theorem:cdf_Y}.
\end{proof}

Section \ref{sec:MLE} studies functional properties of \eqref{eq:pdf_Tukey_h} in more detail.

\subsection{Tukey's h versus student's t}
Student's $t_{\nu}$ distribution with $\nu$ degrees of freedom is often used to model heavy-tailed data  \citep{Yan05, WongChanKam09_tmixtureAR}, as its tail index equals $\nu$. Thus the $n$th moment of a student t RV $T$ exists if $n < \nu$. In particular,
\begin{equation}
\label{eq:first_three_moments_nu}
\E T = \E T^3 = 0 \text{ if } \nu < 1 \text{ or} < 3, \quad \E T^2 = \frac{\nu}{\nu-2} = \frac{1}{1-\frac{2}{\nu}} \text{ if } \frac{1}{\nu} < \frac{1}{2},
\end{equation}
and kurtosis
\begin{equation}
\label{eq:kurt_t_nu}
\gamma_2(\nu) = 3 \frac{\nu - 2}{\nu -4} = 3 \frac{1 - 2 \frac{1}{\nu}}{1 - 4 \frac{1}{\nu}} \text{ if } \frac{1}{\nu} < \frac{1}{4}.
\end{equation}

Comparing \eqref{eq:kurt_t_nu} and \eqref{eq:Y_moments} with \eqref{eq:kurt_Y_delta} and \eqref{eq:first_three_moments_nu} shows a natural association between $1 / \nu$ and $\delta$ and a close similarity between the first four moments of student's $t$ and Tukey's $h$ (Fig.\ \ref{fig:var_kurt}). By continuity and monotonicity, the first four moments of a location-scale $t$ distribution can always be exactly matched by a corresponding location-scale Lambert W $\times$ Gaussian. Thus if student's $t$ is used to model heavy tails, and not as the true distribution of a test statistic, it might be worthwhile to also fit heavy tail Lambert W $\times$ Gaussian distributions for an equally valuable ``second opinion''. For example, a parallel analysis on S\&P 500 $\log$-returns in Section \ref{sec:SP500} leads to divergent inference regarding the existence of fourth moments.\edit{what about loglik?} Additionally, the Lambert W approach allows to Gaussianize and thus reveal hidden patterns in the data; patterns that can be easily overseen in presence of heavy tails (Section \ref{sec:solar_flares}).

\begin{figure}[!t]
\centering
\subfloat[Variance]{\includegraphics[width=.45\textwidth]{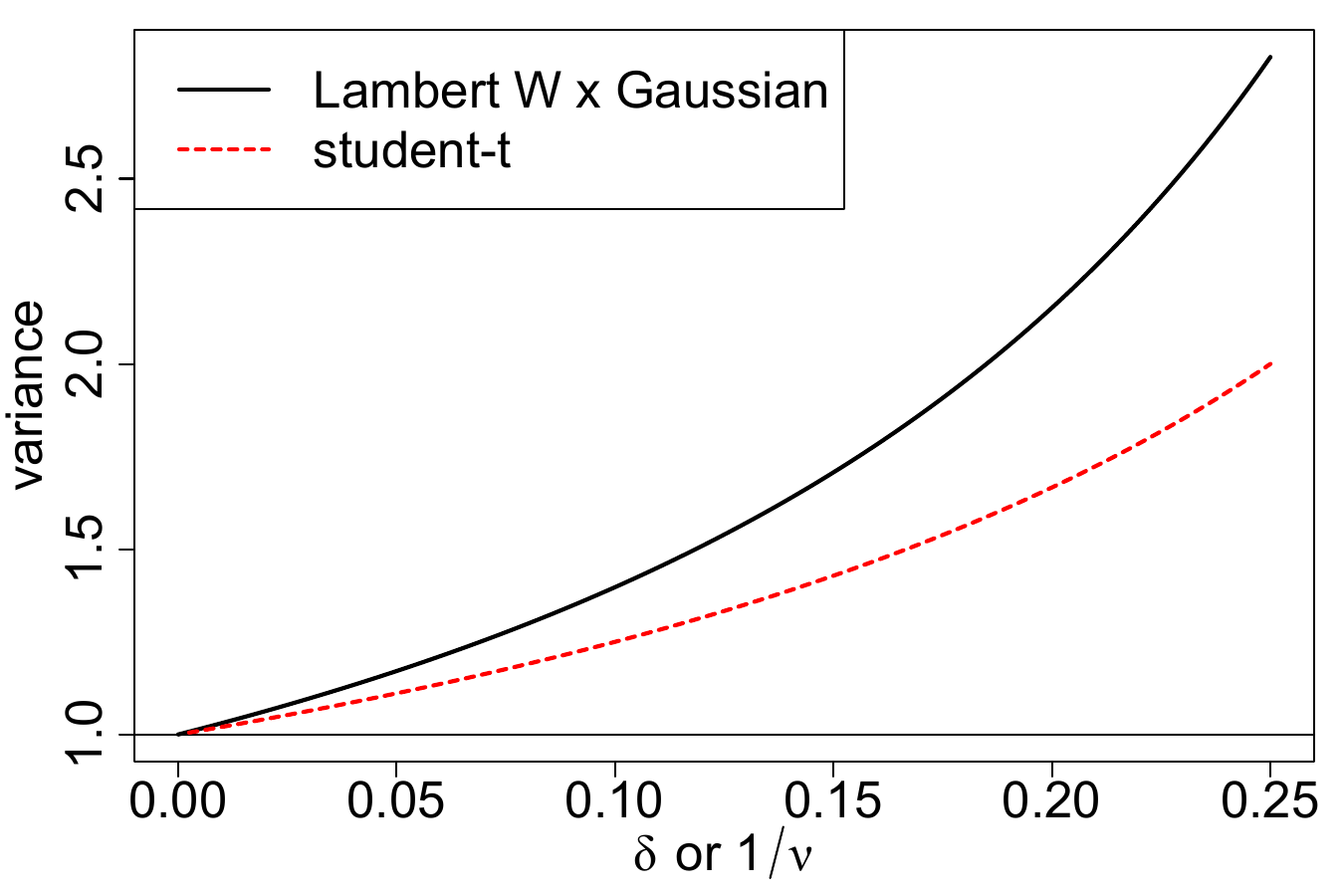}\label{fig:var_t_LambertW}}
\hspace{0.01\textwidth}
\subfloat[Kurtosis]{\includegraphics[width=.45\textwidth]{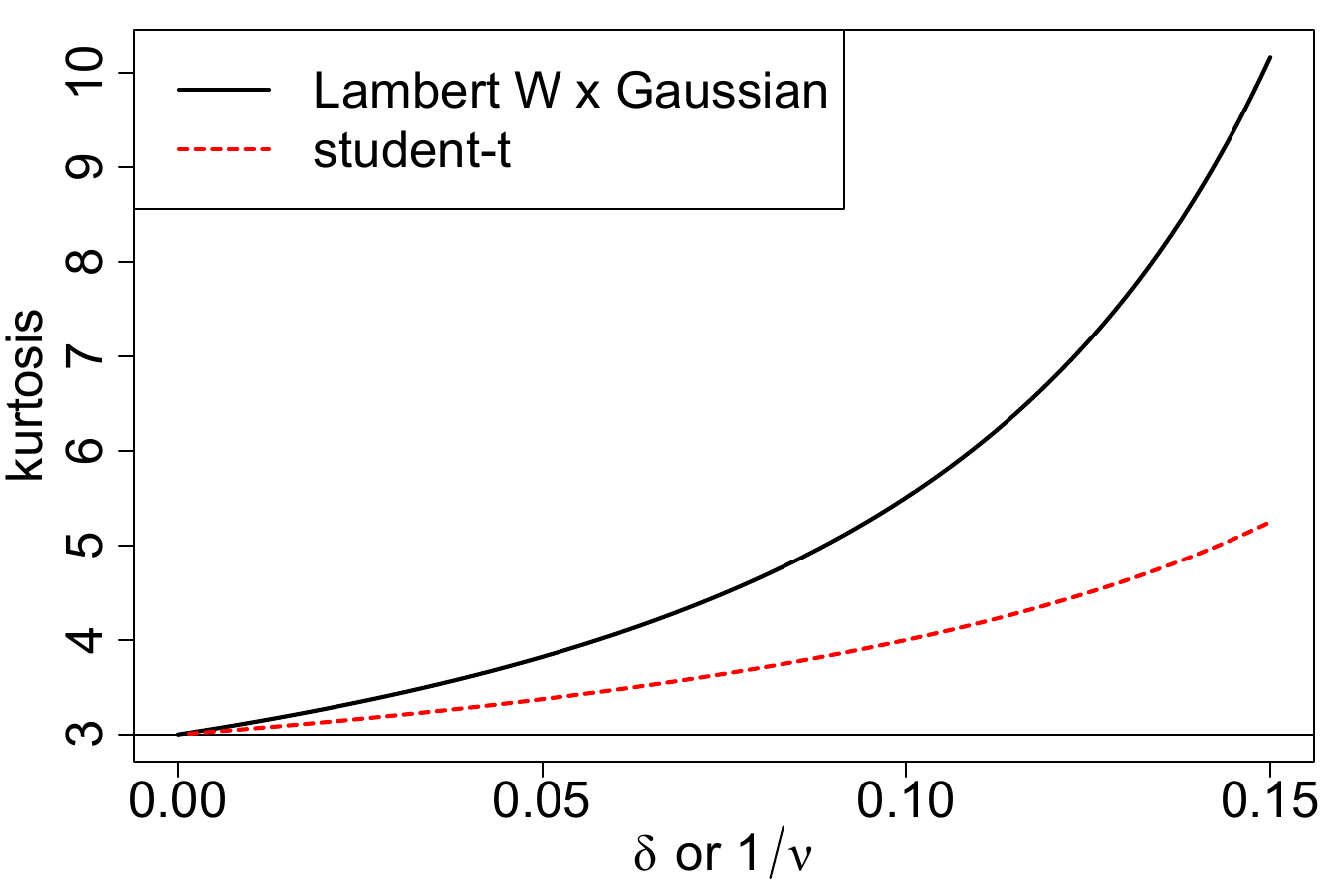}\label{fig:kurt_t_LambertW}}
\caption{\label{fig:var_kurt} Comparing moments of Lambert W $\times$ Gaussian and student's $t$.}

\end{figure}

\section{Parameter Estimation}
\label{sec:estimation}
For a sample of $N$ independent identically distributed (i.i.d.) observations $\mathbf{y}=( y_1, \ldots, y_N )$ from transformation \eqref{eq:normalized_LambertW_Y}, $\theta = (\boldsymbol \beta, \delta)$ has to be estimated from the data. Due to the lack of a closed form pdf of $Y$, this has been typically done by matching quantiles or a method of moments estimator \citep{Field04, MorgenthalerTukey00_FittingQuantiles, Todd08_Parametric_pdf_cdf_tukey}. These inefficient methods can now be replaced by the -- fast and usually efficient -- maximum likelihood estimator (MLE) using the pdf in \eqref{eq:pdf_LambertW_Y}. \citet{RaynerMacGillivray02_numericalMLE} introduce a numerical MLE procedure based on quantile functions, but they conclude that ``sample sizes significantly larger than $100$ should be used to obtain reliable estimates through maximum likelihood''. Simulations in Section \ref{sec:simulations} show that log-likelihood maximization with the Lambert W methodology converges quickly and is accurate even for sample sizes as small as $N = 10$.

\subsection{Maximum Likelihood Estimation (MLE)}
\label{sec:MLE}
For an i.i.d.\ sample $\mathbf{y} \sim g_Y\left(y \mid \boldsymbol \beta, \delta \right)$ the log-likelihood function equals
\begin{equation}
\label{eq:loglikelihood_Y}
\ell \left( \theta \mid \mathbf{y} \right) = \sum_{i=1}^{N} \log g_Y(y_i  \mid  \boldsymbol \beta, \delta).
\end{equation}
The \textsc{MLE} is that $\theta = (\boldsymbol \beta, \delta)$ which maximizes \eqref{eq:loglikelihood_Y}, i.e.\
\begin{equation}
\widehat{\theta}_{MLE} = \left(\widehat{\boldsymbol \beta}, \widehat{\delta}\right)_{MLE} = \arg \max_{\boldsymbol \beta, \delta} \ell \left(  \boldsymbol \beta, \delta \mid \mathbf{y}  \right).
\end{equation}
Since $g_Y(y_i \mid \boldsymbol \beta, \delta)$ is a function of $f_X(x_i \mid \boldsymbol \beta)$, the \textsc{MLE} depends on the specification of the input density. Eq.\ \eqref{eq:loglikelihood_Y} can be decomposed as
\begin{align}
\label{eq:loglikelihood_Y_with_xi}
\ell \left( \boldsymbol \beta, \delta \mid \mathbf{y} \right) & = \ell \left(  \boldsymbol \beta \mid \mathbf{x}_{\tau} \right) + \mathcal{R}\left( \tau \mid \mathbf{y} \right),
\end{align}
where
\begin{equation}
\label{eq:likelihood_x.hat}
\ell \left( \boldsymbol \beta \mid \mathbf{x}_{\tau} \right) = \sum_{i=1}^{N} \log f_X\left( W_{\delta}\left( \frac{y_i - \mu_x}{\sigma_x} \right) \sigma_x + \mu_x \mid \boldsymbol \beta \right) = \sum_{i=1}^{N} \log f_X\left( \mathbf{x}_{\tau} \mid \boldsymbol \beta \right)
\end{equation}
is the log-likelihood of the back-transformed data $\mathbf{x}_{\tau}$ and
\begin{equation}
\label{eq:R_penalty_all}
\mathcal{R}\left( \tau \mid \mathbf{y} \right) = \sum_{i=1}^{n} \log R\left( \mu_x, \sigma_x, \delta \mid  y_i \right),
\end{equation}
where
\begin{equation}
\label{eq:R_penalty}
R\left( \mu_x, \sigma_x, \delta \mid  y_i \right)  = \frac{W_{\delta}\left( \frac{y_i - \mu_x}{\sigma_x} \right)}{  \frac{y_i - \mu_x}{\sigma_x}\left[ 1 + \delta \left( W_{\delta}\left( \frac{y_i - \mu_x}{\sigma_x} \right) \right)^2 \right]}.
\end{equation}
Note that $R\left( \mu_x, \sigma_x, \delta \mid y_i \right)$ only depends on $\mu_x(\boldsymbol \beta)$ and $\sigma_x(\boldsymbol \beta)$ (and $\delta$), but not necessarily on every coordinate of $\boldsymbol \beta$.\\

\begin{figure}[!t]
\centering
\subfloat[Penalty $R\left( \mu_x, \sigma_x, \delta \mid y_i \right)$ \eqref{eq:R_penalty} as a function of $\delta$ and $y$ ($\mu_x = 0$ and $\sigma_x = 1$).
]{\includegraphics[width=.32\textwidth]{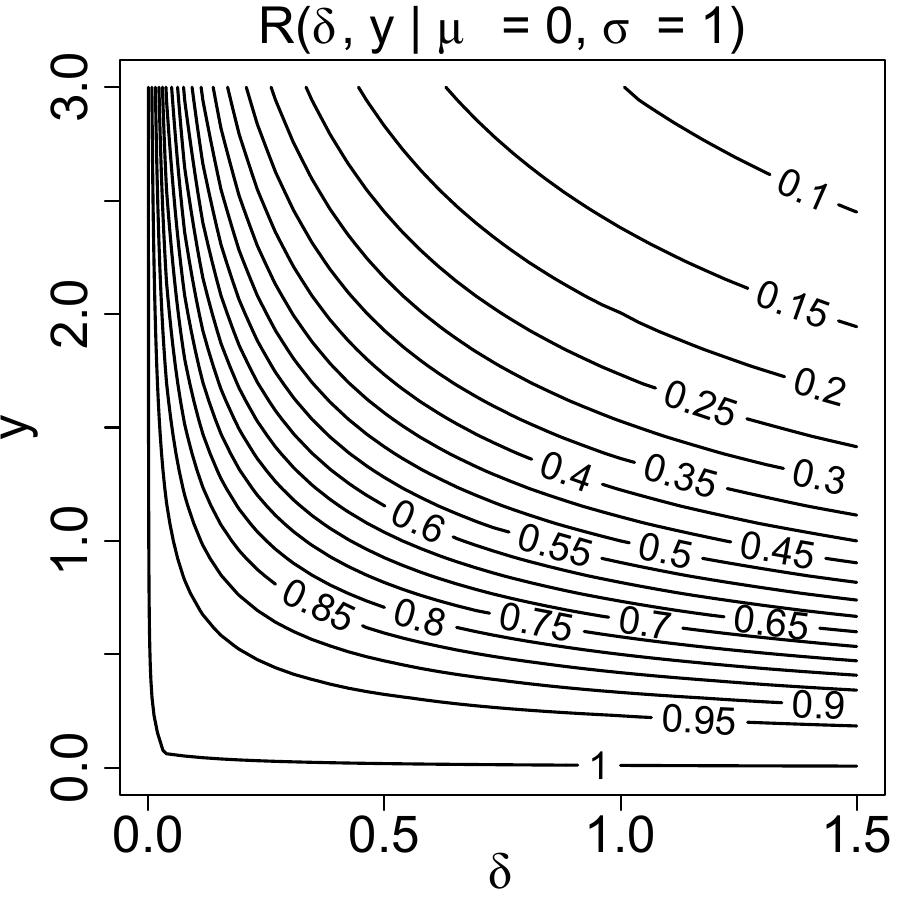}\label{fig:R_delta_y}}
\hspace{0.02\linewidth}
\subfloat[Sample $\mathbf{z}$ of Lambert W $\times$ Gaussian with $\delta = 1/3$ (left); Log-likelihood $\ell \left( \theta \mid \mathbf{y} \right)$ (solid, black) decomposes in input log-likelihood (dotted, green) and penalty (dashed, red).
]
{\includegraphics[width=.6\textwidth]{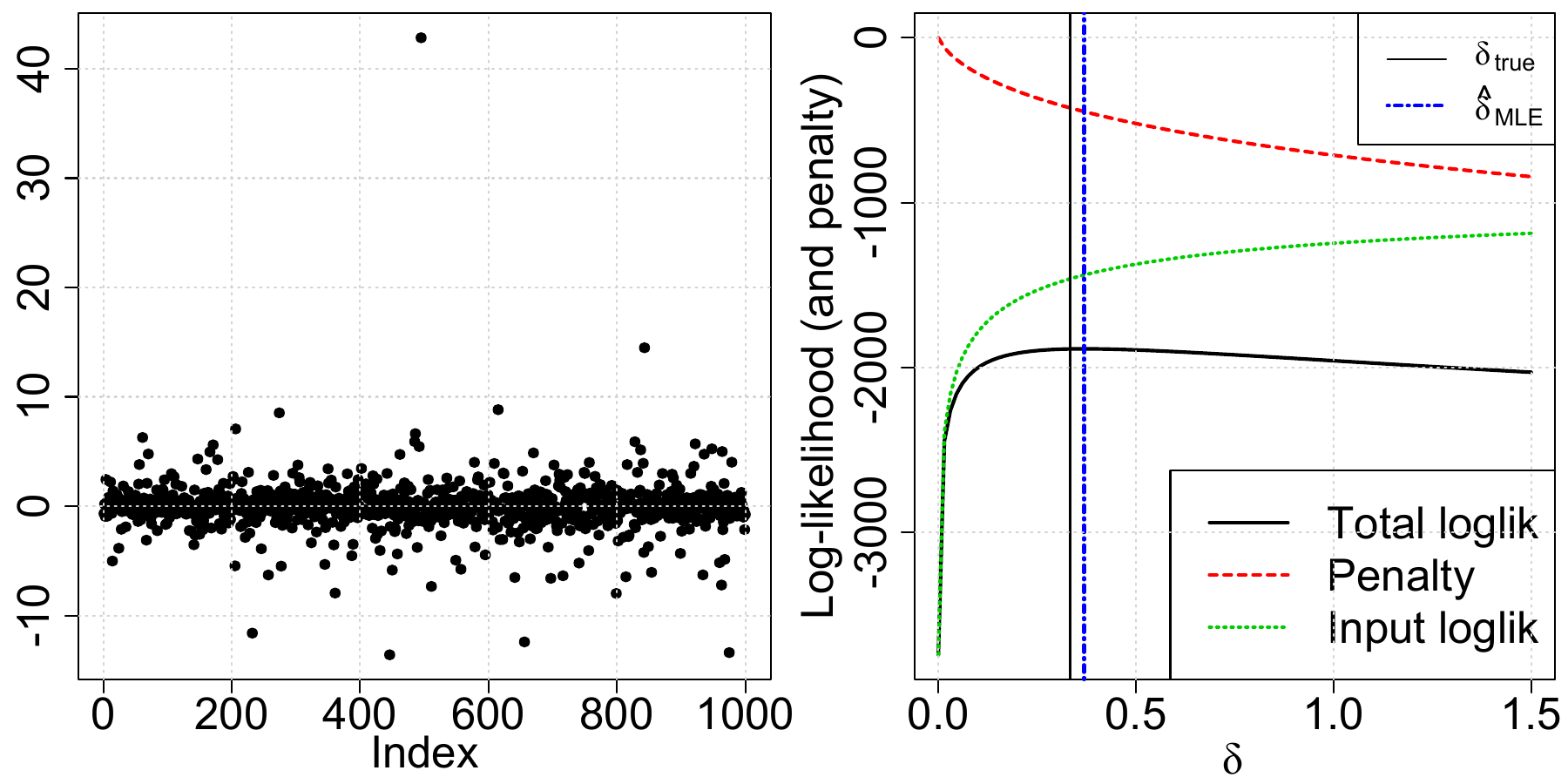}\label{fig:sample_loglik_decomposition}}
\caption{\label{fig:loglik_penalty} Log-likelihood decomposition for Lambert W $\times$ $F_X$ distributions.}
\end{figure} 

Decomposition \eqref{eq:loglikelihood_Y_with_xi} shows the difference between the exact MLE $ \left(\widehat{\boldsymbol \beta}, \widehat{\delta}\right)$ based on $\mathbf{y}$ and the approximate MLE $\widehat{\boldsymbol \beta}_{\mathbf{x}_{\tau}}$ based on $\mathbf{x}_{\tau}$ alone: if we knew $\tau = (\mu_x, \sigma_x, \delta)$ beforehand, then we could back-transform $\mathbf{y}$ to $\mathbf{x}_{\tau}$ and estimate $\widehat{\boldsymbol \beta}_{\mathbf{x}_{\tau}}$ from $\mathbf{x}_{\tau}$ (maximize \eqref{eq:likelihood_x.hat} with respect to $\boldsymbol \beta$). In practice, however, $\tau$ must also be estimated and this enters the likelihood via the additive term $\mathcal{R}\left( \tau \mid \mathbf{y} \right)$. A little calculation shows that for any $y_i \in \R$, $\log R\left( \mu_x, \sigma_x, \delta \mid y_i \right) \leq 0$ if $\delta \geq 0$, with equality if and only if $\delta = 0$. Thus $\mathcal{R}\left( \tau \mid \mathbf{y} \right)$ can be interpreted as a penalty for transforming the data. Maximizing \eqref{eq:loglikelihood_Y_with_xi} faces a trade-off between transforming the data to follow $f_X(x \mid \boldsymbol \beta)$ (and thus increasing $\ell \left(\boldsymbol \beta \mid \mathbf{x}_{\widehat{\tau}} \right)$) versus the penalty of a more extreme transformation (and thus decreasing $\mathcal{R}\left( \tau \mid \mathbf{y} \right)$) -- see Fig.\ \ref{fig:sample_loglik_decomposition}.\\

Figure \ref{fig:R_delta_y} shows a contour plot of $R\left( \mu_x = 0, \sigma_x = 1, \delta \mid y \right)$ as a function of $\delta$ and $y = z$. The penalty for transforming the data increases (in absolute value) either if $\delta$ gets larger (for fixed $y$) or for larger $y$ (for fixed $\delta$). In both cases, increasing $\delta$ makes the transformed data $W_{\delta}(z)$ get closer to $0 = \mu_x$, which in turn increases its input likelihood. For $\delta = 0$, the penalty disappears since input equals output; for $y = 0$ there is no penalty since $W_{\delta}(0) = 0$ for all $\delta$.

Figure \ref{fig:sample_loglik_decomposition} shows a random sample ($N = 1000$) $\mathbf{z} \sim $ Lambert W $\times$ Gaussian with $\delta = 1/3$ and the decomposition of the log-likelihood as in \eqref{eq:loglikelihood_Y_with_xi}. Since $\boldsymbol \beta = (0,1)$ is known, the likelihood and penalty are only functions of $\delta$. The monotonicity of the penalty (decreasing, red) and the input likelihood (increasing, green) as a function of $\delta$ is not particular to this sample, but holds true in general (see Theorem \ref{thm:MLE_delta} below). This monotonicity in each component implies that their sum (black line) has a unique maximum; here $\widehat{\delta}_{MLE} = 0.37$ (blue, dashed vertical line).\\

The maximization of \eqref{eq:loglikelihood_Y_with_xi} can be carried out numerically. Here I show existence and uniqueness of $\widehat{\delta}_{MLE}$ assuming that $\mu_x$ and $\sigma_x$ are known. Theoretical results for $\widehat{\theta}_{MLE}$ remain for future work. Given the ``nice'' form of $g_Y(y)$ - continuous, twice differentiable,\footnote{Assuming that $f_X(\cdot)$ is twice differentiable.} its support does not depend on the parameter, etc.\ - the MLE for $\theta = (\boldsymbol \beta, \delta)$ should have the usual optimality properties \citep{LehmannCasella98_TPE}.

\subsubsection{Properties of The MLE For The Heavy Tail Parameter}
Without loss of generality let $\mu_x = 0$ and $\sigma_x = 1$. In this case
\begin{align}
\label{eq:loglik_delta_z}
\ell \left( \delta \mid \mathbf{z} \right) & \propto -\frac{1}{2} \sum_{i=1}^{N} \left[ W_{\delta}(z_i) \right]^2 + \sum_{i=1}^{N} \log \frac{W_{\delta}\left( z_i \right)}{z_i} - \log  \left( 1 + \delta \left[ W_{\delta}(z_i) \right]^2 \right) \\
& = -\frac{1 + \delta}{2} \sum_{i=1}^{N} \left[ W_{\delta}(z_i) \right]^2 - \sum_{i=1}^{N} \log  \left( 1 + \delta \left[ W_{\delta}(z_i) \right]^2 \right).
\end{align}

\begin{theorem}[Unique MLE for $\delta$] 
\label{thm:MLE_delta}
Let $Z$ have a Lambert W $\times$ Gaussian distribution, where $\mu_x = 0$ and $\sigma_x = 1$ are assumed to be known and fixed. Also consider only the case $\delta \in [0, \infty)$.\footnote{While for some samples $\mathbf{z}$ the MLE also exists for $\delta <0$, it can not be guaranteed for all $\mathbf{z}$. If $\delta < 0$ (and $z \neq 0$), then $W_{\delta}(z)$ is either not unique in $\R$ (principal and non-principal branch) or may not even have a real-valued solution.}
\begin{enumerate}[a)]
\item If 
\begin{equation}
\label{eq:cond_delta_MLE_greater_0}
\frac{\sum_{i=1}^{n} z_i^4}{ \sum_{i=1}^{n} z_i^2} \leq 3,
\end{equation}
then $\widehat{\delta}_{MLE} = 0$. 
\end{enumerate}
If \eqref{eq:cond_delta_MLE_greater_0} does \emph{not} hold, then 
\begin{enumerate}[a)]
\setcounter{enumi}{1}
\item $\widehat{\delta}_{MLE} > 0$ exists and is a positive solution to
\begin{equation}
\label{eq:derivative_equal_0}
\sum_{i=1}^{N} z_i^2 W'(\delta z_i^2) \left(\frac{1}{2} W_{\delta}\left(z_i\right)^{2} - \left( \frac{1}{2} + \frac{1}{1 + W\left(\delta z_i^2 \right)} \right) \right) = 0.
\end{equation}
\item  There is only one such $\delta$ satisfying \eqref{eq:derivative_equal_0}, i.e.\ $\widehat{\delta}_{MLE}$ is unique.
\end{enumerate}
\end{theorem}

Condition \eqref{eq:cond_delta_MLE_greater_0} says that $\widehat{\delta}_{MLE} > 0$ only if the data is heavy-tailed enough. Points b) and c) guarantee that there is no ambiguity in the heavy tail estimate. This is an advantage over student's $t$ distribution, for example, which has numerical problems and local maxima for unknown (and small) $\nu$ ($\leftrightarrow$ large $\delta$) \citep[see also][]{Fernandez99_tdistribution_pitfalls, LiuRubin95}. On the contrary,  $\widehat{\delta}_{MLE}$ is always a global maximum.\\

The log-likelihood and its gradient depend on $\delta$ and $\mathbf{z}$ only via $W_{\delta}(\mathbf{z})$. Given the heavy tails in $\mathbf{z}$ (for $\delta > 0$) one might expect convergence issues for larger $\delta$ (e.g.\ expected log-likelihood, Fisher information). However, $W_{\delta}(Z) \sim \mathcal{N}(0,1)$ for the true $\delta \geq 0$, and close to a standard Gaussian if $\widehat{\delta}_{MLE} \approx \delta$. Thus the performance of the MLE should not get worse for large $\delta$ as long as the initial estimate is close enough to the truth. Simulations in Section \ref{sec:simulations} support this conjecture, even for $\widehat{\theta}_{MLE}$.

\subsection{Iterative Generalized Method of Moments (IGMM)}
A disadvantage of the \textsc{MLE} is the mandatory a-priori specification of the input distribution. Especially for heavy-tailed data the eye is a bad judgement to choose a particular parametric $f_X(x \mid \boldsymbol \beta)$. It would be useful to directly estimate $\tau$, without the intermediate step of estimating $\theta$ first (and thus no distributional assumption for the input is necessary).

\citet{GMGLambertW_Skewed} presented an estimator for $\tau$ based on iterative generalized methods of moments (IGMM). The idea of IGMM is to find a $\tau$ such that the back-transformed data $\mathbf{x}_{\tau}$ has desired properties, e.g., is symmetric or has kurtosis $3$.  An estimator for $\mu_x$, $\sigma_x$, and $\delta$ can be constructed completely analogously to the skewed IGMM, with the advantage that the heavy tail transformation is bijective (the skewed transformation is not).  Since the algorithm is entirely analogous to the skewed case, details are given in the Supplementary Material, Appendix \ref{sec:details_IGMM}.\\

An advantage of \textsc{IGMM} is that it requires less specific knowledge about the input distribution. Usually, it is also faster than the MLE. Once $\widehat{\tau}_{IGMM}$ has been obtained, the back-transformed $\mathbf{x}_{\widehat{\tau}_{\textsc{IGMM}}}$ can be used to check if $X$ has characteristics of a known parametric distribution $F_X(x \mid \boldsymbol \beta)$.  It must be noted though that testing for a particular distribution $F_X$ are too optimistic as $\mathbf{x}_{\widehat{\tau}}$ will have ``nicer'' properties regarding $F_X$ than the true $\mathbf{x}$ would have. However, estimating the transformation requires only three parameters and for a large enough sample, losing three degrees of freedom should not matter for all practical purposes.

\section{Simulations}
\label{sec:simulations}
This section explores finite sample properties of estimators for $\theta = (\mu_x, \sigma_x, \delta)$ and $(\mu_y, \sigma_y)$ under Gaussian input $X \sim \mathcal{N}(\mu_x, \sigma_x^2)$. In particular, it compares Gaussian \textsc{MLE} (estimation of $\mu_y$ and $\sigma_y$ only), \textsc{IGMM} and Lambert W $\times$ Gaussian \textsc{MLE}, and - for a heavy tail competitor -- the median.\footnote{For IGMM, optimization was restricted to $\delta \in [0, 10]$.} All results below are based on $n = 1,000$ replications.

\subsection{Estimating $\delta$ Only}
\label{sec:estimate_delta_only}
Here I show finite sample properties of $\widehat{\delta}_{MLE}$ for $U \sim \mathcal{N}(0,1)$, where $\mu_x = 0$ and $\sigma_x = 1$ are known and fixed.  Theorem \ref{thm:MLE_delta} shows that $\widehat{\delta}_{MLE}$ is unique: either at the boundary $\delta = 0$ or at the globally optimal solution to \eqref{eq:derivative_equal_0}. Results in Table \ref{tab:MLE_sim} were obtained by numerical optimization restricted to $\delta \geq 0$ ($\Leftrightarrow \log \delta \in \R$) using the \texttt{nlm} function in R.

\begin{table}[!tbp]
 \begin{center}
\caption{\label{tab:MLE_sim} Finite sample properties of $\widehat{\delta}_{MLE}$. For each $N$, $\delta$ was estimated $n = 1,000$ times from a random sample $\mathbf{z} \sim$ Tukey's $h$. The left column for each $\delta$ shows bias, $\overline{\widehat{\delta}}_{MLE}  - \delta$; each right column shows the root mean square error (RMSE) times $\sqrt{N}$.
} 
 \fbox{
 \begin{tabular}{r|rr|c|rr|c|rr|c|rr}
\multicolumn{1}{l}{N}&
\multicolumn{2}{c}{$\delta = 0$}&
\multicolumn{1}{c}{\ }&
\multicolumn{2}{c}{$\delta = 1/10$}&
\multicolumn{1}{c}{\ }&
\multicolumn{2}{c}{$\delta = 1/3$}&
\multicolumn{1}{c}{\ }&
\multicolumn{2}{c}{$\delta = 1/2$}
\tabularnewline \cline{2-3} \cline{5-6} \cline{8-9} \cline{11-12}
\hline
 10&$0.025$&$0.191$&&$-0.017$&$0.394$&&$-0.042$&$0.915$&&$-0.082$&$1.167$\tabularnewline
 50&$0.013$&$0.187$&&$-0.010$&$0.492$&&$-0.018$&$0.931$&&$-0.016$&$1.156$\tabularnewline
 100&$0.010$&$0.200$&&$-0.010$&$0.513$&&$-0.009$&$0.914$&&$-0.006$&$1.225$\tabularnewline
 400&$0.005$&$0.186$&&$-0.003$&$0.528$&&$ 0.000$&$0.927$&&$-0.004$&$1.211$\tabularnewline
 1000&$0.003$&$0.197$&&$ 0.000$&$0.532$&&$-0.001$&$0.928$&&$-0.001$&$1.203$\tabularnewline
 2000&$0.003$&$0.217$&&$-0.001$&$0.523$&&$ 0.000$&$0.935$&&$-0.001$&$1.130$\tabularnewline
\hline
\multicolumn{1}{c}{N }&
\multicolumn{2}{c}{$\delta = 1$}&
\multicolumn{1}{c}{\ }&
\multicolumn{2}{c}{$\delta = 2$}&
\multicolumn{1}{c}{\ }&
\multicolumn{2}{c}{$\delta = 5$}&
\multicolumn{1}{c}{\ }&
\multicolumn{2}{c}{ }
\tabularnewline \cline{2-3} \cline{5-6} \cline{8-9} \cline{11-12}
\hline
 10&$-0.054$&$1.987$&&$-0.104$&$3.384$&&$-0.050$&$7.601$&&$$&$$\tabularnewline
 50&$-0.017$&$1.948$&&$-0.009$&$3.529$&&$ 0.014$&$7.942$&&$$&$$\tabularnewline
 100&$-0.014$&$2.024$&&$-0.001$&$3.294$&&$ 0.011$&$7.798$&&$$&$$\tabularnewline
 400&$ 0.001$&$1.919$&&$-0.002$&$3.433$&&$ 0.001$&$7.855$&&$$&$$\tabularnewline
 1000&$ 0.001$&$1.955$&&$ 0.001$&$3.553$&&$-0.001$&$7.409$&&$$&$$\tabularnewline
 2000&$ 0.001$&$1.896$&&$ 0.000$&$3.508$&&$-0.001$&$7.578$&&$$&$$\tabularnewline
\end{tabular}
}

\end{center}

\end{table}

Table \ref{tab:MLE_sim} shows that the MLE is unbiased for every $\delta$ and settles down (about $N = 100$) to an asymptotic variance, which is increasing with $\delta$. Assuming $\mu_x$ and $\sigma_x$ to be known is unrealistic and thus these finite sample properties are only an indication of the behavior of the joint MLE, $\widehat{\theta}_{MLE}$. Nevertheless they are very remarkable for extremely heavy-tailed data ($\delta > 1$), where standard statistical methods typically break down. One explanation in this behavior lies in the particular form of the likelihood \eqref{eq:loglik_delta_z} and its gradient \eqref{eq:derivative_equal_0} (Theorem \ref{thm:MLE_delta}). Although both depend on $\mathbf{z}$, they only do so through $W_{\delta}\left( \mathbf{z} \right) = \mathbf{u} \sim \mathcal{N}(0,1)$. Hence as long as $\widehat{\delta}_{MLE}$ is sufficiently close to the true $\delta$, \eqref{eq:loglik_delta_z} and \eqref{eq:derivative_equal_0} are functions of almost Gaussian RVs and standard asymptotic results should still apply.

\subsection{Estimating All Parameters Jointly}
Here we consider the realistic scenario where $\mu_x$ and $\sigma_x$ are also unknown. We consider various sample sizes ($N=50$, $100$, and $1000$) and different degrees of heavy tails, $\delta \in \lbrace 0, 1/3, 1, 1.5 \rbrace$, each one representing a particularly interesting situation: 
\begin{inparaenum}[i)]
\item Gaussian data (does additional - superfluous - estimation of $\delta$ affect other estimates?),
\item fourth moments do not exist anymore, 
\item non-existing mean,
\item extremely heavy-tailed data -- can we get useful estimates at all?
\end{inparaenum}

\begin{table}[!ht]
\centering
\caption{\label{tab:simulations} In each subtable: (first rows) average, (middle rows) proportion of estimates below truth, (bottom rows) empirical standard deviation \text{times $\sqrt{N}$}.}
\subfloat[Truly Gaussian data: $\delta = 0$ \label{tab:simulations_delta_0}]{
\tiny
\centering
\begin{tabular}{|c||r||rr||rrr|r||rrr|r||r|}\hline 
 \multicolumn{1}{|c||}{$\delta = 0$}&\multicolumn{1}{|c||}{median}&\multicolumn{2}{c||}{Gaussian MLE} &\multicolumn{4}{c||}{IGMM}& \multicolumn{4}{c||}{Lambert W MLE} & \multicolumn{1}{c|}{NA}\tabularnewline
 \hline
 \multicolumn{1}{|c||}{N}&\multicolumn{1}{c||}{$ $ }&\multicolumn{1}{c}{$\mu_y$}&\multicolumn{1}{c||}{$\sigma_y$}&\multicolumn{1}{c}{$\mu_x$}&\multicolumn{1}{c}{$\sigma_x$}&\multicolumn{1}{c|}{$\delta$}&\multicolumn{1}{c||}{$\sigma_y$}&\multicolumn{1}{c}{$\mu_x$}&\multicolumn{1}{c}{$\sigma_x$}&\multicolumn{1}{c|}{$\delta$}&\multicolumn{1}{c||}{$\sigma_y$}&\multicolumn{1}{c|}{ratio}\tabularnewline

\hline
$  50$&$0.00$&$0.00$&$0.98$&$0.00$&$0.97$&$0.02$&$0.99$&$0.00$&$0.96$&$0.02$&$0.98$&$0$\tabularnewline
$ 100$&$0.00$&$0.00$&$0.99$&$0.00$&$0.98$&$0.01$&$1.00$&$0.00$&$0.97$&$0.01$&$0.99$&$0$\tabularnewline
$1000$&$0.00$&$0.00$&$1.00$&$0.00$&$0.99$&$0.00$&$1.00$&$0.00$&$0.99$&$0.00$&$1.00$&$0$\tabularnewline
\hline
$  50$&$0.50$&$0.50$&$0.57$&$0.51$&$0.60$&$0.66$&$0.54$&$0.51$&$0.65$&$0.66$&$0.56$&$0$\tabularnewline
$ 100$&$0.50$&$0.51$&$0.56$&$0.51$&$0.62$&$0.62$&$0.53$&$0.52$&$0.65$&$0.62$&$0.56$&$0$\tabularnewline
$1000$&$0.50$&$0.49$&$0.52$&$0.49$&$0.62$&$0.56$&$0.52$&$0.49$&$0.63$&$0.56$&$0.52$&$0$\tabularnewline
\hline
$  50$&$1.24$&$1.01$&$0.72$&$1.01$&$0.76$&$0.21$&$0.73$&$1.02$&$0.78$&$0.26$&$0.72$&$0$\tabularnewline
$ 100$&$1.25$&$1.02$&$0.70$&$1.02$&$0.76$&$0.23$&$0.70$&$1.03$&$0.78$&$0.26$&$0.70$&$0$\tabularnewline
$1000$&$1.26$&$0.98$&$0.73$&$0.98$&$0.79$&$0.22$&$0.73$&$0.98$&$0.79$&$0.22$&$0.73$&$0$\tabularnewline
\hline
\end{tabular}
}\\

\subfloat[No fourth moments: $\delta = 1/3$ \label{tab:simulations_delta_1_3}]{
\tiny
\centering
\begin{tabular}{|c||r||rr||rrr|r||rrr|r||r|}\hline 
 \multicolumn{1}{|c||}{$\delta = 1/3$}&\multicolumn{1}{|c||}{median}&\multicolumn{2}{c||}{Gaussian MLE} &\multicolumn{4}{c||}{IGMM}& \multicolumn{4}{c||}{Lambert W MLE} & \multicolumn{1}{c|}{NA}\tabularnewline
 \hline
 \multicolumn{1}{|c||}{N}&\multicolumn{1}{c||}{$ $ }&\multicolumn{1}{c}{$\mu_y$}&\multicolumn{1}{c||}{$\sigma_y$}&\multicolumn{1}{c}{$\mu_x$}&\multicolumn{1}{c}{$\sigma_x$}&\multicolumn{1}{c|}{$\delta$}&\multicolumn{1}{c||}{$\sigma_y$}&\multicolumn{1}{c}{$\mu_x$}&\multicolumn{1}{c}{$\sigma_x$}&\multicolumn{1}{c|}{$\delta$}&\multicolumn{1}{c||}{$\sigma_y$}&\multicolumn{1}{c|}{ratio}\tabularnewline
\hline
$  50$&$0.00$&$0.00$&$ 1.98$&$0.00$&$1.07$&$0.29$&$   \infty$&$0.00$&$1.01$&$0.33$&$  \infty$&$0$\tabularnewline
$ 100$&$0.00$&$0.00$&$ 2.03$&$0.00$&$1.04$&$0.31$&$   \infty$&$0.00$&$1.00$&$0.33$&$  \infty$&$0$\tabularnewline
$1000$&$0.00$&$0.00$&$ 2.18$&$0.00$&$1.00$&$0.33$&$ 2.34$&$0.00$&$1.00$&$0.33$&$ 2.34$&$0$\tabularnewline
\hline
$  50$&$0.50$&$0.51$&$ 0.78$&$0.50$&$0.38$&$0.63$&$ 0.60$&$0.50$&$0.52$&$0.54$&$ 0.54$&$0$\tabularnewline
$ 100$&$0.50$&$0.51$&$ 0.78$&$0.51$&$0.42$&$0.61$&$ 0.60$&$0.50$&$0.51$&$0.54$&$ 0.54$&$0$\tabularnewline
$1000$&$0.48$&$0.51$&$ 0.77$&$0.51$&$0.47$&$0.56$&$ 0.55$&$0.51$&$0.50$&$0.53$&$ 0.52$&$0$\tabularnewline
\hline
$  50$&$1.27$&$2.21$&$ 6.56$&$1.44$&$1.45$&$1.10$&$NA$&$1.23$&$1.35$&$1.14$&$NA$&$0$\tabularnewline
$ 100$&$1.30$&$2.33$&$11.28$&$1.43$&$1.42$&$1.12$&$NA$&$1.19$&$1.34$&$1.09$&$NA$&$0$\tabularnewline
$1000$&$1.23$&$2.25$&$16.76$&$1.39$&$1.45$&$1.20$&$15.97$&$1.17$&$1.33$&$1.08$&$12.30$&$0$\tabularnewline
\hline

\end{tabular}
}\\

\subfloat[Non-existing mean: $\delta = 1$ \label{tab:simulations_delta_1}]{
\tiny
\centering
\begin{tabular}{|c||r||rr||rrr|r||rrr|r||r|}\hline 
 \multicolumn{1}{|c||}{$\delta = 1$}&\multicolumn{1}{|c||}{median}&\multicolumn{2}{c||}{Gaussian MLE} &\multicolumn{4}{c||}{IGMM}& \multicolumn{4}{c||}{Lambert W MLE} & \multicolumn{1}{c|}{NA}\tabularnewline
 \hline
 \multicolumn{1}{|c||}{N}&\multicolumn{1}{c||}{$ $ }&\multicolumn{1}{c}{$\mu_y$}&\multicolumn{1}{c||}{$\sigma_y$}&\multicolumn{1}{c}{$\mu_x$}&\multicolumn{1}{c}{$\sigma_x$}&\multicolumn{1}{c|}{$\delta$}&\multicolumn{1}{c||}{$\sigma_y$}&\multicolumn{1}{c}{$\mu_x$}&\multicolumn{1}{c}{$\sigma_x$}&\multicolumn{1}{c|}{$\delta$}&\multicolumn{1}{c||}{$\sigma_y$}&\multicolumn{1}{c|}{ratio}\tabularnewline
\hline
50&$0.00$&$  -0.10$&$    24.6$&$-0.01$&$1.18$&$0.90$&$\infty$&$0.00$&$1.01$&$0.99$&$\infty$&$0$\tabularnewline
100&$0.00$&$   0.74$&$    72.4$&$ 0.00$&$1.09$&$0.95$&$\infty$&$0.00$&$1.01$&$0.99$&$\infty$&$0$\tabularnewline
1000&$0.00$&$   3.84$&$   348.1$&$ 0.00$&$1.01$&$1.00$&$\infty$&$0.00$&$1.00$&$1.00$&$\infty$&$0$\tabularnewline
\hline
50&$0.53$&$   0.52$&$     1.0$&$ 0.51$&$0.34$&$0.65$&$  1$&$0.51$&$0.52$&$0.52$&$  1$&$0$\tabularnewline
100&$0.50$&$   0.52$&$     1.0$&$ 0.51$&$0.38$&$0.63$&$  1$&$0.50$&$0.53$&$0.53$&$  1$&$0$\tabularnewline
1000&$0.49$&$   0.52$&$     1.0$&$ 0.51$&$0.48$&$0.53$&$  1$&$0.49$&$0.51$&$0.51$&$  1$&$0$\tabularnewline
\hline
50&$1.27$&$  65.85$&$   424.3$&$ 2.10$&$2.50$&$2.32$&$NA$&$1.19$&$1.70$&$2.16$&$NA$&$0$\tabularnewline
100&$1.30$&$ 410.75$&$  4050.2$&$ 2.01$&$2.28$&$2.59$&$NA$&$1.17$&$1.74$&$2.25$&$NA$&$0$\tabularnewline
1000&$1.26$&$3307.58$&$104052.7$&$ 1.93$&$2.21$&$2.81$&$NA$&$1.11$&$1.64$&$2.18$&$NA$&$0$\tabularnewline
\hline

\end{tabular}
}\\
\subfloat[Extreme heavy tails: $\delta = 1.5$ \label{tab:simulations_delta_1_5}]{
\tiny
\centering
\begin{tabular}{|c||r||rr||rrr|r||rrr|r||r|}\hline 
 \multicolumn{1}{|c||}{$\delta = 1.5$}&\multicolumn{1}{|c||}{median}&\multicolumn{2}{c||}{Gaussian MLE} &\multicolumn{4}{c||}{IGMM}& \multicolumn{4}{c||}{Lambert W MLE} & \multicolumn{1}{c|}{NA}\tabularnewline
 \hline
 \multicolumn{1}{|c||}{N}&\multicolumn{1}{c||}{$ $ }&\multicolumn{1}{c}{$\mu_y$}&\multicolumn{1}{c||}{$\sigma_y$}&\multicolumn{1}{c}{$\mu_x$}&\multicolumn{1}{c}{$\sigma_x$}&\multicolumn{1}{c|}{$\delta$}&\multicolumn{1}{c||}{$\sigma_y$}&\multicolumn{1}{c}{$\mu_x$}&\multicolumn{1}{c}{$\sigma_x$}&\multicolumn{1}{c|}{$\delta$}&\multicolumn{1}{c||}{$\sigma_y$}&\multicolumn{1}{c|}{ratio}\tabularnewline
\hline
50&$-0.02$&$     6.84$&$    309$&$-0.02$&$1.23$&$1.37$&$ \infty $&$-0.01$&$1.00$&$1.49$&$\infty$&$0.01$\tabularnewline
100&$ 0.00$&$   -51.16$&$   3080$&$-0.01$&$1.12$&$1.44$&$ \infty $&$ 0.00$&$1.01$&$1.50$&$\infty$&$0.00$\tabularnewline
1000&$ 0.00$&$   176.13$&$  14251$&$ 0.00$&$1.01$&$1.49$&$ \infty $&$ 0.00$&$1.00$&$1.50$&$\infty$&$0.00$\tabularnewline
\hline
50&$ 0.53$&$     0.48$&$      1$&$ 0.51$&$0.34$&$0.64$&$  1$&$ 0.53$&$0.53$&$0.54$&$  1$&$0.01$\tabularnewline
100&$ 0.51$&$     0.53$&$      1$&$ 0.54$&$0.37$&$0.61$&$  1$&$ 0.52$&$0.51$&$0.51$&$  1$&$0.00$\tabularnewline
1000&$ 0.50$&$     0.50$&$      1$&$ 0.50$&$0.47$&$0.54$&$  1$&$ 0.49$&$0.53$&$0.52$&$  1$&$0.00$\tabularnewline
\hline
50&$ 1.32$&$  1347.71$&$   9261$&$ 2.57$&$3.20$&$3.12$&$ NA $&$ 1.15$&$1.86$&$2.76$&$NA$&$0.01$\tabularnewline
100&$ 1.33$&$ 42156.28$&$ 418435$&$ 2.39$&$2.87$&$3.44$&$ NA $&$ 1.12$&$1.78$&$2.84$&$NA$&$0.00$\tabularnewline
1000&$ 1.26$&$124462.82$&$3903629$&$ 2.18$&$2.66$&$3.67$&$ NA $&$ 1.11$&$1.80$&$2.85$&$NA$&$0.00$\tabularnewline
\hline

\end{tabular}
}
\end{table}

The convergence tolerance for IGMM was set to $tol = 1.22 \cdot 10^{-4}$. Table \ref{tab:simulations} summarizes the simulation. Each sub-table is organized as follows: columns represent parameter estimates; the three main rows are the average over $n = 1,000$ replications (top), the proportion of estimates below the true value (middle), and the empirical standard deviation around the empirical average times $\sqrt{N}$ -- not around the truth (bottom).

The Gaussian MLE estimates $\sigma_y$ directly, while IGMM and the Lambert W $\times$ Gaussian MLE estimates $\delta$ and $\sigma_x$, which implicitly give $\widehat{\sigma}_y$ through $\sigma_y(\delta, \sigma_x) = \sigma_x \cdot \frac{1}{\sqrt{(1 - 2 \delta)^{3/2}}}$ if $\delta < 1/2$ (see \eqref{eq:Y_moments}). For a fair comparison each sub-table also includes a column for $\widehat{\sigma}_y = \widehat{\sigma}_x \cdot \frac{1}{\sqrt{(1 - 2 \widehat{\delta})^{3/2}}}$. Some of these entries contain ``$\infty$'', even for $\delta < 1/2$;  this occurs if at least one $\widehat{\delta} \geq 1/2$. 

For any $\delta < 1$, $\mu_x = \mu_y$, thus they can be directly compared. For $\delta \geq 1$, the mean does not exist; each sub-table for these $\delta$ interprets $\mu_y$ as the median.
 
\paragraph{Gaussian data: $\delta = 0$}
This setting checks if imposing the Lambert W framework, even though its use is superfluous, causes a quality loss in the estimation of $\mu_y = \mu_x$ or $\sigma_y = \sigma_x$. Furthermore, critical values for $H_0: \delta = 0$ (Gaussian tails) can be obtained. Table \ref{tab:simulations_delta_0} shows that all estimators are unbiased and quickly tend to a large-sample variance.  Additional estimation of $\delta$ does not affect the efficiency of $\widehat{\mu}_x$ compared to estimating solely $\mu$ (both for IGMM and Lambert W $\times$ Gaussian MLE). Estimating $\sigma_y$ directly by Gaussian MLE does not give better results than the Lambert W $\times$ Gaussian MLE: both are unbiased and have similar standard deviation. 

\paragraph{No fourth moment: $\delta = 1/3$}
Here $\sigma_y(\delta, \sigma_x = 1) = 2.28$, but fourth moments do not exist anymore.  This results in an increasing empirical standard deviation of $\widehat{\sigma}_{y}$ as $N$ grows. In contrast, estimates for $\sigma_x$ are not drifting off. In presence of these large heavy tails the median is much less variable than Gaussian MLE and IGMM. Yet, Lambert W $\times$ Gaussian MLE for $\mu_x$ even outperforms the median.

\paragraph{Non-existing mean: $\delta = 1$} Here the mean is non-finite. Thus both sample moments diverge, and their standard errors are also growing quickly. The median still provides a very good estimate for the location, but is again inferior to both Lambert W estimators, which are unbiased and seem to converge to an asymptotic variance at rate $\sqrt{N}$.

\paragraph{Extreme heavy tails: $\delta = 1.5$} As in Section \ref{sec:estimate_delta_only}, IGMM and Lambert W MLE continue to be unbiased even though the data is extremely heavy-tailed. Moreover, Lambert W MLE also has the smallest empirical standard deviation overall. In particular, the Lambert W MLE for $\mu_x$ has an approximately $20\%$ lower standard deviation than the median. 

The last column shows that for some $N$ about $1\%$ of the $n = 1,000$ simulations generated invalid likelihood values (NA and $\infty$). Here the search for the optimal $\delta$ lead into regions with a numerical overflow in the evaluation of $W_{\delta}(z)$. For a comparable summary, these few cases were omitted and new simulations added until a full $n = 1,000$ finite estimates were found. Since this only happened in $1\%$ of the cases and also such heavy-tailed data is rarely encountered in practice, this numerical issue is not a real limitation in statistical practice.

\subsection{Discussion of the Simulations}
This simulation study confirms well-known facts about the sample average, standard deviation, and median and compares them to finite sample properties of the two Lambert W estimators. The median is known to be robust, which shows here as its quality does not depend on the thickness of the tails.

IGMM is unbiased for $\tau$ independent of the magnitude of $\delta$. As expected the Lambert W MLE for $\theta$ has the best properties: it is unbiased for all $\delta$, and for $\delta = 0$ it performs as well as the classic sample mean and standard deviation. For small $\delta$ it has the same empirical standard deviation as the Gaussian MLE, but a lower one than the median for large $\delta$.

Hence the only advantage of estimating $\mu_y$ and $\sigma_y$ by sample moments of $\mathbf{y}$ is speed; otherwise the Lambert W $\times$ Gaussian MLE is at least as good as the Gaussian MLE and clearly outperforms it for heavy-tailed data.

\section{Applications}
\label{sec:applications}
Tukey's $h$ distribution has already proven useful to model heavy-tailed data, but parametric inference was limited to quantile fitting or methods of moments estimation \citep{Todd08_Parametric_pdf_cdf_tukey, Fischer10_financial_teletraffic, Field04}. Theorem \ref{theorem:cdf_Y} allows us to estimate $\theta$ by ML. 

This section shows the usefulness of the presented methodology on simulated as well as real world data: 
\begin{inparaenum}[i)]
\item Section \ref{sec:sample_mean_Cauchy} demonstrates Gaussianizing on the Cauchy sample from the Introduction;
\item Section \ref{sec:SP500} shows that heavy tail Lambert W $\times$ Gaussian distributions provide an excellent fit to daily \textsc{S\&P 500} log-return series; and
\item Section \ref{sec:solar_flares} shows how removing heavy tails reveals hidden patterns in power-law type data.
\end{inparaenum}

\subsection{Estimating Location of a Cauchy With The Sample Mean}
\label{sec:sample_mean_Cauchy}
It is well-known that the sample mean $\overline{\mathbf{y}}$ is a poor estimate of the location parameter of a Cauchy distribution, since the sampling distribution of  $\overline{\mathbf{y}}$ is again a Cauchy; in particular, its variance does not go to $0$ for $n \rightarrow \infty$.

Heavy-tailed Lambert W $\times$ Gaussian distributions have similar properties to a Cauchy for $\delta \approx 1$. The mean of $X$ equals the location of $Y$, due to symmetry around $\mu_x$ (for all $\delta \geq 0$) and $c$, respectively. Thus we can estimate $\tau$ from the Cauchy sample $\mathbf{y}$, transform $\mathbf{y}$ to $\mathbf{x}_{\widehat{\tau}}$, estimate $\mu_x$ from $\mathbf{x}_{\widehat{\tau}} = W_{\widehat{\tau}}(\mathbf{y})$, and thus obtain an estimate of $c$.\\

The data $\mathbf{y} \sim \mathcal{C}(0,1)$ in Fig.\ \ref{fig:Cauchy_sample} has heavy tails with two extreme (positive) samples. A Cauchy ML fit gives $\widehat{c} = 0.03 (0.055)$ and $\widehat{s} = 0.86 (0.053)$ (standard errors in parenthesis). A Lambert W $\times$ Gaussian MLE gives $\widehat{\mu}_x = 0.03 (0.055)$, $\widehat{\sigma}_x = 1.05 (0.072)$, and $\widehat{\delta} = 0.86 (0.082)$. Thus both fits correctly fail to reject $\mu_x = c = 0$. Table \ref{tab:summary_cauchy} shows summary statistics on both samples. Since the Cauchy distribution does not have a well-defined mean, $\overline{\mathbf{y}} = 2.304 (2.101)$ is not meaningful. However, $\mathbf{x}_{\widehat{\tau}_{MLE}}$ is approximately Gaussian and we use the sample average to do inference: $\overline{\mathbf{x}}_{\widehat{\tau}_{MLE}} = 0.033 (0.0472)$ correctly fails to reject a zero location for $\mathbf{y}$. The transformed $\mathbf{x}_{\widehat{\tau}_{MLE}}$ features additional Gaussian characteristics (symmetric, no excess kurtosis), and even the null hypothesis of Normality cannot be rejected (p-value $\geq 0.5$).\\

Figure \ref{fig:Cauchy_Gaussian_avg} shows the running sample average for the original sample and its Gaussianized version. For a fair comparison $\widehat{\tau}_{MLE}^{(n)}$ was re-estimated cumulatively for each $n = 5, \ldots, 500$, and then used to compute $(x_{1},\ldots, x_n)$. Even for small $n$ the transformation works extremely well: the highly influential point around $n \approx 50$ greatly affects $\overline{\mathbf{y}}$, but has no relevant effect on $\overline{\mathbf{x}}_{\widehat{\tau}_{MLE}^{(n)}}$. Overall, the sample average of the Gaussianized data has the usual good properties.  And even for very small $n$ it is already clear that the location of the underlying Cauchy distribution is approximately zero.\\

Although a toy example, it shows that removing (strong) heavy tails from data works and provides new, ``nice'' data which can then be used for more refined methods.

\begin{table}[!t]
\centering
\caption{\label{tab:summary_stats_both} Summary statistics for observed (heavy-tailed) $\mathbf{y}$ and back-transformed (Gaussianized) data $\mathbf{x}_{\widehat{\tau}_{MLE}}$. $**$ stands for $< 10^{-16}$; $*$  for $< 2.2 \cdot 10^{-16}$.}

\setcounter{subtable}{-1}
\subfloat 
{
\centering
\small
\fbox{
\begin{tabular}{r}
  \\ 
  \hline
  Min  \\ 
  Max  \\ 
  Mean  \\ 
  Median  \\ 
  Stdev  \\ 
  Skewness  \\ 
  Kurtosis \\ 
  \hline
  SW  \\
  AD \\
\end{tabular}
}
  }
\hspace{-0.4cm}
\setcounter{subtable}{0}
\subfloat[\label{tab:summary_cauchy} $\mathbf{y} \sim \mathcal{C}(0,1)$ \newline (Section \ref{sec:sample_mean_Cauchy})]{
\centering
\small
\fbox{
\begin{tabular}{rrr}
 $\mathbf{y}$ & $\mathbf{x}_{\widehat{\tau}}$ & $\mathbf{x}_{\widehat{\lambda}}$ \\ 
  \hline
 -161.59 & -3.16 & 0 \\ 
 952.95 & 3.81 & 33.18 \\ 
 2.30 & 0.03 & 14.98 \\ 
 0.04 & 0.04 & 14.96 \\ 
 46.980 & 1.06 & 1.20 \\ 
 17.43 & 0.12 & 3.90\\ 
 343.34 & 3.21 & 161.75 \\ 
  \hline
 $*$ & 0.71 & $**$\\
 $**$ & 0.51 & $**$ \\
\end{tabular}
}
  }
\hspace{-0.5cm}
  \subfloat[\label{tab:summary_stats_SP500} $\mathbf{y}$ = \textsc{S\&P 500} \newline (Section \ref{sec:SP500}) ]{
    \centering
    \small
    \fbox{
\begin{tabular}{rr}
 $\mathbf{y}$ & $\mathbf{x}_{\widehat{\tau}}$ \\ 
  \hline
 -7.11 & -2.42 \\ 
4.99 & 2.23 \\ 
 0.05 & 0.05 \\ 
 0.04 & 0.04 \\ 
 0.95 & 0.71 \\ 
 -0.30 & -0.04 \\ 
 7.70 & 2.93 \\
  \hline
 $*$ & 0.24 \\
 $*$ & 0.18 \\
\end{tabular}
}
  }
\hspace{-0.5cm}
\subfloat[\label{tab:summary_stats_solarflares} $\mathbf{y}$ = solar flares \newline (Section \ref{sec:solar_flares}) ]{
\centering
\small
\fbox{
\begin{tabular}{rr}
 $\mathbf{y}$ & $\mathbf{x}_{\widehat{\tau}}$ \\ 
  \hline
 20 & 20 \\ 
 231300 & 157 \\ 
 689.4 & 89.0 \\ 
 87 & 87 \\ 
 6520.6 & 27.0 \\ 
 22.2 & 0.1 \\ 
 582.1 & 1.9 \\ 
  \hline
 $**$ & $**$ \\
 $**$ & $**$ \\
\end{tabular}
}
}
\end{table}

\subsection{Heavy Tails in Finance: S\&P 500 Case Study}
\label{sec:SP500}
A lot of financial data displays negative skewness and excess kurtosis. Since financial data is in general not i.i.d., it is often modeled with a (skew) student-t distribution underlying a (generalized) auto-regressive conditional heteroskedastic (GARCH) \citep{Engle82, Bollerslev86} or a stochastic volatility (SV) model \citep{MelinoTurnbull90, DeoHurvichYu06}. Using the Lambert W approach we can build upon the knowledge and implications of Gaussianity (and avoid deriving properties of a GARCH or SV model with heavy-tailed innovations), and simply ``Gaussianize'' the reutns before fitting more complex -- GARCH or SV -- models.  

\begin{remark}
Time series models with Lambert W $\times$ Gaussian white noise are far beyond the scope of this work, but can be a direction of future research. Here I only consider the unconditional distribution.
\end{remark}

Figure \ref{fig:normfit_SP500} shows the \textsc{S\&P 500} log-returns with a total of $N = 2,780$ daily observations.\footnote{R package \texttt{MASS}, dataset \texttt{SP500}.} Table \ref{tab:summary_stats_SP500} confirms the heavy tails (sample kurtosis $7.70$), but also indicates negative skewness ($-0.296$). As the sample skewness $\widehat{\gamma}_1(\mathbf{y})$ is very sensitive to outliers, we should test for symmetry by fitting a skewed distribution and testing its skewness parameter(s) for zero. In case of the double-tail Lambert W $\times$ Gaussian this means to test $H_0: \delta_{\ell} = \delta_r = \delta$ versus $H_1: \delta_{\ell} \neq \delta_r$. Since the likelihood can now be computed by \eqref{eq:loglikelihood_Y_with_xi}, we can use a likelihood ratio test with one degree of freedom ($3$ versus $4$ parameters). The log-likelihood of the double-tail Lambert W $\times$ Gaussian fit (Table \ref{tab:SP500_LambertW_dd}) equals $-3606.0 = -2972.27 + (-633.73)$ (input + penalty), while the one $\delta$ fit gives $-3606.56 = -2971.47 + (-635.09)$.  Here the double tails pay a lower penalty for transforming the data, but in turn give less Gaussian transformed sample.  Comparing twice their difference to a $\chi^2_1$ distribution gives a p-value of $0.29$.  For comparison, a skew-t fit \citep{AzzaliniCapitanio03}, with location $c$, scale $s$, shape $\alpha$, and $\nu$ degrees of freedom  also yields \footnote{Function \texttt{st.mle} in the R package \href{cran.r-project.org/web/packages/sn}{\texttt{sn}}.} a non-significant $\widehat{\alpha}$ (Table \ref{tab:SP500_skew_t}). Thus both fits cannot reject symmetry.\\

\begin{figure}
\centering
\subfloat[Observed heavy tail returns $\mathbf{y}$]{\includegraphics[width=.45\textwidth]{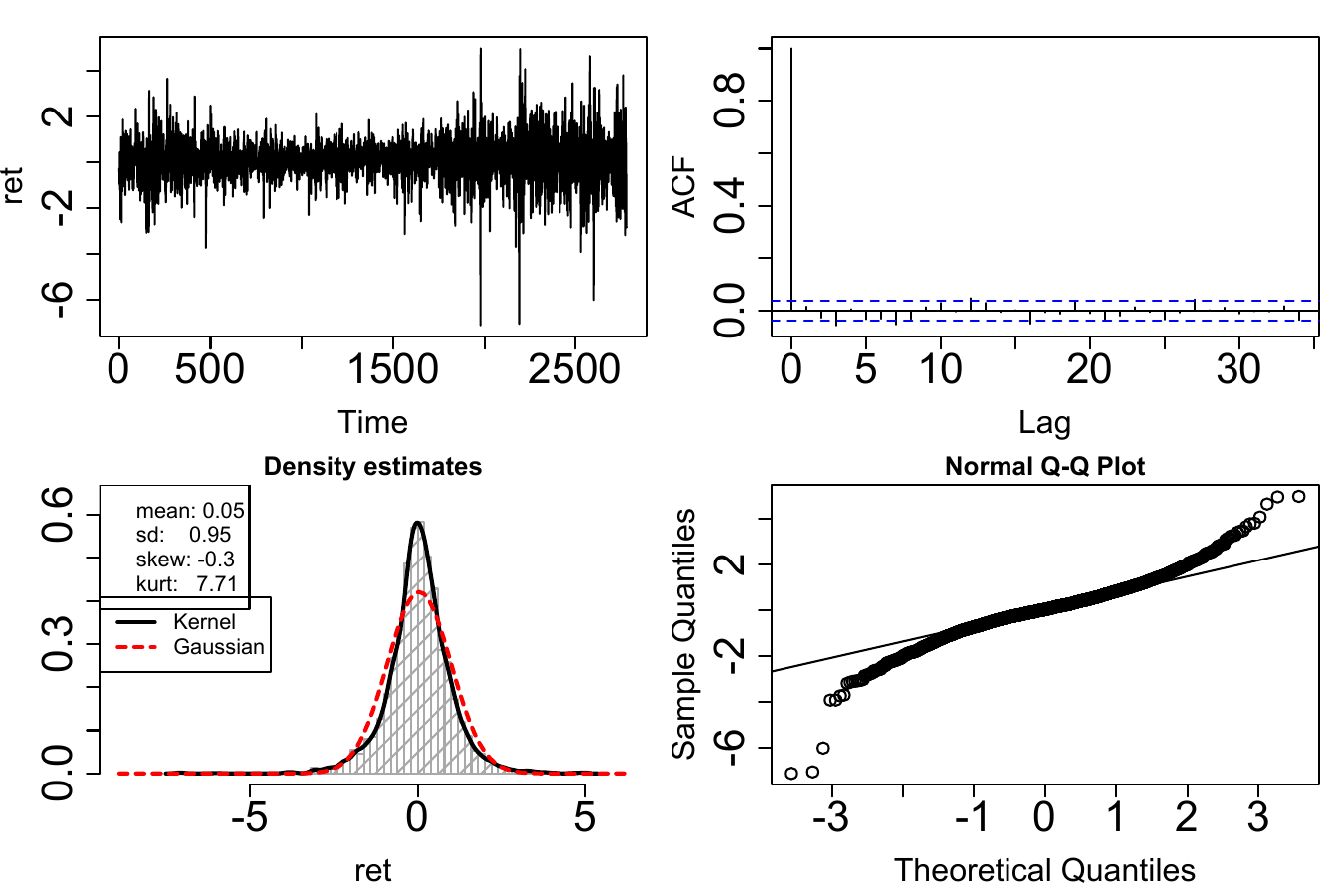}\label{fig:normfit_SP500}}
\hspace{0.05\textwidth}
\subfloat[Gaussianized returns $\mathbf{x}_{\widehat{\tau}_{MLE}}$ ]{\includegraphics[width=.45\textwidth]{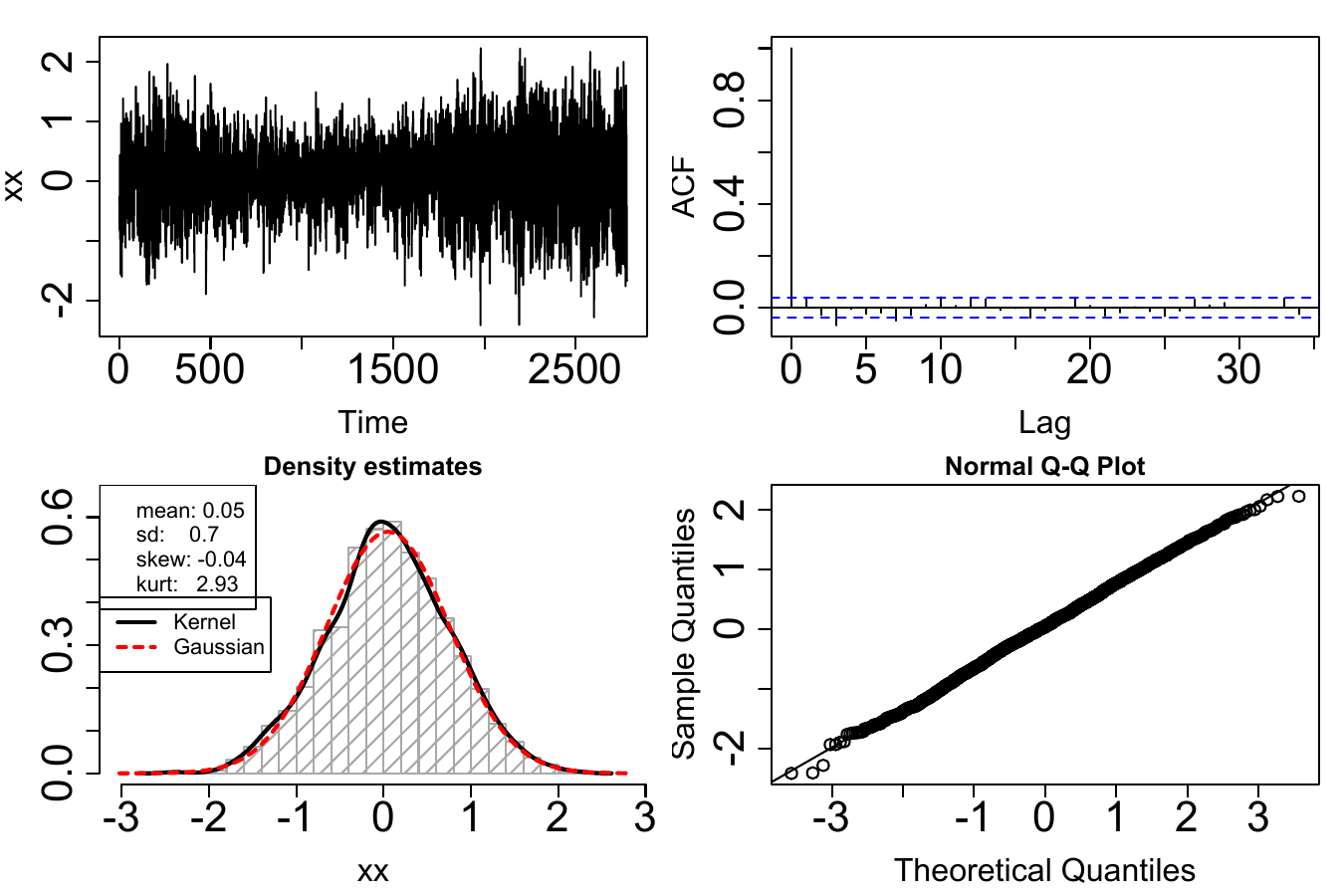}\label{fig:normfit_SP500_input}}
\caption{\label{fig:input_output} Lambert W Gaussianization of S\&P 500 log-returns: $\widehat{\tau} = (0.05, 0.70,   0.17)$. In (a) and (b): data (top left); autocorrelation function (ACF) (top right); histogram, Gaussian fit, and KDE (bottom left); Normal QQ plot (bottom right).}
\end{figure} 

Assume we have to make a decision if we should trade a certificate replicating the \textsc{S\&P 500}. Since we can either buy or sell, it is not important if the average return is positive or negative, as long as it is significantly different from zero.

\subsubsection{Gaussian Fit to Returns}
If we ignore heavy tails and estimate $(\mu_y, \sigma_y)$ by Gaussian MLE, $\widehat{\mu}_y = 0$ can not be rejected on a $\alpha = 1\%$ level (Table \ref{tab:SP500_Gaussian}). However, a plain sample average over-estimates the variance in presence of heavy tails, and thus adds bias to the test statistic.
\subsubsection{Heavy Tail Fit to Returns}
Both a heavy tail Lambert W $\times$ Gaussian (Table \ref{tab:SP500_LambertW}) and student-$t$ fit (Table \ref{tab:SP500_t}) reject the zero mean null (p-values, $10^{-4}$ and $3 \cdot 10^{-5}$, respectively). The standard errors for the location parameter are essentially the same.\\

While location and scale estimates are almost identical, the tail estimates lead to very different conclusions: while for $\widehat{\nu} = 3.71$ only moments up to order $3$ exist, in the Lambert W $\times$ Gaussian case moments up to order $5$ exist ($1/0.172 = 5.81$). This is especially noteworthy as many theoretical results in the (financial) time series literature rely on finite fourth moments \citep{Zadrozny05, Mantegna98}; consequently many empirical studies test if financial data actually satisfy this assumption \citep{Cont01_Empiricalproperties, Huismanetal01_TailIndexSmallSamples}. For this particular dataset student's $t$ and a Lambert W $\times$ Gaussian fit give different answers to the same question. Since previous empirical studies often use student's $t$ as a baseline \citep{WongChanKam09_tmixtureAR}, it might be worthwhile to re-examine their findings in light of heavy tail Lambert W $\times$ Gaussian distributions.

\begin{table}[!t]
  \centering
\caption{\label{tab:MLE_for_SP500} MLE fits to \textsc{S\&P 500} $\mathbf{y}$ (a, b, c, d, e) and the Gaussianized data  $\mathbf{x}_{\widehat{\tau}_{MLE}}$ (f).}
\subfloat[\label{tab:SP500_LambertW_dd} double-tail Lambert W $\times$ Gaussian = Tukey's $hh$ (\textsc{S\&P 500})]{
      \small
      \centering
\begin{tabular}{ccccc}
 &  Est. &  se &  t & Pr($> \mid t \mid $) \\
  \hline
$\mu_x$ 		& 0.06 	& 0.015 	& 3.66 & 0.00 \\ 
$\sigma_x$ 		& 0.71 	& 0.016 	& 44.00 & 0.00 \\ 
  $\delta_{\ell}$ & 0.19 & 0.021 	& 8.99 & 0.00 \\ 
  $\delta_r$ 	& 0.16 	& 0.019 	& 8.24 & 0.00 \\ 
   \hline
\end{tabular}
}
\subfloat[\label{tab:SP500_skew_t} skew $t$ (\textsc{S\&P 500})]{
      \small
      \centering
\begin{tabular}{ccccc}
 &  Est. &  se &  t & Pr($> \mid t \mid $) \\
  \hline
$c$ 	& 0.10 	& 0.061 	& 1.65 	& 0.10 \\ 
$s$ 	& 0.67 	& 0.017 	& 38.47 & 0.00 \\ 
$\alpha$ & -0.08 & 0.101 	& -0.77 & 0.44 \\ 
$\nu$ 	& 3.73 	& 0.297 	& 12.57 & 0.00 \\ 
   \hline
\end{tabular}
}\\

  \subfloat[\label{tab:SP500_LambertW} Lambert W $\times$ Gaussian = Tukey's $h$ (\textsc{S\&P 500})]{
    \small
    \centering
\begin{tabular}{ccccc}
 &  Est. &  se &  t & Pr($> \mid t \mid $) \\
  \hline
$\mu_x$ 	& 0.06 & 0.015 	& 3.65 & 0.000 \\
  $\sigma_x$ & 0.71 & 0.016 	& 43.95 & 0.000 \\ 
  $\delta$ 	& 0.17 & 0.016 		& 11.05 & 0.000 \\
 \end{tabular}
  }
  \subfloat[\label{tab:SP500_t} student-t (\textsc{S\&P 500})]{
    \small
    \centering
    \begin{tabular}{ccccc}
 &  Est. &  se &  t & Pr($> \mid t \mid $) \\
  \hline
$c$ & 0.06 & 0.015 & 3.65 & 0.00 \\ 
  $s$ & 0.67 & 0.017 & 39.51 & 0.00 \\
  $\nu$ & 3.72 & 0.295 & 12.61 & 0.00 \\ 
 \end{tabular}
  }\\

  \subfloat[\label{tab:SP500_Gaussian} Gaussian (\textsc{S\&P 500})]{
    \small
    \centering
   \begin{tabular}{ccccc}
 &  Est. &  se &  t & Pr($> \mid t \mid $) \\
  \hline
$\mu_y$  	& 0.05 	& 0.018 	& 2.55 & 0.01 \\ 
  $\sigma_y$ & 0.95 & 0.013 	& 74.57 & 0.00 \\ 
 \end{tabular}
  }
  \subfloat[\label{tab:SP500_input_MLE} Gaussian ($\mathbf{x}_{\widehat{\tau}_{MLE}}$)]{
    \small
    \centering
\begin{tabular}{ccccc}
 &  Est. &  se &  t & Pr($> \mid t \mid $) \\
  \hline
$\mu_{x_{\widehat{\tau}}} $		&0.05 	& 0.013 & 3.81 & 0.00 \\  
$\sigma_{x_{\widehat{\tau}}}$ 	& 0.71 & 0.009 & 74.57 & 0.00 \\
 \end{tabular}
  }
\end{table}

\subsubsection{``Gaussianizing'' Returns}
A typical parameter inference study would conclude here. Using Lambert's W function we can analyze the back-transformed $\mathbf{x}_{\widehat{\tau}_{MLE}}$ to test if a Lambert W $\times$ Gaussian distribution is indeed appropriate. Figure \ref{fig:normfit_SP500_input} shows that $\mathbf{x}_{\widehat{\tau}_{MLE}}$ is indistinguishable from a Gaussian sample. Not even one Normality test can reject Gaussianity: p-values are $0.18$, $0.18$, $0.31$, and $0.24$, respectively (Anderson Darling, Cramer-von-Mises, Shapiro-Francia, Shapiro-Wilk; see \citet{Thode02}). Table \ref{tab:summary_stats_SP500} also shows that Lambert W ``Gaussianiziation'' was successful: $\widehat{\gamma}_2(\mathbf{x}_{\widehat{\tau}}) = 2.93$ and $\widehat{\gamma}_2(\mathbf{x}_{\widehat{\tau}})= -0.039$ are within the typical variation for a Gaussian sample. Thus
\begin{align}
\label{eq:final_model_SP500}
Y = \left( U e^{\frac{0.172}{2} U^2} \right) 0.705 + 0.055, \quad U = \frac{X -  0.055}{0.705}, \quad U \sim \mathcal{N}(0,1)
\end{align}
is an adequate (unconditional) Lambert W $\times$ Gaussian model for the \textsc{S\&P 500} log-returns $\mathbf{y}$. For trading, this means that the expected return is significantly larger than zero ($\widehat{\mu}_x = 0.055 >0$), and thus replicating certificates should be bought.

\subsubsection{Gaussian MLE for Gaussianized Data}
For $\delta_l = \delta_r \equiv \delta < 1$, also $\mu_x \equiv \mu_y$. We can therefore replace testing $\mu_{y} = 0$ versus $\mu_{y} \neq 0$ for a non-Gaussian $\mathbf{y}$, with the very well understood hypothesis test $\mu_{x} = 0$ versus $\mu_{x} \neq 0$ for the Gaussian  $\mathbf{x}_{\widehat{\tau}_{MLE}}$. In particular, standard errors based on $\frac{\widehat{\sigma}}{\sqrt{N}}$ - and thus t and p-values - should be closer to the ``truth'' (Table  \ref{tab:SP500_LambertW} and \ref{tab:SP500_t}) than a Gaussian MLE on the non-Gaussian $\mathbf{y}$ (Table \ref{tab:SP500_Gaussian}). Table \ref{tab:SP500_input_MLE} shows that standard errors for $\widehat{\mu}_{\mathbf{x}}$ are even a bit too small compared to the heavy-tailed versions. Since the ``Gaussianizing'' transformation was estimated, treating $\mathbf{x}_{\widehat{\tau}_{MLE}}$ as if it was original data is too optimistic regarding its Gaussianity (recall the penalty \eqref{eq:R_penalty_all} in the total likelihood \eqref{eq:loglikelihood_Y_with_xi}).\\

This example confirms that if a model and its theoretical properties are based on Gaussianity, but the observed data is heavy-tailed, then Gaussianizing the data first gives more reliable inference than applying the Gaussian methods to the original, heavy-tailed data (Fig.\ \ref{fig:LambertW_heavy_flowchart}). Clearly, a joint estimation of the model parameters based on Lambert W $\times$ Gaussian errors (or any other heavy-tailed distribution) would be optimal. However, theoretical properties and estimation techniques may not have been developed and implemented yet, or are simply not known to researchers who are non-experts in heavy-tailed statistics. The Lambert Way to Gaussianize data thus is a pragmatic method to improve statistical inference on heavy-tailed data, while preserving the ease of usage and interpretation of Gaussian models.

\subsection{Removing Power Law From Solar Flare Counts}
\label{sec:solar_flares}
The previous section focused on Lambert W $\times$ $F_X$ distributions as a ``true'' model for the data $\mathbf{y}$. Here I consider it merely as a data transformation to remove heavy tails. In the same way as scaling $\mathbf{y}$ to zero-mean, unit-variance data, $(\mathbf{y} - \overline{\mathbf{y}}) / \widehat{\sigma}_y$, does not necessarily mean we believe the underlying process is Gaussian, we can also convert $\mathbf{y}$ to $\mathbf{x}_{\tau} = W_{\tau}(\mathbf{y})$ without assuming that $\mathbf{y}$ is actually Lambert W $\times$ Gaussian. While $\mathbf{x}_{\widehat{\tau}}$ might lose the interpretability of the observed data (e.g.\, units become distorted), it can be helpful for exploratory data analysis (EDA), as the eye is a bad judgment to detect regularities corrupted by heavy tails. Removing them can reveal hidden patterns and thus greatly improve the accuracy of statistical inference for $\mathbf{y}$.\\
\begin{figure}[!t]
\centering
\setcounter{subfigure}{0}

\subfloat[Peak X-ray intensity $\mathbf{y}$.]
{\includegraphics[width=.31\textwidth]
	{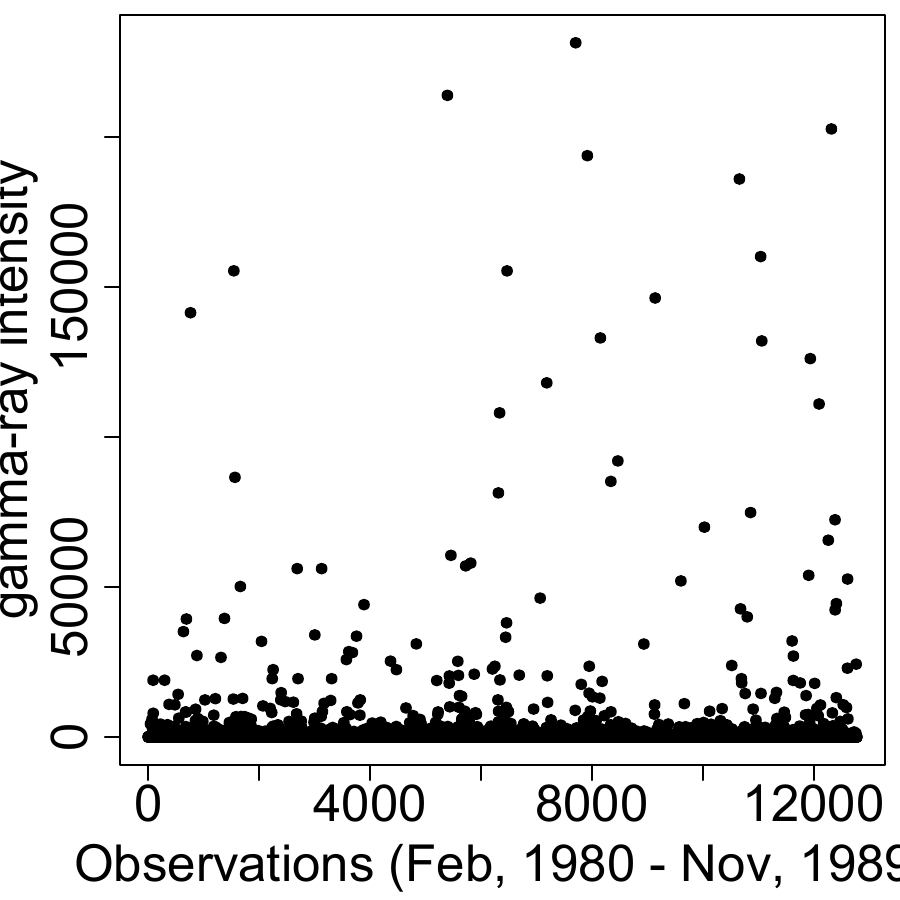}\label{fig:flares}}
\hspace{0.01\textwidth}
\subfloat[KDE fit of $\mathbf{y}$ over time. For better visualization $y < 165$.]
{\includegraphics[trim = 15mm 10mm 35mm 20mm , clip, width=.31\textwidth]
	{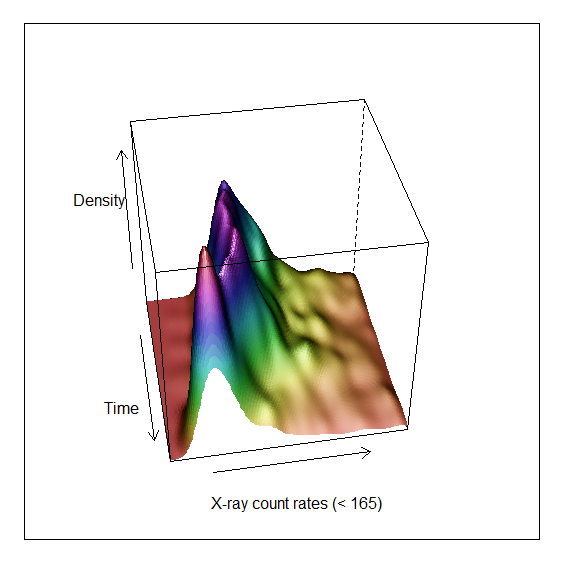}\label{fig:Gaussianized_Flares_cutoff_KDE3D}}
\hspace{0.01\textwidth}
\subfloat[KDE fit of $\mathbf{x}_{\widehat{\tau}}$ over time (no truncation in $x$).] {\includegraphics[trim = 15mm 10mm 35mm 20mm, clip, width=.31\textwidth]
	{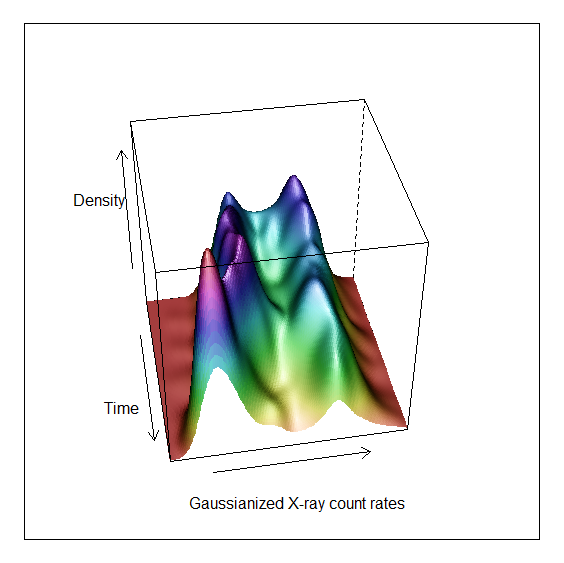}\label{fig:Gaussianized_Flares_KDE3D}}\\

\subfloat[Zoom to $\mathbf{y} \leq 400$; horizontal lines $\widehat{y}_{\min} = 323 \pm 89$ for power-law cut-off. 
]
{\includegraphics[width=.31\textwidth]{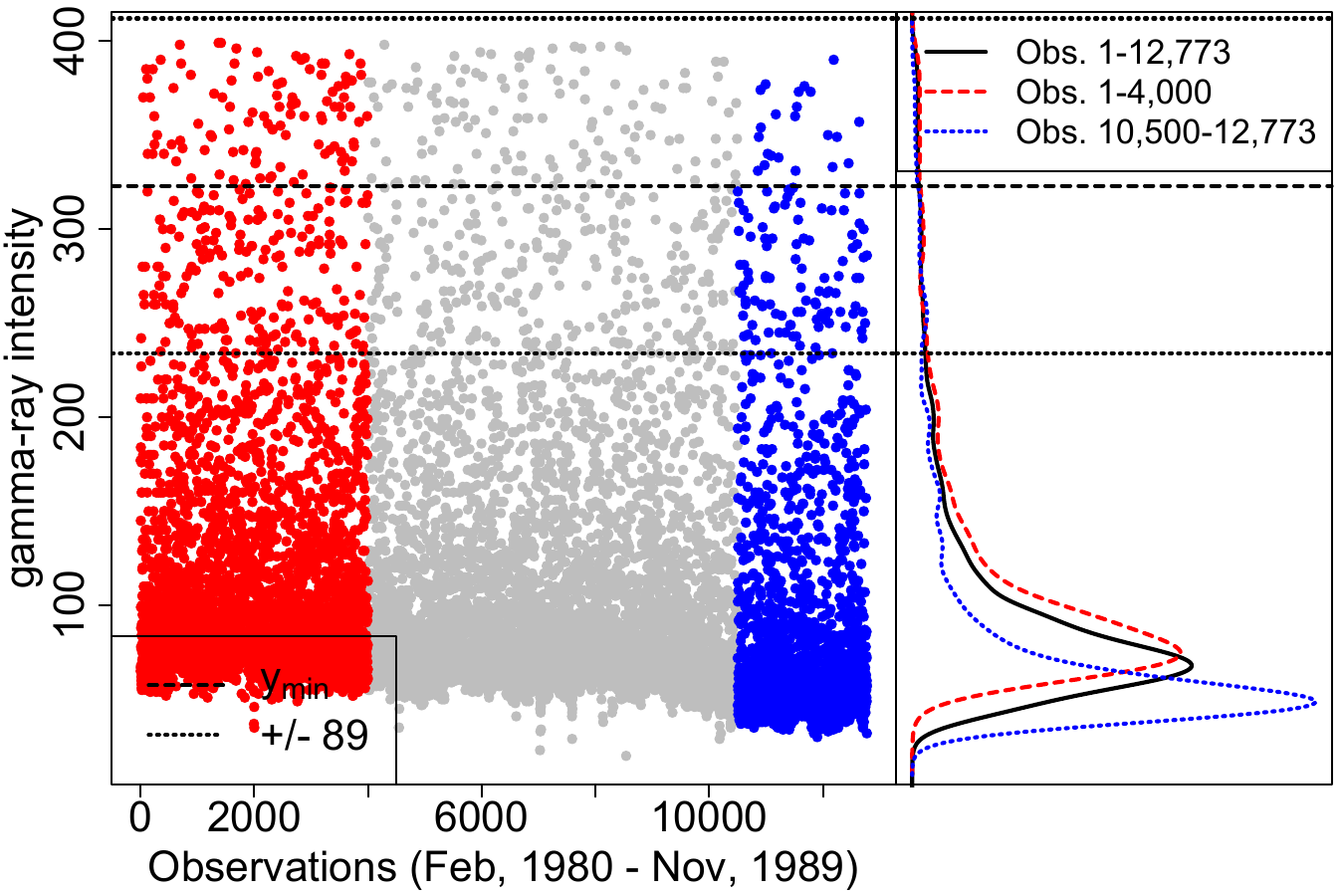}\label{fig:flares_cutoff} }
\hspace{0.01\textwidth}
\subfloat[Back-transformed $\mathbf{x}_{\widehat{\tau}_1}$, $\widehat{\tau}_1 = (74.46, 26.32, 1.53)$; horizontal lines at $W_{\widehat{\tau}_1}(\widehat{y}_{\min} \pm 89) = 114.94$ and $(110.96, 117.58)$. 
]{\includegraphics[width=.31\textwidth]
	{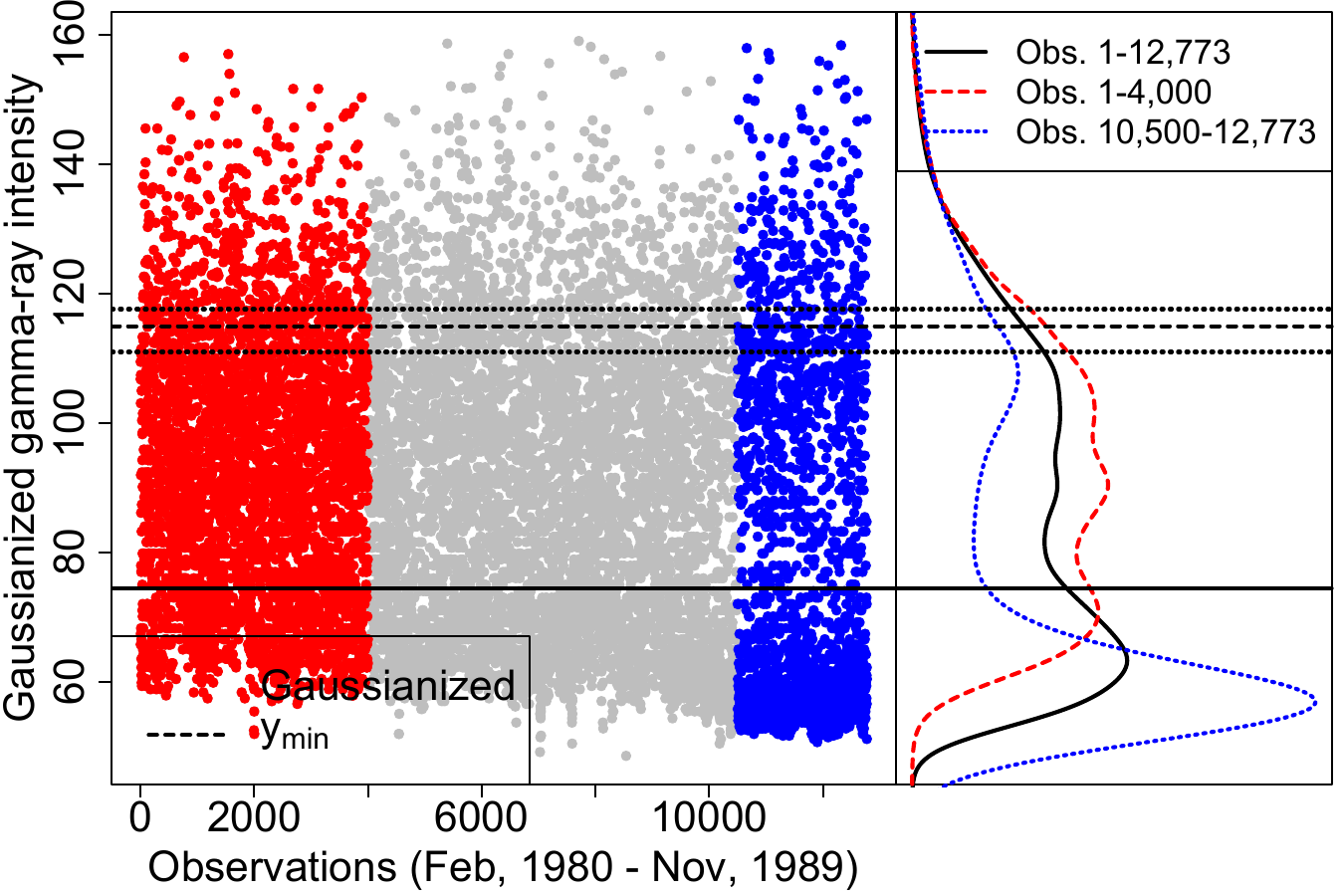}\label{fig:Gaussianized_flares_h}}
\hspace{0.01\textwidth}
\subfloat[Back-transformed $\mathbf{x}_{\widehat{\tau}}$, $\widehat{\tau} = (86.97, 26.80, 0, 2.37)$; horizontal lines at $W_{\widehat{\tau}}(\widehat{y}_{\min}\pm 89)  = 121.16$ and  $(117.74, 123.37)$. 
]
{\includegraphics[width=.31\textwidth]
	{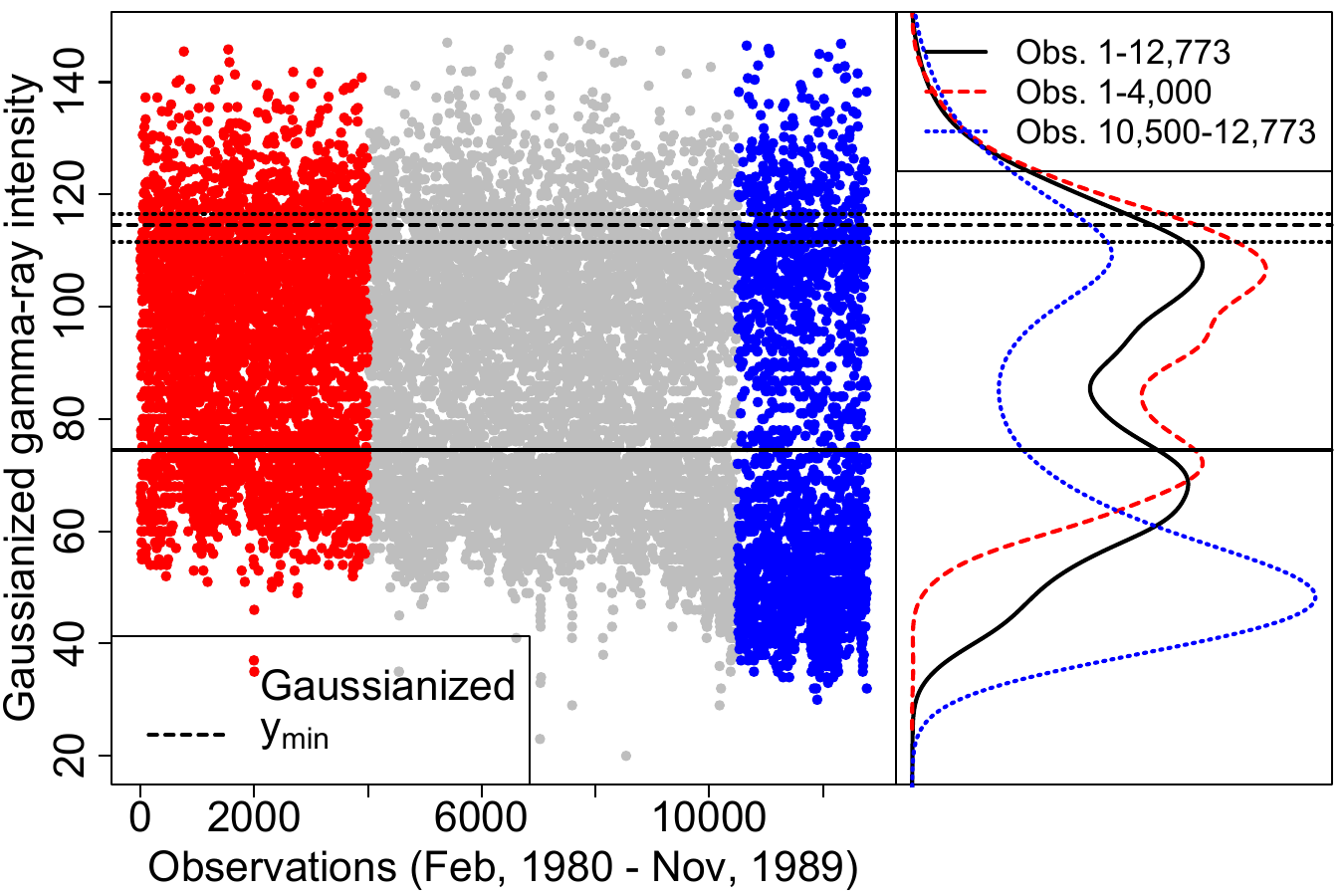}\label{fig:Gaussianized_flares}}
\hspace{0.01\textwidth}

\caption{\label{fig:solar_flares} Peak X-ray count rates of solar flares.}
\end{figure}

Here I study solar flare gamma-ray count rates \citep{ClausetShaliziNewman09_powerlaws, Newman05_power}. The data\footnote{Dataset \texttt{SolarFlares} in the \texttt{LambertW} package.} were collected approximately four times a day from Feb.\ 1980 until Nov.\ 1989 giving $T = 12,773$ observations. See \citet{Dennis91_SolarFlaresOriginal} for details and scientific background.\\

The gamma-ray count rates exhibit a strong right heavy tail (Fig.\ \ref{fig:flares}), which makes more detailed visual inspection as well as simple EDA difficult. A zoom to $y_i \leq 400$ in Fig.\ \ref{fig:flares_cutoff} shows that a lot of counts lie between $50$ and $100$ and this level drops off at the end of the observation cycle. This drop is not an intrinsic characteristic of solar flares but due to a decreasing sensitivity of the X-ray detectors over time \citep{Dennis91_SolarFlaresOriginal}.  For the sake of comparison with  \citet{ClausetShaliziNewman09_powerlaws, Newman05_power} most estimates are based on all $T = 12,773$ observations. Figures \ref{fig:flares_cutoff}, \ref{fig:Gaussianized_flares_h}, and \ref{fig:Gaussianized_flares} also show separate density estimates for the first $4,000$\edit{why 4000?} and last $2,273$ observations, and while the estimates change, the qualitative findings do not.\\

\citet{ClausetShaliziNewman09_powerlaws} find that a power-law ($\widehat{a} = 1.79 (0.02)$) with cut-off ($\widehat{y}_{\min} = 323 (89)$) gives the best fit amongst various alternatives. However, this first EDA might not be complete: not only visually heavy tails can obscure underlying non-trivial structure, but also estimates - such as the power law fit or non-parametric density estimates (Fig.\ \ref{fig:flares_cutoff} and \ref{fig:Gaussianized_Flares_cutoff_KDE3D}) - are affected by the heavy right tail. Here I show that Gaussianizing this data reveals new insights for the data-generating process, with a new interpretation for the optimality of the cut-off.

A Lambert W $\times$ Gaussian MLE fit $\widehat{\theta} = (\widehat{\mu}_x, \widehat{\sigma}_x, \widehat{\delta_{\ell}},\widehat{\delta_{r}}) = (86.97, 26.80, 0, 2.37)$ confirms that only the right tail (tail index $1/2.373 = 0.421$) needs a Gaussianizing transformation.\footnote{For comparison Fig.\ \ref{fig:Gaussianized_flares_h} also shows the back-transformed data $\mathbf{x}_{\widehat{\tau}_1}$ using the same $\delta$ on each tail ($\widehat{\tau}_1 = (74.46, 26.32, 1.53)$). However, due to the clear right heavy tail I will continue with the $(\delta_l, \delta_r)$ transformation.} 

The last column of Fig.\ \ref{fig:solar_flares} shows EDA for the Gaussianized data. Removing the heavy right tail  reveals a bimodal structure, which gives additional meaning to $\widehat{y}_{\min} = 323$. The Gaussianized cut-off value equals $W_{\widehat{\tau}}(323) = 121.16$ with the transformed standard deviation interval $[117.74, 123.37]$ (corresponding to $323 \pm 89$). Fitting a two component Gaussian mixture model to $\mathbf{x}_{\widehat{\tau}}$ yields $\widehat{\lambda} \, \mathcal{N}_1(67.10, 14.04^2) + (1 - \widehat{\lambda}) \, \mathcal{N}_2(113.12, 14.27^2)$ with $\widehat{\lambda} = 0.52$ and optimal decision boundary between classes of $90.48$. The mean of the larger component, $113.12$, lies within one standard deviation of the optimal Gaussianized cut-off $121.16$: for lower cut-offs the left-tail of the larger component -- or for much lower cut-offs even the smaller component -- would counteract the power-law decay of the upper gamma-ray count rates.\\

As mentioned above, this analysis is not intended to describe the underlying process of solar flare gamma rays; it should rather show new insights that can be gained by Gaussianizing. Future research based on these new findings might lead to new physical interpretations of the statistical properties gamma-ray count rates, see for example \citet{Aschwanden11}.

\section{Discussion and Outlook}
\label{sec:discussion_outlook}
I adapt the skewed Lambert W input / output framework to introduce heavy tails in continuous RVs $X \sim F_X(x)$. For Gaussian input this not only contributes to existing work on Tukey's $h$ distribution, but also gives convincing empirical results: unimodal data with heavy tails can be transformed to Gaussian data/RVs. Properties of a Gaussian model $\mathcal{M}_{\mathcal{N}}$ on the back-transformed data mimic the features of the ``true'' skewed, heavy-tailed model $\mathcal{M}_{G}$ very closely.

Since Gaussianity is the single most typical, and often required, assumption in many areas of statistics, machine learning, and signal processing, future research can take many directions. From a theoretical perspective properties of Lambert W $\times$ $F_X$ distributions viewed as a generalization of already well-known distributions $F_X$ can be studied. This area will profit from existing literature on the Lambert W function, which has been discovered only recently by the statistics community. Empirical work can focus on transforming the data and compare performances of approximate Gaussian versus joint heavy-tail analysis. The comparisons in this work showed that approximate inference for Gaussianized data is comparable with the direct heavy tail modeling, and so provides an easy tool to improve inference for heavy-tailed data in statistical practice.

\blinded{I also provide the R package \href{cran.r-project.org/web/packages/LambertW}{\texttt{LambertW}}, publicly available at \href{cran.r-project.org/}{\texttt{CRAN}}, to facilitate the use of Lambert W $\times$ $F_X$ distributions in practice.}
\blinded{
\subsection*{Acknowledgments}
I want to thank Andrew F.\ Siegel who brought Tukey's $h$ distribution to my attention, and Brian R.\ Dennis who gave detailed background information and suggestions on the solar flares dataset.
}

\cleardoublepage
\phantomsection
\addcontentsline{toc}{section}{References}
\bibliographystyle{chicago} 
\bibliography{../LambertW/LambertWdistribution}	

\cleardoublepage

\linenumbers
\appendix
\begin{center}
\LARGE Supplementary Material
\end{center}

\section{Auxiliary Results and Properties}
\label{sec:Auxiliary}
\subsection{Inverse Transformation $W_{\delta}(z)$}
The function $W_{\delta}(z)$ is the building block of Lambert W $\times$ $F_X$ distributions. This section lists useful properties of $W_{\delta}(z)$ as a function of $z$ as well as a function of $\delta$. 

\begin{properties}
\label{prop:identities_W_delta_0}
For $\delta = 0$,
\begin{equation}
W_{\delta}(z_i)\mid_{\delta = 0} = z_i, \quad W'(\delta z_i^2) \mid_{\delta = 0} = z_i^2, \quad \text{ and } W\left(\delta z_i^2 \right) \mid_{\delta = 0} = 0.
\end{equation}
\end{properties}

By definition $\frac{W_{\delta}(z)}{z} = e^{- \frac{ \delta }{2} W_{\delta}(z)^2 }$ and therefore 
\begin{equation}
\log \frac{W_{\delta}(z)}{z} = - \frac{ \delta }{2} W_{\delta}(z)^2 = -\frac{W(\delta z^2)}{2}.
\end{equation}

\begin{lemma}[Derivative of $W_{\delta}(z)$ with respect to $z$]
\label{lem:derivative_W_delta_wrt_z}
It holds
\begin{equation}
\frac{d}{dz} W_{\delta}\left(z\right) = - \frac{W_{\delta}\left(z\right)}{z \left( 1 + \delta W_{\delta}\left(z\right)^2 \right)} = e^{- \frac{1}{2} W(\delta z^2) } \frac{1}{1 + W(\delta z^2)}
\end{equation}
\end{lemma}

\begin{proof}
One of the many interesting properties of the Lambert W function relates to its derivative which satisfies
\begin{equation}
\label{eq:d1W}
W'(z) = \frac{W(z)}{z (1+ W(z))} = \frac{1}{e^{W(z)} (1 + W(z))}, \quad z \neq 0, -1/e.
\end{equation}
Hence,
\begin{eqnarray}
\frac{d}{dz} \frac{W\left( \delta z^2 \right) }{\delta} &=& W'\left(\delta z^2\right) \cdot 2 z  = 
 \frac{ W\left( \delta z^2 \right)}{\delta z^2 \left( 1 + W\left( \delta z^2 \right) \right)} \cdot 2 z  =  \frac{2  W\left( \delta z^2 \right)}{\delta z \left( 1 + W\left( \delta z^2 \right) \right)}
\end{eqnarray}

Therefore,
\begin{eqnarray}
 \frac{d}{dz} W_{\delta}(z)	&=&  \frac{1}{2} \left( \frac{1}{\delta} W\left( \delta z^2 \right) \right)^{-1/2} \cdot \frac{d}{dz} \frac{W\left( \delta z^2 \right) }{\delta}  \\
 &= & \frac{1}{2} \left( \frac{1}{\delta} W\left( \delta z^2 \right) \right)^{-1/2} \cdot \frac{2 W\left( \delta z^2 \right)}{\delta z \left( 1 + W\left( \delta z^2 \right) \right)} \\
  &= & \frac{1}{\delta^{1/2}}  \left( W\left( \delta z^2 \right) \right)^{-1/2} \cdot \frac{W\left( \delta z^2 \right)}{z \left( 1 + W\left( \delta z^2 \right) \right)}
\end{eqnarray}

As $W\left( \delta z^2 \right) = \delta u^2$ the last line simplifies to
\begin{equation}
\frac{1}{\delta^{1/2}} \frac{1}{\delta^{1/2} u} \cdot \frac{\delta u^2}{z \left( 1 + \delta u^2 \right)} = \frac{u}{z \left( 1 + \delta u^2 \right)}.
\end{equation}
Now use again $u = W_{\delta}(z)$.
\end{proof}

\begin{lemma}[Derivative of $W_{\delta}(z)^2$ with respect to $\delta$]
\label{lem:derivative_W_delta_square}
For all $z \in \R$
\begin{align}
\partialx{ }{\delta} \left[ W_{\delta}(z) \right]^2 &= - \frac{1}{1 + W\left(\delta z^2\right)} W_{\delta}\left(z\right)^{4} \leq 0.
\end{align}

\end{lemma}

\begin{proof}
By definition $\left[ W_{\delta}(z) \right]^2 = \frac{W(\delta z^2)}{\delta}$. Thus
\begin{align}
\partialx{ }{\delta}  \frac{W\left(\delta z^2\right)}{\delta} &= \frac{\delta \partialx{ }{\delta}  W\left(\delta z^2\right) -  W\left(\delta z^2\right) \cdot 1}{\delta^2} \\
&= \frac{\delta W'\left(\delta z^2\right) z^2 -  W\left(\delta z^2\right)}{\delta^2} \\
&=  \frac{\delta \frac{W\left(\delta z^2\right)}{\delta z^2 (1 + W\left(\delta z^2\right))} z^2 -  W\left(\delta z^2\right)}{\delta^2} \\
&=  \frac{\frac{W\left(\delta z^2\right)}{1 + W\left(\delta z^2\right)} -  W\left(\delta z^2\right)}{\delta^2} \\
&=  \frac{\frac{-W\left(\delta z^2\right)^2}{1 + W\left(\delta z^2\right)}}{\delta^2} \\
&= -  \frac{1}{1 + W\left(\delta z^2\right)} \left[W_{\delta}(z) \right]^4.
\end{align}
Since both terms are non-negative for all $z \in \R$, the result follows.
\end{proof}

That is $W_{\delta}(z)^2$ is a decreasing function in $\delta$ for every $z \in \R$, i.e.\ the more we remove heavy tails the more $z$ gets shrinked (non-linearly) towards $0 = \lim_{\delta \rightarrow \infty} W_{\delta}(z)$. In particular, $\left[ W_{\delta}(z) \right]^2 < z^2 \Leftrightarrow \frac{W_{\delta}(z)}{z} < 1$ and $\frac{W_{\delta + \varepsilon}(z)}{z} < \frac{W_{\delta}(z)}{z}$for $\delta \geq 0$ and $\varepsilon > 0$.

\begin{lemma}[Derivative of $W_{\delta}(z)$ with respect to $\delta$]
\label{lem:derivative_W_delta}
It holds
\begin{equation}
\partialx{}{\delta} W_{\delta}\left(z\right) = - \frac{1}{2} \frac{1}{1 + W\left(\delta z^2\right)} W_{\delta}\left(z\right)^{3}
\end{equation}
\end{lemma}

\begin{proof}

\begin{align}
\partialx{}{\delta} W_{\delta}\left(z\right) &= \operatorname{sgn}(z) \partialx{ }{\delta}  \left( \frac{W\left(\delta z^2\right)}{\delta} \right)^{1/2}  \\
&= \operatorname{sgn}(z) \frac{1}{2}  \left( \frac{W\left(\delta z^2\right)}{\delta} \right)^{-1/2} \partialx{ }{\delta}  \frac{W\left(\delta z^2\right)}{\delta} \\
&= \frac{1}{2} \frac{1}{W_{\delta}(z)} \partialx{ }{\delta} \left[ W_{\delta}(z) \right]^2 \\
&= - \frac{1}{2} \frac{1}{1 + W\left(\delta z^2\right)} W_{\delta}\left(z\right)^{3},
\end{align}
where the last line follows by Lemma \ref{lem:derivative_W_delta_square}.
\end{proof}

\subsection{Penalty $\log R\left( \delta \mid  z_i \right)$ for Standard Gaussian Input}
For $\mu_x = 0$ and $\sigma_x = 1$ the penalty equals ($y_i = z_i$)
\begin{align}
\label{eq:R_penalty_01}
R\left( \delta \mid  z_i \right) & = \frac{W_{\delta}\left( z_i \right)}{ z_i \left[ 1 + \delta \left( W_{\delta}\left(z_i\right) \right)^2 \right]} = \frac{W_{\delta}\left( z_i \right)}{ z_i \left[ 1 + W\left(\delta z_i^2 \right) \right]}
\end{align}
and thus
\begin{align}
\log R\left( \delta \mid  z_i \right) & = \log \frac{W_{\delta}\left( z_i \right)}{z_i} - \log \left[ 1 + W\left(\delta z_i^2 \right) \right] \\
& = - \frac{W(\delta z_i^2 )}{2} - \log \left[ 1 + W\left(\delta z_i^2 \right) \right]
\end{align}

\begin{lemma}[Derivative of $\log R\left( \delta \mid  z \right)$ with respect to $\delta$]
\label{lem:deriv_penalty}
For all $\delta \geq 0$ and all $z \in \R$
\begin{equation}
\label{eq:deriv_penalty}
\partialx{\log R\left( \delta \mid  z \right)}{\delta}  = - z^2 W'(\delta z^2) \left( \frac{1}{2} + \frac{1}{1 + W\left(\delta z^2 \right)} \right) \leq 0.
\end{equation}
\end{lemma}

\begin{proof} 
We have
\begin{align}
\partialx{\log R\left( \delta \mid  z \right)}{\delta} & = \frac{1}{W_{\delta}\left( z \right)} \partialx{W_{\delta}\left( z \right)}{\delta} - \frac{1}{1 + W(\delta z^2)} W'(\delta z^2) z^2 \\
& \stackrel{\text{Lemma \ref{lem:derivative_W_delta}}}{=}  \frac{1}{W_{\delta}\left( z \right)} \left( - \frac{1}{2} \frac{1}{1 + W\left(\delta z^2\right)} W_{\delta}\left(z\right)^{3} \right) - \frac{1}{1 + W(\delta z^2)} W'(\delta z^2) z^2 \\
& = - \frac{1}{1 + W(\delta z^2)} \left( \frac{1}{2} W_{\delta}\left(z\right)^{2}  + W'(\delta z^2) z^2 \right)
\end{align}

Using $W'(\delta z^2) = \frac{W(\delta z^2)}{\delta z^2 (1 + W(\delta z^2))}$ and re-factorizing gives \eqref{eq:deriv_penalty}.

\end{proof}

\subsection{Gaussian log-Likelihood at $W_{\delta}(z)$}

\begin{lemma}[Derivative of the Gaussian log-likelihood at $W_{\delta}(z)$]
\label{lem:deriv_Gaussian_loglik}
For all $z \in \R$ and for $\delta \geq 0$
\begin{equation}
\partialx{ }{\delta}  \ell(\mu_x = 0, \sigma_x = 1 \mid W_{\delta}(z))  = \frac{1}{2} \frac{1}{1 + W\left(\delta z^2\right)} \left[ W_{\delta}\left(z\right) \right]^{4} \geq 0. 
\end{equation}
\end{lemma}

\begin{proof}
The $\log$ of the standard Gaussian pdf evaluated at $W_{\delta}(z)$ simplifies to
\begin{align}
\log \frac{1}{\sqrt{2 \pi}} e^{-\frac{1}{2} \left[ W_{\delta}(z) \right]^2} = \log  \frac{1}{\sqrt{2 \pi}} - \frac{1}{2}  \left[ W_{\delta}(z) \right]^2.
\end{align}
The rest follows by Lemma \ref{lem:derivative_W_delta_square}.
\end{proof}

Lemma \ref{lem:deriv_Gaussian_loglik} shows that increasing $\delta$ always increases the input log-likelihood $\ell(\delta \mid  \mathbf{u}_{\delta} = W_{\delta}(\mathbf{z}))$ - see also Fig.\ \ref{fig:sample_loglik_decomposition}. For $\delta \rightarrow \infty$ the Gaussianized $\mathbf{u}_{\delta}$ goes to $\mathbf{0}$, which clearly maximizes the Gaussian likelihood if $\mu = 0$.

\section{Proofs}
\label{sec:proofs}
\subsection{Inverse transformation}

\begin{proof}[Proof of Lemma \ref{lem:inverse_trafo_h}]

Without loss of generality assume that $\mu_x = 0$ and $\sigma_x =1$. Squaring  \eqref{eq:Tukey_h} and multiplying by $\delta$ yields
\begin{eqnarray}
\label{eq:delta_Z2}
\delta Z^2 &= & \delta U^2 \exp\left( \delta U^2 \right)
\end{eqnarray}

The inverse of \eqref{eq:delta_Z2} is by definition Lambert's $W(z)$ function \citep{Rosenlicht69}
\begin{equation}
W(z) \exp W(z) = z, \quad z \in \C.
\end{equation}
$W(z)$ is bijective for $z \geq 0$. Since $\delta U^2 \geq 0$ for all $\delta \geq 0$, applying $W(\cdot)$ to \eqref{eq:delta_Z2}, dividing by $\delta$, and taking the square root gives
\begin{eqnarray}
U &=& \pm \sqrt{\frac{W\left( \delta Z^2 \right)}{\delta}}
\end{eqnarray}
Since $\exp \left( \frac{\delta}{2} U^2 \right) > 0$ for all $\delta \in \R$ and all $U$, it follows that $Z = U \exp \left( \delta/2 U^2 \right)$ and $U$ must have the same sign, which concludes the proof.
\end{proof}

\subsection{Cdf and pdf}

\begin{proof}[Proof of Theorem \ref{theorem:cdf_Y}]
By definition,
\begin{align}
G_Y(y) &= \Prob (Y \leq y) = \Prob \left( \left\{ U \exp \left(\frac{\delta}{2} U^2 \right) \right\} \sigma_x + \mu_x \leq y \right) \\
& = \Prob \left(  U \exp \left(\frac{\delta}{2} U^2 \right)  \leq z \right) = \Prob\left( U \leq W_{\delta}(z) \right) \\
& =  F_U\left( U \leq W_{\delta}(z) \right).
\end{align}

Taking the derivative with respect to $y$ gives
\begin{eqnarray}
\frac{d}{dy} G_Y(y \mid \boldsymbol \beta, \delta) &=& f_X(W_{\delta}(z) \sigma_x + \mu_x \mid \boldsymbol \beta) \cdot \sigma_x \frac{d}{dy} W_{\delta}\left( \frac{y - \mu_x}{\sigma_x} \right) \\
&=& f_U(W_{\delta}(z) \mid \boldsymbol \beta) \cdot \sigma_x \frac{1}{\sigma_x} \frac{d}{dz} W_{\delta}\left( \frac{y - \mu_x}{\sigma_x} \right) \\
&=& f_U(W_{\delta}(z) \mid \boldsymbol \beta) \cdot \frac{d}{dz} W_{\delta}\left( z \right).
\end{eqnarray}

Using Lemma \ref{lem:derivative_W_delta_wrt_z} yields \eqref{eq:pdf_LambertW_Y}.
\end{proof}

\subsection{MLE for $\delta$}

\begin{lemma}[Derivative of the Lambert W $\times$ Gaussian log-likelihood]
\label{lem:derivative_likelihood_full}
We have
\begin{align}
\label{eq:MLE_deriv_equation}
D(\delta \mid \mathbf{z}) := \partialx{}{\delta} \ell(\delta \mid \mathbf{z}) & = \sum_{i=1}^{N} z_i^2 W'(\delta z_i^2) \left(\frac{1}{2} W_{\delta}\left(z_i\right)^{2} - \left( \frac{1}{2} + \frac{1}{1 + W\left(\delta z_i^2 \right)} \right) \right)  \\
& = \frac{1}{2} \sum_{i=1}^{N} \frac{W_{\delta}\left(z_i\right)^{4}}{1 + \delta W_{\delta}\left(z_i\right)^{2}} -  \sum_{i=1}^{N} \frac{W_{\delta}\left(z_i\right)^{2}}{1 + \delta W_{\delta}\left(z_i\right)^{2}} \left( \frac{1}{2} + \frac{1}{1 + \delta W_{\delta}\left(z_i\right)^{2}} \right) \\
\label{eq:MLE_deriv_equation_2}
& = \frac{1}{2} \sum_{i=1}^{N} \frac{W_{\delta}\left(z_i\right)^{4}}{1 + W(\delta z_i^2)} - \sum_{i=1}^{N} \frac{W_{\delta}\left(z_i\right)^{2}}{1 + W(\delta z_i^2)} \left( \frac{1}{2} + \frac{1}{1 + W\left(\delta z_i^2 \right)} \right).
\end{align}
\end{lemma}
\begin{proof}
Apply Lemmas \ref{lem:deriv_penalty} and \ref{lem:deriv_Gaussian_loglik} to $\partialx{}{\delta} \ell(\delta \mid \mathbf{z}) = \partialx{ }{\delta} \log R\left( \delta \mid  z \right) + \partialx{ }{\delta}  \ell(\mu_x = 0, \sigma_x = 1 \mid W_{\delta}(z))$.
\end{proof}

\begin{proof}[Proof sketch of Theorem \ref{thm:MLE_delta}]

\begin{enumerate}[a)]
\item If condition \eqref{eq:cond_delta_MLE_greater_0} holds, then $D(\delta \mid \mathbf{z}) < 0$ at $\delta  = 0$ and stays negative for all $\delta > 0$. Hence the maximizer occurs at the boundary $\delta = 0$.
\item If \eqref{eq:cond_delta_MLE_greater_0} does not hold, then $D(\delta = 0 \mid \mathbf{z}) > 0$, decreases in $\delta$ and crosses the zero line (one candidate for $\widehat{\delta}_{MLE}$ occurs here). 
\item As $\delta$ gets larger, $D(\delta \mid \mathbf{z})$ reaches a minimum (negative value) and starts increasing. However, for $\delta \rightarrow \infty$ the derivative approaches zero from below and never equals zero again; thus $\widehat{\delta}_{MLE}$ is unique.
\end{enumerate}
\end{proof}

\begin{proof}[Proof of Theorem \ref{thm:MLE_delta}] 

\begin{enumerate}[a)]
\item The log-likelihood is increasing at $\delta = 0$ if and only if (set $\delta = 0$ in \eqref{eq:MLE_deriv_equation_2} and use Property \ref{prop:identities_W_delta_0})
\begin{align}
\label{eq:MLE_deriv_equation_delta0}
\sum_{i=1}^{N} z_i^4 > 3 \sum_{i=1}^{N} z_i^2.
\end{align} 
Eq.\ \eqref{eq:MLE_deriv_equation_delta0} means that transforming the data (choosing $\widehat{\delta} > 0$) increases the overall likelihood only if the data is heavy-tailed enough. Note that the sum of squares is not squared again. Hence condition \eqref{eq:MLE_deriv_equation_delta0} is not equivalent for the data having empirical kurtosis larger than $3$.

\item  If \eqref{eq:MLE_deriv_equation_delta0} does not hold, then $\widehat{\delta}_{MLE}$ must satisfy $D(\delta \mid \mathbf{z}) \mid_{\delta = \widehat{\delta}_{MLE}} = 0$ from \eqref{eq:MLE_deriv_equation} in Lemma \ref{lem:derivative_likelihood_full}. It remains to be shown that this equation has (at least) one positive solution.
\begin{enumerate}[i)]
\item Since $\lim_{\delta \rightarrow \infty} W_{\delta}(z) = 0$ for all $z \in \R$, \eqref{eq:MLE_deriv_equation_2} is also true in the limit; however, we can ignore this solution as we require $\widehat{\delta}_{MLE} \in \R$.
\item By continuity and $\lim_{\delta \rightarrow \infty} W_{\delta}(z) = 0$, for sufficiently large $\delta_M$, $W_{\delta_M}(z_i) < 1$ for all $z_i \in \R$.  Hence $W_{\delta_M}(z_i)^4 < W_{\delta_M}(z_i)^2$ and therefore
\begin{align}
\frac{1}{2} \sum_{i=1}^{N} \frac{W_{\delta}\left(z_i\right)^{4}}{1 + \delta W_{\delta}\left(z_i\right)^{2}} &< \frac{1}{2} \sum_{i=1}^{N} \frac{W_{\delta}\left(z_i\right)^{2}}{1 + \delta W_{\delta}\left(z_i\right)^{2}} \\
 &<  \sum_{i=1}^{N} \frac{W_{\delta}\left(z_i\right)^{2}}{1 + \delta W_{\delta}\left(z_i\right)^{2}} \left( \frac{1}{2} + \frac{1}{1 + \delta W_{\delta}\left(z_i\right)^{2}} \right)  \text{ for } \delta \geq \delta_M,
\end{align}
showing that $D(\delta \mid \mathbf{z}) \mid_{\delta \geq \delta_M} < 0$. That is, $D(\delta \mid \mathbf{z})$ approaches $0$ from below for $\delta \rightarrow \infty$.
\item By continuity and $D(\delta \mid \mathbf{z}) \mid_{\delta = 0} > 0$ (if \eqref{eq:MLE_deriv_equation_delta0} does not hold), it must cross the $D(\delta \mid \mathbf{z}) = 0$ line at least once in the interval $(0, \delta_M)$, proving the existence of $\widehat{\delta}_{MLE}$.
\end{enumerate}

\item The log-likelihood can be decomposed in
\begin{equation}
\ell \left( \delta \mid \mathbf{z} \right) \propto \underbrace{-\frac{1}{2} \sum_{i=1}^{N} \left[ W_{\delta}(z_i) \right]^2}_{\ell(\mu_x = 0, \sigma_x = 1 \mid W_{\delta}(\mathbf{z}))} + \underbrace{ \sum_{i=1}^{N} \log \frac{W_{\delta}\left( z_i \right)}{z_i} - \log  \left[ 1 + W\left(\delta z_i^2 \right) \right] }_{\mathcal{R}(\delta \mid \mathbf{z})}.
\end{equation}
Lemmas \ref{lem:deriv_penalty} and \ref{lem:deriv_Gaussian_loglik} show that $\mathcal{R}(\delta \mid \mathbf{z})$ is monotonically decreasing and $\ell(\mu_x = 0, \sigma_x = 1 \mid W_{\delta}(\mathbf{z}))$ is monotonically increasing in $\delta$.  

Furthermore, $\lim_{\delta \rightarrow \infty} \ell(\mu_x = 0, \sigma_x = 1 \mid W_{\delta}(\mathbf{z})) = 0$, that is the input likelihood is monotonically increasing but bounded from above (by $0 = \log 1$). On the other hand the penalty is decreasing without bounds, $\lim_{\delta \rightarrow \infty} \mathcal{R}(\delta \mid \mathbf{z}) = -\infty$. Thus their sum attains a global maximum either at the unique mode of $\ell \left( \delta \mid \mathbf{z} \right)$ or at the boundary $\delta = 0$ - see also Fig. \ref{fig:sample_loglik_decomposition}.

\end{enumerate}

\end{proof}

\section{Details on IGMM}
\label{sec:details_IGMM}

Here I present an iterative method to obtain $\widehat{\tau}$, which builds on the input/output aspect and theoretical properties of the input $X$. For example, if a random variable should be exponentially distributed (e.g.\ waiting times), but the observed data shows heavier tails then it is natural to estimate $\sigma_x = \lambda^{-1}$ and $\delta$ such that the back-transformed data has skewness $2$, as this is a particular property of exponential RVs - independent of the rate parameter $\lambda$; to remove heavy tails in $\mathbf{y}$ we should choose $\tau$ such that the back-transformed data $\mathbf{x}_{\tau}$ has sample kurtosis $3$; or for uniform input, we can try to find a $\tau$ such that $\mathbf{x}_{\tau}$ has a flat density estimate.

Here I describe the estimator for $\tau$ to remove heavy-tails in location-scale data, in the sense that the kurtosis of the input equals $3$. It can be easily adapted to match other properties of the input as outlined above.\\

For a moment assume that $\mu_x = \mu_x^{(0)}$ and $\sigma_x = \sigma_x^{(0)}$ are known and fixed; only $\delta$ has to be estimated. A natural choice for $\delta$ is the one that results in back transformed data $\mathbf{x}_{\tau}$ ($\tau = (\mu_x^{(0)}, \sigma_x^{(0)}, \delta)$) with sample kurtosis $\widehat{\gamma}_2(\mathbf{x}_{\tau})$ equal to the theoretical kurtosis $\gamma_2(X)$. Formally, 
\begin{equation}
\label{eq:delta_GMM}
\widehat{\delta}_{\textsc{GMM}} = \arg \min_{\delta} \vnorm{ \gamma_2(X) - \widehat{\gamma}_2(\mathbf{x}_{\tau}) },
\end{equation}
where $\vnorm{\cdot}$ is a proper norm in $\R$.

While the concept of this estimator is identical to its skewed version \citep{GMGLambertW_Skewed}, it has one important advantage: the inverse transformation is bijective. Thus here we do not have to consider ``lost'' data points when applying the inverse transformation.

\begin{algorithm}[t]
  \caption{\label{alg:delta_GMM} Find optimal $\delta$\blinded{: function \texttt{delta\_GMM($\cdot$)} in the \texttt{LambertW} package.}}

  \begin{algorithmic}[1]
	\renewcommand{\algorithmicrequire}{\textbf{Input:}}
	\renewcommand{\algorithmicensure}{\textbf{Output:}}
\REQUIRE standardized data vector $\mathbf{z}$; theoretical kurtosis $\gamma_2(X)$
\ENSURE $\widehat{\delta}_{GMM}$ as in \eqref{eq:delta_GMM}

    \medskip
\STATE $\widehat{\delta}_{GMM} = \arg \min_{\delta} \vnorm{\widehat{\gamma}_2(\mathbf{u}_{\delta}) - \gamma_2(X)}$, where $\mathbf{u}_{\delta} = W_{\delta}(\mathbf{z})$ subject to $\delta \geq 0$
  \RETURN $\widehat{\delta}_{GMM}$
  \end{algorithmic}  
\end{algorithm}

\paragraph{Discussion of Algorithm \ref{alg:delta_GMM}:}
The kurtosis of $Y$ as a function of $\delta$ is continuous and monotonically increasing (see \eqref{eq:kurt_Y_delta}). Also $u = W_{\delta}(z)$ has a smaller slope than the identity $u = z$, and the slope is decreasing as $\delta$ is increasing. Thus if the kurtosis of the original data is larger than the target kurtosis of the back-transformed data, $\widehat{\gamma}_2(\mathbf{y}) > \gamma_2(X)$, then there always exists a $\delta^{(*)}$ that achieves $ \widehat{\gamma}_2(\mathbf{x}_{\tau^*}) \equiv \gamma_2(X)$.  By the re-parametrization $\tilde{\delta} = \log \delta$ the bounded optimization problem can be solved by standard (unbounded) optimization algorithms.\\

In practice, $\mu_x$ and $\sigma_x$ are rarely known but also have to be estimated from the data. As $\mathbf{y}$ is shifted and scaled \emph{ahead of} the back-transformation $W_{\delta}(\cdot)$, the initial choice of $\mu_x$ and $\sigma_x$ affects the optimal choice of $\delta$. Therefore the optimal triple $\widehat{\tau} = (\widehat{\mu}_x, \widehat{\sigma}_x, \widehat{\delta})$ must be obtained iteratively. 

\begin{algorithm}[t]
  \caption{\label{alg:IGMM} Iterative Generalized Method of Moments (IGMM) \blinded{: function \texttt{IGMM($\cdot$)} in the \texttt{LambertW} package.}}

  \begin{algorithmic}[1]
	\renewcommand{\algorithmicrequire}{\textbf{Input:}}
	\renewcommand{\algorithmicensure}{\textbf{Output:}}
\REQUIRE  data vector $\mathbf{y}$; tolerance level $tol$; theoretical kurtosis $\gamma_2(X)$
\ENSURE IGMM parameter estimate $\widehat{\tau}_{\textsc{IGMM}} = (\widehat{\mu}_x, \widehat{\sigma}_x, \widehat{\delta})$

    \medskip
    \STATE Set $\tau^{(-1)} = (0,0,0)$
    \STATE Starting values: $\tau^{(0)} = (\mu_x^{(0)}, \sigma_x^{(0)},\delta^{(0)})$, where $\mu_x^{(0)} = \tilde{\mathbf{y}}$ and $\sigma_x^{(0)} = \overline{\sigma}_y \cdot \left( \frac{1}{\sqrt{(1-2 \delta^{(0)})^{3/2}}} \right)^{-1}$ are the sample median and standard deviation of $\mathbf{y}$ divided by the standard deviation factor (see also \eqref{eq:Y_moments}), respectively. $\delta^{(0)} = \frac{1}{66} \left( \sqrt{66 \widehat{\gamma}_2(\mathbf{y}) - 162} - 6 \right)$ $\rightarrow$ see \eqref{eq:delta_Taylor_rule} for details.
    \STATE $k=0$
    \WHILE{$\vnorm{\tau^{(k)} - \tau^{(k-1)}} > tol$}
    \STATE $\mathbf{z}^{(k)} = (\mathbf{y} - \mu_x^{(k)})/\sigma_x^{(k)}$
    \STATE Pass $\mathbf{z}^{(k)}$ to Algorithm \ref{alg:delta_GMM} $\longrightarrow \delta^{(k+1)}$
    \STATE back-transform $\mathbf{z}^{(k)}$ to $\mathbf{u}^{(k+1)} = W_{\delta^{(k+1)}}(\mathbf{z}^{(k)})$; compute $\mathbf{x}^{(k+1)} = \mathbf{u}^{(k+1)} \, \sigma_x^{(k)} + \mu_x^{(k)}$
    \STATE Update parameters: $\mu_x^{(k+1)} = \overline{\mathbf{x}}_{k+1}$ and $\sigma_x^{(k+1)} = \widehat{\sigma}_{x_{k+1}}$  \label{line:set_new_theta}
    \STATE $\tau^{(k+1)} = ( \mu_x^{(k+1)}, \sigma_x^{(k+1)}, \delta^{(k+1)})$
    \STATE $k=k+1$
    \ENDWHILE
    \RETURN $\tau_{IGMM} = \tau^{(k)}$
  \end{algorithmic}
\end{algorithm}

\paragraph{Discussion of Algorithm \ref{alg:IGMM}:} Algorithm \ref{alg:IGMM} first computes $\mathbf{z}^{(k)} = (\mathbf{y} - \mu_x^{(k)})/\sigma_x^{(k)}$ using $\mu_x^{(k)}$ and $\sigma_x^{(k)}$ from the previous step. This normalized output can then be passed to Algorithm \ref{alg:delta_GMM} to obtain an updated $\delta^{(k+1)}= \widehat{\delta}_{GMM}$. Using this new $\delta^{(k+1)}$ one can back-transform $\mathbf{z}^{(k)}$ to $\mathbf{u}^{(k+1)} = W_{\delta^{(k+1)}}(\mathbf{z}^{(k)})$, and consequently obtain a better approximation to the ``true'' latent $\mathbf{x}$ by $\mathbf{x}^{(k+1)} = \mathbf{u}^{(k+1)} \, \sigma_x^{(k)} + \mu_x^{(k)}$. However, $\delta^{(k+1)}$ - and therefore $\mathbf{x}^{(k+1)}$ - has been obtained using $\mu_x^{(k)}$ and $\sigma_x^{(k)}$, which are not necessarily the most accurate estimates in light of the updated approximation $\widehat{\mathbf{x}}_{( \mu_x^{(k)}, \sigma_x^{(k)}, \delta^{(k+1)})}$. Thus Algorithm \ref{alg:IGMM} computes new estimates $\mu_x^{(k+1)}$ and $\sigma_x^{(k+1)}$ by the sample mean and standard deviation of $\widehat{\mathbf{x}}_{( \mu_x^{(k)}, \sigma_x^{(k)}, \delta^{(k+1)})}$, and starts another iteration by passing the updated normalized output $\mathbf{z}^{(k+1)} = \frac{\mathbf{y} - \mu_x^{(k+1)}}{\sigma_x^{(k+1)}}$ to Algorithm \ref{alg:delta_GMM} to obtain a new $\delta^{(k+2)}$. 

It returns the optimal $\widehat{\tau}_{\textsc{IGMM}}$ once convergence has been reached, i.e., if $\vnorm{\tau^{(k)} - \tau^{(k+1)}} < tol$.\\

\begin{remark}[IGMM for double-tail Lambert W $\times$ $F_X$]
For a double-tail fit the one-dimensional optimization in  Algorithm \ref{alg:delta_GMM} has to be replaced with a two-dimensional optimization
\begin{equation}
\label{eq:delta2_GMM}
\left( \widehat{\delta}_{\ell}, \widehat{\delta}_r \right)_{\textsc{GMM}} = \arg \min_{\delta_{\ell}, \delta_r} h\left( \gamma_2(X) - \widehat{\gamma}_2(\mathbf{x}_{(\mu_x^{*}, \sigma_x^{*}, \delta_{\ell}, \delta_r)}) \right).
\end{equation}
Algorithm \ref{alg:IGMM} remains unchanged.
\end{remark}

\clearpage

\section{Simulation Details}

\paragraph{Slightly heavy-tailed: $\delta = 1/10$.} 
Here the RV $Y$ has slight excess kurtosis ($3 + 2.51$) and $\sigma_y(\delta, \sigma_x = 1) = 1.18$. The Lambert W estimates of  $\widehat{\tau}$ are unbiased, and have smaller empirical standard deviation for $\widehat{\mu}_x$ than the Gaussian MLE or the median. Also using Lambert W estimators does not give worse estimates for $\sigma_y$.

\begin{table}[!ht]
\centering

\subfloat[Slightly heavy-tailed data: $\delta = 1/10$ \label{tab:simulations_delta_1_10}]{
\tiny
\centering
\begin{tabular}{|c||r||rr||rrr|r||rrr|r||r|}\hline 
 \multicolumn{1}{|c||}{$\delta = 1/10$}&\multicolumn{1}{|c||}{median}&\multicolumn{2}{c||}{Gaussian MLE} &\multicolumn{4}{c||}{IGMM}& \multicolumn{4}{c||}{Lambert W MLE} & \multicolumn{1}{c|}{NA}\tabularnewline
 \hline
 \multicolumn{1}{|c||}{N}&\multicolumn{1}{c||}{$ $ }&\multicolumn{1}{c}{$\mu_y$}&\multicolumn{1}{c||}{$\sigma_y$}&\multicolumn{1}{c}{$\mu_x$}&\multicolumn{1}{c}{$\sigma_x$}&\multicolumn{1}{c|}{$\delta$}&\multicolumn{1}{c||}{$\sigma_y$}&\multicolumn{1}{c}{$\mu_x$}&\multicolumn{1}{c}{$\sigma_x$}&\multicolumn{1}{c|}{$\delta$}&\multicolumn{1}{c||}{$\sigma_y$}&\multicolumn{1}{c|}{ratio}\tabularnewline
\hline
$  50$&$-0.02$&$-0.02$&$1.15$&$-0.02$&$1.02$&$0.08$&$1.18$&$-0.02$&$0.99$&$0.09$&$ \infty $&$0$\tabularnewline
$ 100$&$ 0.00$&$ 0.00$&$1.17$&$ 0.00$&$1.02$&$0.09$&$1.18$&$ 0.00$&$1.00$&$0.09$&$1.18$&$0$\tabularnewline
$ 250$&$ 0.00$&$ 0.00$&$1.18$&$ 0.00$&$1.01$&$0.09$&$1.18$&$ 0.00$&$1.00$&$0.10$&$1.18$&$0$\tabularnewline
$1000$&$ 0.00$&$ 0.00$&$1.18$&$ 0.00$&$1.00$&$0.10$&$1.18$&$ 0.00$&$1.00$&$0.10$&$1.18$&$0$\tabularnewline
\hline
$  50$&$ 0.56$&$ 0.53$&$0.61$&$ 0.54$&$0.48$&$0.64$&$0.55$&$ 0.53$&$0.55$&$0.58$&$0.56$&$0$\tabularnewline
$ 100$&$ 0.50$&$ 0.49$&$0.57$&$ 0.50$&$0.45$&$0.61$&$0.54$&$ 0.49$&$0.51$&$0.56$&$0.54$&$0$\tabularnewline
$ 250$&$ 0.50$&$ 0.48$&$0.56$&$ 0.47$&$0.46$&$0.56$&$0.53$&$ 0.48$&$0.51$&$0.54$&$0.53$&$0$\tabularnewline
$1000$&$ 0.48$&$ 0.49$&$0.53$&$ 0.48$&$0.50$&$0.54$&$0.51$&$ 0.48$&$0.52$&$0.51$&$0.52$&$0$\tabularnewline
\hline
$  50$&$ 1.27$&$ 1.22$&$1.13$&$ 1.18$&$1.03$&$0.52$&$1.27$&$ 1.16$&$1.07$&$0.62$&$NA $&$0$\tabularnewline
$ 100$&$ 1.28$&$ 1.19$&$1.21$&$ 1.15$&$1.07$&$0.60$&$1.26$&$ 1.12$&$1.09$&$0.64$&$1.28$&$0$\tabularnewline
$ 250$&$ 1.26$&$ 1.19$&$1.20$&$ 1.12$&$1.09$&$0.63$&$1.22$&$ 1.09$&$1.09$&$0.65$&$1.23$&$0$\tabularnewline
$1000$&$ 1.23$&$ 1.17$&$1.26$&$ 1.11$&$1.14$&$0.66$&$1.26$&$ 1.08$&$1.11$&$0.63$&$1.23$&$0$\tabularnewline
\hline
\end{tabular}
}\\

\caption{\label{tab:simulations} Based on $n=1,000$ replications. In each sub-table: (first rows) average, (middle rows) proportion of estimates below true value, (bottom rows) empirical standard deviation times $\sqrt{N}$.}
\end{table}

\begin{algorithm}[t]
  \caption{\label{alg:sim_LambertW} Random sample generation \blinded{: function \texttt{rLambertW($\cdot$)} in \texttt{LambertW} package.}}
 \begin{algorithmic}[1]
	\renewcommand{\algorithmicrequire}{\textbf{Input:}}
	\renewcommand{\algorithmicensure}{\textbf{Output:}}
\REQUIRE number of observations $n$; parameter vector $\theta$; specification of the input distribution $F_X(x)$
\ENSURE random sample $(y_1, \ldots, y_n)$ of a Lambert W $\times$ $F_X$ RV.

    \medskip
    \STATE Simulate $n$ samples $\mathbf{x} = (x_1, \ldots, x_n) \sim F_X(x)$.
    \STATE Compute $\mu_x = \mu_x(\boldsymbol \beta)$ and $\sigma_x = \sigma_x(\boldsymbol \beta)$ (for scale family set $\mu_x = 0$, for non-central, non-scaled also set $\sigma_x = 1$)
    \STATE Compute normalized $\mathbf{u} = (\mathbf{x}-\mu_x) / \sigma_x$.
    \STATE $\mathbf{z} = \mathbf{u} \exp \left( \frac{\delta}{2} \mathbf{u}^2 \right)$
    \RETURN  $\mathbf{y} = \mathbf{z} \sigma_x + \mu_x$
  \end{algorithmic}
\end{algorithm}

\end{document}